\documentclass[a4paper,11pt]{amsart}
\usepackage[hmarginratio={1:1},vmarginratio={1:1},lmargin=60pt,tmargin=60pt]{geometry}

\usepackage[T1]{fontenc}
\usepackage[utf8]{inputenc}
\usepackage[numbers]{natbib}

\usepackage{amsmath,amssymb,amsthm,amscd,mathtools, dsfont}
\usepackage{aliascnt,bbm,bm,mathrsfs,stackrel,textcomp}
\usepackage{array,enumitem,etoolbox,float,graphicx,xcolor,colortbl}
\usepackage{tikz}
\usetikzlibrary{arrows.meta,calc,cd,decorations.markings,decorations.pathreplacing,decorations.pathmorphing,matrix,positioning,shapes,shapes.callouts}

\usepackage[hypertexnames=false]{hyperref}
\usepackage{bookmark}
\let\emph\relax
\DeclareTextFontCommand{\emph}{\bfseries\em}

\definecolor{mygray}{gray}{0.6}
\definecolor{mygraydark}{gray}{0.4}
\definecolor{mygraylight}{gray}{0.85}
\definecolor{spinach}{RGB}{46,139,87}
\definecolor{tomato}{RGB}{255,99,71}
\definecolor{orchid}{RGB}{143,40,194}
\definecolor{neon}{RGB}{77,77,255}
\definecolor{lightneon}{RGB}{110,110,255}
\definecolor{pumpkin}{RGB}{224,180,80}
\definecolor{citron}{RGB}{190,180,90}
\definecolor{lava}{RGB}{207,16,32}
\definecolor{cream}{RGB}{255,253,208}
\definecolor{verdigris}{RGB}{67,179,174}
\definecolor{mydarkblue}{RGB}{10,10,170}
\definecolor{darkspinach}{RGB}{20,70,20}
\definecolor{darktomato}{RGB}{155,40,30}
\definecolor{darkorchid}{RGB}{50,10,100}
\definecolor{darklava}{RGB}{150,8,16}
\definecolor{zero}{RGB}{0,0,0}
\definecolor{one}{RGB}{255,0,0}
\definecolor{two}{RGB}{0,255,0}
\definecolor{three}{RGB}{0,0,255}

\hypersetup{
pdftoolbar=true,
pdfmenubar=true,
pdffitwindow=false,
pdfstartview={FitH},
pdftitle={Presentations for categories of crystals},
pdfauthor={David He and Daniel Tubbenhauer},
pdfcreator={David He and Daniel Tubbenhauer},
pdfproducer={David He and Daniel Tubbenhauer},
pdfkeywords={Crystals, monoidal categories, generators and relations, Jones--Wenzl projectors},
pdfnewwindow=true,
colorlinks=true,
linkcolor=mydarkblue,
citecolor=teal,
filecolor=magenta,
urlcolor=orchid,
linkbordercolor=lava,
citebordercolor=teal,
urlbordercolor=orchid,
linktocpage=true
}

\usepackage{tikz}
\usetikzlibrary{matrix, fit, calc, backgrounds, arrows, shapes, shapes.geometric, shapes.misc}

\tikzstyle{tikzfig}=[baseline=-0.25em,scale=0.5]
\pgfkeys{/tikz/tikzit fill/.initial=0}
\pgfkeys{/tikz/tikzit draw/.initial=0}
\pgfkeys{/tikz/tikzit shape/.initial=0}
\pgfkeys{/tikz/tikzit category/.initial=0}

\pgfdeclarelayer{edgelayer}
\pgfdeclarelayer{nodelayer}
\pgfsetlayers{background,edgelayer,nodelayer,main}
\tikzstyle{none}=[inner sep=0mm]
\tikzstyle{every loop}=[]

\newcommand{\tikzfig}[1]{%
{\tikzstyle{every picture}=[tikzfig]
\IfFileExists{#1.tikz}
  {\input{#1.tikz}}
  {%
    \IfFileExists{./figures/#1.tikz}
      {\input{./figures/#1.tikz}}
      {\tikz[baseline=-0.5em]{\node[draw=red,font=\color{red},fill=red!10!white] {\textit{#1}};}}%
  }}%
}

\tikzstyle{blob}=[fill=white, draw=black, shape=circle, inner sep=1.5pt]
\tikzstyle{fullblob}=[fill=black, draw=black, shape=circle, tikzit fill=black, tikzit draw=black, inner sep=1.5pt]
\tikzstyle{jonesrectangle}=[fill=white, draw=black, shape=rectangle, minimum width=1cm, minimum height=0.6cm]

\tikzstyle{thickstrand}=[-, very thick]
\tikzstyle{arrow}=[draw={mygray}, ->, thick]	
\tikzstyle{bluearrow}=[draw={blue!75!black}, thick, ->]
\tikzstyle{redarrow}=[draw={purple!75!black}, thick, ->]
\tikzstyle{blueline}=[draw={rgb,255: red,0; green,0; blue,255}, thick, -]
\tikzstyle{redline}=[draw={rgb,255: red,255; green,0; blue,0}, thick, -]
\tikzstyle{braid-over}=[-, draw=white, double=black, double distance=0.8pt, tikzit draw={rgb,255: red,128; green,0; blue,128}]
\tikzstyle{background}=[-, dashed, draw={rgb,255: red,191; green,191; blue,191}, thin]

\tikzset{
anchorbase/.style={baseline={([yshift=-0.5ex]current bounding box.center)}},
crossline/.style={preaction={draw=white,line width=10pt,-},preaction={draw=black,line width=1.8pt,-}},
usual/.style={line width=2pt,color=black},
dot/.style={
decoration={markings,post length=0.25mm,pre length=0.25mm,
mark=at position #1 with {\node[circle,radius=0.2cm,inner sep=-1.5pt,color=black,fill=black]{};}},
postaction={decorate}
},
dot/.default=1,
sp4rel/.style={baseline={([yshift=-0.5ex]current bounding box.center)},line cap=round,line join=round},
sp4one/.style={line width=2pt,color=black},
sp4two/.style={line width=2pt,color=cyan}
}
\tikzstyle directed=[postaction={decorate,decoration={markings,mark=at position #1 with {\arrow[line width=0.2mm, black]{>}}}}]
\tikzstyle rdirected=[postaction={decorate,decoration={markings,mark=at position #1 with {\arrow[line width=0.2mm, black]{<}}}}]
\tikzset{slup/.style={directed=.6},sldown/.style={rdirected=.6}}

\def\NewTheorem#1{%
\newaliascnt{#1}{equation}%
\newtheorem{#1}[#1]{#1}%
\aliascntresetthe{#1}%
\expandafter\def\csname #1autorefname\endcsname{#1}%
}
\def\equationautorefname~#1\null{(#1)\null}

\numberwithin{equation}{subsection}

\NewTheorem{Proposition}
\NewTheorem{Theorem}
\NewTheorem{Lemma}
\NewTheorem{Corollary}
\theoremstyle{definition}
\NewTheorem{Definition}
\AtEndEnvironment{Definition}{\null\hfill$\Diamond$}%
\NewTheorem{Notation}
\AtEndEnvironment{Notation}{\null\hfill$\Diamond$}%
\NewTheorem{Example}
\AtEndEnvironment{Example}{\null\hfill$\Diamond$}%
\theoremstyle{remark}
\NewTheorem{Remark}
\AtEndEnvironment{Remark}{\null\hfill$\Diamond$}%

\setlist[enumerate]{itemsep=0.15cm,label=\emph{\upshape(\alph*)}}
\setlist[enumerate,2]{itemsep=0.15cm,label=\emph{\upshape(\roman*)}}
\setlist[enumerate,3]{itemsep=0.15cm,label=\emph{\upshape(\Alph*)}}
\newcolumntype{C}{>{$}c<{$}}
\newcolumntype{P}[1]{>{\centering\arraybackslash}p{#1}}

\newcommand{\Z}{\mathbb{Z}}

\newcommand{\D}{\mathcal{D}}

\newcommand{\CC}{\mathcal{C}}

\newcommand{\kk}{\mathbbm{k}}

\DeclareMathOperator{\CrysMo}{CrysMo}

\DeclareMathOperator{\Mat}{Mat}

\DeclareMathOperator{\Top}{Top}

\DeclareMathOperator{\wt}{wt}

\DeclareMathOperator{\Fund}{\mathrm{Fund}}

\DeclareMathOperator{\Hom}{Hom}

\DeclareMathOperator{\g}{\mathfrak{g}}

\DeclareMathOperator{\gcrys}{\mathfrak{g}-\mathrm{Crys}}

\newcommand{\Id}{\mathds 1}

\newcommand{\webcap}{\begin{tikzpicture}[anchorbase,scale=0.4]
\draw[usual] (0,0)to[out=90,in=180] (0.5,0.5)to[out=0,in=90] (1,0);
\end{tikzpicture}}

\newcommand{\tup}{\begin{tikzpicture}[anchorbase,scale=0.4]
\draw[usual] (0,0)to[out=90,in=180] (0.5,0.5);
\draw[usual] (1,0)to[out=90,in=0] (0.5,0.5);
\draw[usual] (0.5,0)to (0.5,0.5);
\end{tikzpicture}}

\newcommand{\spfourcapone}{%
\begin{tikzpicture}[sp4rel,scale=0.82]
\draw[sp4one] (0,0) node[below] {$1$}
to[out=90,in=180] (0.55,0.50)
to[out=0,in=90] (1.10,0) node[below] {$1$};
\end{tikzpicture}}

\newcommand{\spfourcupone}{%
\begin{tikzpicture}[sp4rel,scale=0.82]
\draw[sp4one] (0,0) node[above] {$1$}
to[out=270,in=180] (0.55,-0.50)
to[out=0,in=270] (1.10,0) node[above] {$1$};
\end{tikzpicture}}

\newcommand{\spfourcaptwo}{%
\begin{tikzpicture}[sp4rel,scale=0.82]
\draw[sp4two] (0,0) node[below] {$2$}
to[out=90,in=180] (0.55,0.50)
to[out=0,in=90] (1.10,0) node[below] {$2$};
\end{tikzpicture}}

\newcommand{\spfourcuptwo}{%
\begin{tikzpicture}[sp4rel,scale=0.82]
\draw[sp4two] (0,0) node[above] {$2$}
to[out=270,in=180] (0.55,-0.50)
to[out=0,in=270] (1.10,0) node[above] {$2$};
\end{tikzpicture}}

\newcommand{\spfourmergeaa}{%
\begin{tikzpicture}[sp4rel,scale=0.82]
\draw[sp4one] (-0.55,0) node[below] {$1$} -- (0,0.55);
\draw[sp4one] ( 0.55,0) node[below] {$1$} -- (0,0.55);
\draw[sp4two] (0,0.55) -- (0,1.20) node[above] {$2$};
\end{tikzpicture}}

\newcommand{\spfoursplitaa}{%
\begin{tikzpicture}[sp4rel,scale=0.82]
\draw[sp4two] (0,0) node[below] {$2$} -- (0,0.65);
\draw[sp4one] (0,0.65) -- (-0.55,1.20) node[above] {$1$};
\draw[sp4one] (0,0.65) -- ( 0.55,1.20) node[above] {$1$};
\end{tikzpicture}}

\newcommand{\spfourmergeonetwo}{%
\begin{tikzpicture}[sp4rel,scale=0.82]
\draw[sp4one] (-0.55,0) node[below] {$1$} -- (0,0.55);
\draw[sp4two] ( 0.55,0) node[below] {$2$} -- (0,0.55);
\draw[sp4one] (0,0.55) -- (0,1.20) node[above] {$1$};
\end{tikzpicture}}

\newcommand{\spfoursplitonetwo}{%
\begin{tikzpicture}[sp4rel,scale=0.82]
\draw[sp4one] (0,0) node[below] {$1$} -- (0,0.65);
\draw[sp4one] (0,0.65) -- (-0.55,1.20) node[above] {$1$};
\draw[sp4two] (0,0.65) -- ( 0.55,1.20) node[above] {$2$};
\end{tikzpicture}}

\newcommand{\spfourcrossonetwo}{%
\begin{tikzpicture}[sp4rel,scale=0.82]
\draw[sp4one] (-0.55,0) node[below] {$1$} -- (0,0.55);
\draw[sp4one] (0,0.55) -- (0.55,1.10) node[above] {$1$};
\draw[sp4two] (0.55,0) node[below] {$2$} -- (0,0.55);
\draw[sp4two] (0,0.55) -- (-0.55,1.10) node[above] {$2$};
\end{tikzpicture}}

\newcommand{\spfourcrosstwoone}{%
\begin{tikzpicture}[sp4rel,scale=0.82]
\draw[sp4two] (-0.55,0) node[below] {$2$} -- (0,0.55);
\draw[sp4two] (0,0.55) -- (0.55,1.10) node[above] {$2$};
\draw[sp4one] (0.55,0) node[below] {$1$} -- (0,0.55);
\draw[sp4one] (0,0.55) -- (-0.55,1.10) node[above] {$1$};
\end{tikzpicture}}

\newcommand{\spfourswitchup}{%
\begin{tikzpicture}[sp4rel,scale=0.82]
\draw[sp4one] (-0.55,0) node[below] {$1$} -- (-0.55,0.55);
\draw[sp4one] ( 0.55,0) node[below] {$1$} -- ( 0.55,0.55);
\draw[densely dashed,line width=1.1pt] (-0.55,0.55) -- (0.55,0.55);
\draw[sp4two] (-0.55,0.55) -- (-0.55,1.15) node[above] {$2$};
\draw[sp4two] ( 0.55,0.55) -- ( 0.55,1.15) node[above] {$2$};
\end{tikzpicture}}

\newcommand{\spfourswitchdown}{%
\begin{tikzpicture}[sp4rel,scale=0.82]
\draw[sp4two] (-0.55,0) node[below] {$2$} -- (-0.55,0.60);
\draw[sp4two] ( 0.55,0) node[below] {$2$} -- ( 0.55,0.60);
\draw[densely dashed,line width=1.1pt] (-0.55,0.60) -- (0.55,0.60);
\draw[sp4one] (-0.55,0.60) -- (-0.55,1.15) node[above] {$1$};
\draw[sp4one] ( 0.55,0.60) -- ( 0.55,1.15) node[above] {$1$};
\end{tikzpicture}}

\newcommand{\upvert}{\begin{tikzpicture}[anchorbase,scale=0.6]
\draw[usual,slup] (0,-0.7) -- (0,0.7);
\fill (0,0.7) circle (0.07);
\fill (0,-0.7) circle (0.07);
\end{tikzpicture}}

\newcommand{\upverta}{\begin{tikzpicture}[anchorbase,scale=0.8]
\draw[usual] (0,0) -- (0,1);
\end{tikzpicture}}

\newcommand{\downvert}{\begin{tikzpicture}[anchorbase,scale=0.6]
\draw[usual,sldown] (0,-0.7) -- (0,0.7);
\fill (0,0.7) circle (0.07);
\fill (0,-0.7) circle (0.07);
\end{tikzpicture}}

\newcommand{\trisplit}{\begin{tikzpicture}[anchorbase,scale=0.8]
\draw[usual,sldown] (0,0.7) -- (0,0);
\draw[usual,slup] (0,0) -- (0.6,-0.5);
\draw[usual,slup] (0,0) -- (-0.6,-0.5);
\fill (0,0.7) circle (0.07);
\fill (0.6,-0.5) circle (0.07);
\fill (-0.6,-0.5) circle (0.07);
\end{tikzpicture}}

\newcommand{\trisplita}{\begin{tikzpicture}[anchorbase,scale=0.8]
\draw[usual] (0,0.5) -- (0,0);
\draw[usual] (0,0) -- (0.5,-0.5);
\draw[usual] (0,0) -- (-0.5,-0.5);
\end{tikzpicture}}

\newcommand{\trimergedown}{\begin{tikzpicture}[anchorbase,scale=0.6]
\draw[usual,slup] (0,0.7) -- (0,0);
\draw[usual,slup] (0.6,-0.5) -- (0,0);
\draw[usual,slup] (-0.6,-0.5) -- (0,0);
\fill (0,0.7) circle (0.07);
\fill (0.6,-0.5) circle (0.07);
\fill (-0.6,-0.5) circle (0.07);
\end{tikzpicture}}

\newcommand{\cupL}{\begin{tikzpicture}[anchorbase,scale=0.6]
\draw[usual,slup] (-0.6,0) to[out=270,in=270] (0.6,0);
\fill (-0.6,0) circle (0.07);
\fill (0.6,0) circle (0.07);
\end{tikzpicture}}

\newcommand{\vtrimergedown}{\begin{tikzpicture}[anchorbase,scale=0.6]
\draw[usual,sldown] (0,-0.7) -- (0,0);
\draw[usual,sldown] (0.6,0.5) -- (0,0);
\draw[usual,sldown] (-0.6,0.5) -- (0,0);
\fill (0,-0.7) circle (0.07);
\fill (0.6,0.5) circle (0.07);
\fill (-0.6,0.5) circle (0.07);
\end{tikzpicture}}

\newcommand{\capL}{\begin{tikzpicture}[anchorbase,scale=0.6]
\draw[usual,slup] (-0.6,0) to[out=90,in=90] (0.6,0);
\fill (-0.6,0) circle (0.07);
\fill (0.6,0) circle (0.07);
\end{tikzpicture}}

\newcommand{\capLa}{\begin{tikzpicture}[anchorbase,scale=0.8]
\draw[usual] (-0.5,0) to[out=90,in=90] (0.5,0);
\end{tikzpicture}}

\newcommand{\soThreeJW}[1]{%
\begin{tikzpicture}[anchorbase,x=.72cm,y=.72cm,line cap=round,line join=round]
\path[use as bounding box] (-0.12,-0.10) rectangle (1.22,1.10);
#1%
\end{tikzpicture}%
}

\newcommand{\capR}{\begin{tikzpicture}[anchorbase,scale=0.6]
\draw[usual,sldown] (-0.6,0) to[out=90,in=90] (0.6,0);
\fill (-0.6,0) circle (0.07);
\fill (0.6,0) circle (0.07);
\end{tikzpicture}}

\newcommand{\tcupL}{\begin{tikzpicture}[anchorbase,scale=0.6]
\draw[usual,slup] (0,0) to[out=90,in=180] (0.5,0.5);
\draw[usual,slup] (1,0) to[out=90,in=0] (0.5,0.5);
\draw[usual,slup] (0.5,0) -- (0.5,0.5);
\fill (0,0) circle (0.07);
\fill (1,0) circle (0.07);
\fill (0.5,0) circle (0.07);
\end{tikzpicture}}

\newcommand{\tupL}{\tcupL}

\newcommand{\tcupR}{\begin{tikzpicture}[anchorbase,scale=0.6]
\draw[usual,sldown] (0,0) to[out=90,in=180] (0.5,0.5);
\draw[usual,sldown] (1,0) to[out=90,in=0] (0.5,0.5);
\draw[usual,sldown] (0.5,0) -- (0.5,0.5);
\fill (0,0) circle (0.07);
\fill (1,0) circle (0.07);
\fill (0.5,0) circle (0.07);
\end{tikzpicture}}

\newcommand{\tdownL}{\begin{tikzpicture}[anchorbase,scale=0.6]
\draw[usual,slup] (0.5,-0.5) to[out=180,in=270] (0,0);
\draw[usual,slup] (0.5,-0.5) to[out=0,in=270] (1,0);
\draw[usual,slup] (0.5,-0.5) -- (0.5,0);
\fill (0,0) circle (0.07);
\fill (1,0) circle (0.07);
\fill (0.5,0) circle (0.07);
\end{tikzpicture}}

\DeclareMathOperator{\id}{id}

\DeclareMathOperator{\End}{End}

\newcommand{\n}{\underline{n}}

\newcommand{\kkk}{\underline{k}}

\newcommand{\m}{\underline{m}}

\newcommand{\pcrossing}{\begin{tikzpicture}[anchorbase,scale=0.6]
\draw[usual] (-0.5,-0.5) -- (0,0);
\draw[usual,slup] (0,0) -- (0.5,0.5);
\draw[usual] (-0.5,0.5) -- (0,0);
\draw[usual,slup] (0,0) -- (0.5,-0.5);
\end{tikzpicture}}

\newcommand{\mcrossing}{\begin{tikzpicture}[anchorbase,scale=0.6]
\draw[usual] (-0.5,-0.5) -- (0,0);
\draw[usual,sldown] (0,0) -- (0.5,0.5);
\draw[usual] (-0.5,0.5) -- (0,0);
\draw[usual,sldown] (0,0) -- (0.5,-0.5);
\end{tikzpicture}}

\newcommand{\soThreeBranchEmpty}{%
\begin{tikzpicture}[baseline={(current bounding box.center)},x=.36cm,y=.36cm,line cap=round,line join=round]
\path[use as bounding box] (-0.30,-0.10) rectangle (0.30,1.55);
\node[inner sep=0pt] at (0,0.72) {$\emptyset$};
\end{tikzpicture}}

\newcommand{\soThreeBranchI}{%
\begin{tikzpicture}[baseline={(current bounding box.center)},x=.36cm,y=.36cm,line cap=round,line join=round]
\path[use as bounding box] (-0.25,-0.10) rectangle (0.25,1.55);
\draw[usual] (0,0) -- (0,1.45);
\end{tikzpicture}}

\newcommand{\soThreeBranchII}{%
\begin{tikzpicture}[baseline={(current bounding box.center)},x=.36cm,y=.36cm,line cap=round,line join=round]
\path[use as bounding box] (-0.25,-0.10) rectangle (0.95,1.55);
\draw[usual] (0,0) -- (0,1.45);
\draw[usual] (0.70,0) -- (0.70,1.45);
\end{tikzpicture}}

\newcommand{\soThreeBranchMerge}{%
\begin{tikzpicture}[baseline={(current bounding box.center)},x=.36cm,y=.36cm,line cap=round,line join=round]
\path[use as bounding box] (-0.25,-0.10) rectangle (0.95,1.55);
\draw[usual] (0,0) -- (0.35,0.55);
\draw[usual] (0.70,0) -- (0.35,0.55);
\draw[usual] (0.35,0.55) -- (0.35,1.45);
\end{tikzpicture}}

\newcommand{\soThreeBranchCap}{%
\begin{tikzpicture}[baseline={(current bounding box.center)},x=.36cm,y=.36cm,line cap=round,line join=round]
\path[use as bounding box] (-0.25,-0.10) rectangle (0.95,1.00);
\draw[usual] (0,0) to[out=90,in=180] (0.35,0.55) to[out=0,in=90] (0.70,0);
\end{tikzpicture}}

\newcommand{\soThreeBranchIII}{%
\begin{tikzpicture}[baseline={(current bounding box.center)},x=.36cm,y=.36cm,line cap=round,line join=round]
\path[use as bounding box] (-0.25,-0.10) rectangle (1.65,1.55);
\draw[usual] (0,0) -- (0,1.45);
\draw[usual] (0.70,0) -- (0.70,1.45);
\draw[usual] (1.40,0) -- (1.40,1.45);
\end{tikzpicture}}

\newcommand{\soThreeBranchIMerge}{%
\begin{tikzpicture}[baseline={(current bounding box.center)},x=.36cm,y=.36cm,line cap=round,line join=round]
\path[use as bounding box] (-0.25,-0.10) rectangle (1.65,1.55);
\draw[usual] (0,0) -- (0,1.45);
\draw[usual] (0.70,0) -- (1.05,0.55);
\draw[usual] (1.40,0) -- (1.05,0.55);
\draw[usual] (1.05,0.55) -- (1.05,1.45);
\end{tikzpicture}}

\newcommand{\soThreeBranchICap}{%
\begin{tikzpicture}[baseline={(current bounding box.center)},x=.36cm,y=.36cm,line cap=round,line join=round]
\path[use as bounding box] (-0.25,-0.10) rectangle (1.65,1.10);
\draw[usual] (0,0) -- (0,1.00);
\draw[usual] (0.70,0) to[out=90,in=180] (1.05,0.55) to[out=0,in=90] (1.40,0);
\end{tikzpicture}}

\newcommand{\soThreeBranchMergeI}{%
\begin{tikzpicture}[baseline={(current bounding box.center)},x=.36cm,y=.36cm,line cap=round,line join=round]
\path[use as bounding box] (-0.25,-0.10) rectangle (1.65,1.55);
\draw[usual] (0,0) -- (0.35,0.55);
\draw[usual] (0.70,0) -- (0.35,0.55);
\draw[usual] (0.35,0.55) -- (0.35,1.45);
\draw[usual] (1.40,0) -- (1.40,1.45);
\end{tikzpicture}}

\newcommand{\soThreeBranchTree}{%
\begin{tikzpicture}[baseline={(current bounding box.center)},x=.36cm,y=.36cm,line cap=round,line join=round]
\path[use as bounding box] (-0.25,-0.10) rectangle (1.65,1.55);
\draw[usual] (0.70,0.55) -- (0.70,1.45);
\draw[usual] (0.70,0.55) -- (0,0);
\draw[usual] (0.70,0.55) -- (0.70,0);
\draw[usual] (0.70,0.55) -- (1.40,0);
\end{tikzpicture}}

\newcommand{\soThreeBranchTup}{%
\begin{tikzpicture}[baseline={(current bounding box.center)},x=.36cm,y=.36cm,line cap=round,line join=round]
\path[use as bounding box] (-0.25,-0.10) rectangle (1.65,1.10);
\draw[usual] (0,0) to[out=90,in=180] (0.70,0.80);
\draw[usual] (1.40,0) to[out=90,in=0] (0.70,0.80);
\draw[usual] (0.70,0) -- (0.70,0.80);
\end{tikzpicture}}

\newcommand{\soThreeBranchCapI}{%
\begin{tikzpicture}[baseline={(current bounding box.center)},x=.36cm,y=.36cm,line cap=round,line join=round]
\path[use as bounding box] (-0.25,-0.10) rectangle (1.65,1.10);
\draw[usual] (0,0) to[out=90,in=180] (0.35,0.55) to[out=0,in=90] (0.70,0);
\draw[usual] (1.40,0) -- (1.40,1.00);
\end{tikzpicture}}

\newcommand{\motzGen}[1]{%
\begin{tikzpicture}[anchorbase]
\path[use as bounding box] (-0.70,-0.58) rectangle (0.70,0.58);
#1%
\end{tikzpicture}%
}

\newcommand{\motzComma}{%
\begin{tikzpicture}[anchorbase]
\path[use as bounding box] (-0.10,-0.58) rectangle (0.22,0.58);
\node[inner sep=0pt] at (0,-0.32) {$,$};
\end{tikzpicture}\mspace{12mu}%
}

\newcommand{\motzTopDot}{%
\motzGen{
\draw[usual,dot] (0,0.38) -- (0,0);
}}

\newcommand{\motzBottomDot}{%
\motzGen{
\draw[usual,dot] (0,-0.38) -- (0,0);
}}

\newcommand{\motzIdentity}{%
\motzGen{
\draw[usual] (0,-0.48) -- (0,0.48);
}}

\newcommand{\motzCup}{%
\motzGen{
\draw[usual] (-0.50,0.25)
to[out=270,in=180] (0,-0.25)
to[out=0,in=270] (0.50,0.25);
}}

\newcommand{\motzCap}{%
\motzGen{
\draw[usual] (-0.50,-0.25)
to[out=90,in=180] (0,0.25)
to[out=0,in=90] (0.50,-0.25);
}}

\newcommand{\motzSnakeRight}{%
\begin{tikzpicture}[anchorbase]
\path[use as bounding box] (-0.15,-0.15) rectangle (1.85,1.85);
\begin{scope}[xscale=0.85, yscale=1.15]
\draw[usual] (0,0) -- (0,0.5)
arc (180:0:0.5)
arc (180:360:0.5)
-- (2,1);
\end{scope}
\end{tikzpicture}%
}

\newcommand{\motzSnakeLeft}{%
\begin{tikzpicture}[anchorbase]
\path[use as bounding box] (-0.15,-0.15) rectangle (1.85,1.85);
\begin{scope}[xshift=1.70cm,xscale=-1]
\begin{scope}[xscale=0.85, yscale=1.15]
\draw[usual] (0,0) -- (0,0.5)
arc (180:0:0.5)
arc (180:360:0.5)
-- (2,1);
\end{scope}
\end{scope}
\end{tikzpicture}%
}

\newcommand{\motzCapDotLeft}{%
\begin{tikzpicture}[anchorbase]
\path[use as bounding box] (-0.15,0.10) rectangle (1.15,1.20);
\draw[usual] (0,0.60) to[out=90,in=180] (0.50,1.10) to[out=0,in=90] (1.00,0.60);
\draw[usual,dot] (0,0.60) -- (0,0.20);
\draw[usual] (1.00,0.60) -- (1.00,0.20);
\end{tikzpicture}%
}

\newcommand{\motzCapDotRight}{%
\begin{tikzpicture}[anchorbase]
\path[use as bounding box] (-0.15,0.10) rectangle (1.15,1.20);
\draw[usual] (0,0.60) to[out=90,in=180] (0.50,1.10) to[out=0,in=90] (1.00,0.60);
\draw[usual] (0,0.60) -- (0,0.20);
\draw[usual,dot] (1.00,0.60) -- (1.00,0.20);
\end{tikzpicture}%
}

\newcommand{\motzCircle}{%
\begin{tikzpicture}[anchorbase]
\path[use as bounding box] (-0.65,0.05) rectangle (0.65,1.35);
\draw[usual] (-0.50,0.70) to[out=270,in=180] (0,0.20) to[out=0,in=270] (0.50,0.70);
\draw[usual] (-0.50,0.70) to[out=90,in=180] (0,1.20) to[out=0,in=90] (0.50,0.70);
\end{tikzpicture}%
}

\newcommand{\motzDotDot}{%
\begin{tikzpicture}[anchorbase]
\path[use as bounding box] (-0.30,0.05) rectangle (0.30,1.35);
\draw[usual,dot] (0,0.70) -- (0,1.20);
\draw[usual,dot] (0,0.70) -- (0,0.20);
\end{tikzpicture}%
}

\newlength{\cellsize}
\newcommand\tableau[1]{
\vcenter{
\let\\=\cr
\baselineskip=-16000pt
\lineskiplimit=16000pt
\lineskip=0pt
\halign{&\tableaucell{##}\cr#1\crcr}}}

\newcommand{\tableaucell}[1]{{%
\def \arg{#1}\def \void{}%

\ifx \void \arg
\vbox to \cellsize{\vfil \hrule width \cellsize height 0pt}%
\else
\unitlength=\cellsize
\begin{picture}(1,1)
\put(0,0){\makebox(1,1){$#1$}}
\put(0,0){\line(1,0){1}}
\put(0,1){\line(1,0){1}}
\put(0,0){\line(0,1){1}}
\put(1,0){\line(0,1){1}}
\end{picture}%
\fi}}

\def\makeautorefname#1#2{\csdef{#1autorefname}{#2}}

\makeautorefname{section}{Section}%
\makeautorefname{subsection}{Section}%
\makeautorefname{subsubsection}{Section}%

\begin{document}

\title[Presentations for categories of crystals]{Presentations for categories of crystals}
\author[M. Alqady, D. He, and D. Tubbenhauer]{Moaaz Alqady, David He, and Daniel Tubbenhauer}
\address{M.A.: University of Oregon, Department of Mathematics, Fenton Hall, Eugene, OR 97403, USA, \href{https://moaazalqady.github.io}{moaazalqady.github.io}, \href{https://orcid.org/0009-0000-0538-2106}{ORCID 0009-0000-0538-2106}}
\email{malqady@uoregon.edu}

\address{D. H.: The University of Sydney, School of Mathematics and Statistics, Australia}
\email{d.he@sydney.edu.au}

\address{D.T.: The University of Sydney, School of Mathematics and Statistics F07, Office Carslaw 827, NSW 2006, Australia, \href{http://www.dtubbenhauer.com}{www.dtubbenhauer.com}, \href{https://orcid.org/0000-0001-7265-5047}{ORCID 0000-0001-7265-5047}}
\email{daniel.tubbenhauer@sydney.edu.au}

\begin{abstract}
We give generators and relations for the monoidal categories of crystals
generated by the fundamental crystals of a simple complex Lie algebra. We also
spell out several small-rank examples.
\end{abstract}

\subjclass[2020]{Primary: 17B10, 17B37; Secondary: 05E10, 18M05.}
\keywords{Crystals, monoidal categories, generators and relations, Jones--Wenzl projectors, diagrammatics.}

\addtocontents{toc}{\protect\setcounter{tocdepth}{1}}

\maketitle

\tableofcontents

\section{Introduction}\label{sec:intro}

Diagram categories are among the oldest tools in representation theory, even
if this terminology is relatively modern. Schur--Weyl duality relates tensor
powers of the natural representation of \(\mathrm{GL}_n\) to symmetric groups, 
the prototype of string diagrams. 
Brauer's centralizer algebras play the analogous role for orthogonal and symplectic
groups. Rumer--Teller--Weyl diagrams (also called Temperley--Lieb diagrams) 
and many other works give diagrammatic presentations 
of various tensor categories generated by small representations. In modern language, the common
theme is the following question:
\[
\setlength{\fboxsep}{2pt}
\colorbox{orchid!6}{
\begin{minipage}{0.85\textwidth}
\centering
\emph{Can one present a tensor category by pictures and local
relations?}
\end{minipage}}
\]

\subsection{Can we?}

This question has many answers in the semisimple representation theory of
complex reductive groups and their quantum deformations. For example, the
Temperley--Lieb category presents the part of the representation category of
\(\mathfrak{sl}_2\) tensor generated by the vector representation, and web
categories do the same for many other small representations. These categories
are useful precisely because complicated representation-theoretic maps become
local diagrammatic moves.

Crystals are the \(q=0\) shadow of this story.  They grew out of the
work of Kashiwara on crystal bases and Lusztig on canonical bases
\cite{kashiwara1990crystalizing,Lusztig-canonical-bases}, giving
combinatorial skeletons of representations of quantum groups.
Crucially,
\[
\setlength{\fboxsep}{2pt}
\colorbox{orchid!6}{
\begin{minipage}{0.85\textwidth}
\centering
\emph{They remember the combinatorics, but forget most of the linear
algebra.}
\end{minipage}}
\]
This makes them much easier, but also stranger. For example, the
category of crystals is not braided in the usual sense. Henriques and
Kamnitzer showed that it is instead a coboundary category: there is a natural
commutor for crystals, but the braid group is replaced by the cactus group
\cite{Kamnitzcrystal}. It is also not rigid, but only weakly rigid.

Thus there are two parallel stories. 
On the one hand there are the usual web categories, living at generic \(q\), with crossings, cups, caps and local relations. 
On the other hand there is the crystal category, living at \(q=0\), with a tensor product, a commutor, and a cactus group action. 
The point of this paper is to connect these two stories by giving diagrammatic
presentations for categories of crystals. 
It turns out, because the difficult linear algebra and fancy scalars are all gone, there is a general approach to the presentation of crystals, which is what we discuss in this paper. 
Our main result gives a generators-and-relations presentation for categories of crystals in a hands-on and explicit way (i.e. one can use computer code to produce the presentation).

\begin{Remark}
The rank-one case was worked out in a joint work of the first author with Stroiński
\cite{Alqady2025coboundary}. The present paper asks for the corresponding general statement.
\end{Remark}

\subsection{A few details}

More precisely, let \(\mathfrak{g}\) be a simple complex Lie algebra, and let
\(B_i=B(\omega_i)\) be the fundamental crystals (i.e. crystals of the fundamental representations). We consider the monoidal
category \(\Fund(\gcrys)\) whose objects are tensor words
in the \(B_i\).  Its additive idempotent completion is the usual semisimple
category of \(\mathfrak{g}\)-crystals.  Our aim is to present
\(\Fund(\gcrys)\) by generators and relations.

The generating morphisms are of two types.  First, we use crossings
$B_i\otimes B_j \to B_j\otimes B_i$.
In multiplicity-free cases these are forced, while in general one has to make
a choice, for example using the Henriques--Kamnitzer commutors. Second, we use
certain elementary crystal morphisms which we call atoms. These atoms are the
minimal maps appearing in the stable decomposition of
\(B(\lambda)\otimes B_k\), for \(\lambda\) sufficiently deep in the dominant
chamber. 
Equivalently, they are the elementary steps in the branching graph
for tensor products of fundamental crystals. We depict them, for example, as:
\[
\text{ A crossing: }
\begin{tikzpicture}[thick, scale=0.8, baseline=(current bounding box.center)]
\draw[usual] (0,0) -- (1,1);
\draw[usual] (1,0) -- (0,1);
\end{tikzpicture}
,\quad\text{ some atoms: }\quad
\upverta,\trisplita,\capLa
.\]
Our main result is \autoref{iso of categories}.  It says that the category
\(\Fund(\gcrys)\) is generated by crossings and atoms,
subject to explicit local relations.  These relations come in two families.
The first family consists of \(H=I\) relations: they rewrite an upside-down atom followed
by an atom in a standard form.  The second family consists of
orientation-reversing relations: they exclude local pieces which cannot occur
inside distinguished bottom diagrams.  Together these relations imply that
every diagram can be rewritten, in sandwich form, as
\[
\begin{tikzpicture}[
baseline=(current bounding box.center),
x=1cm,y=1cm,
line width=0.8pt,
line cap=butt,
line join=miter,
lab/.style={inner sep=0pt,font=\large},
dots/.style={inner sep=0pt,font=\normalsize}
]
\begin{scope}[shift={(7.30,-0.15)}]
\draw[usual] (0.00,2.10) -- (1.80,2.10) -- (1.50,1.45) -- (0.30,1.45) -- cycle;
\draw[usual] (0.30,0.78) rectangle (1.50,1.45);
\draw[usual] (0.30,0.78) -- (1.50,0.78) -- (1.80,0.10) -- (0.00,0.10) -- cycle;
\node[lab] at (0.90,1.78) {$\bar{\bm{f_i}}$};
\node[lab] at (0.90,1.115) {$\bm{P_i}$};
\node[lab] at (0.90,0.43) {$\bm{g_i}$};
\end{scope}
\end{tikzpicture}
.\]
This gives a basis of every hom-space.  After inserting Jones--Wenzl
projectors, the same construction gives a matrix-unit basis.

The answer is finite and constructive. The atoms are defined directly from
crystal combinatorics, and the coefficients in the relations are obtained by
finite computations in crystals. In particular, the theorem gives an algorithm
for writing down presentations in concrete examples.  The number of atoms can
be large, especially in exceptional type. However, \autoref{efficient pres prop}
shows that all atoms are generated by two-strand atoms and crossings.  In type
\(A_{n-1}\), this replaces the exponential number \(2^n-n-1\) of non-identity
atoms by the cubic number \(\binom{n+1}{3}\) of two-strand atoms.

We also study Jones--Wenzl projectors in crystal categories. These are the
idempotents cutting out the top summand of a tensor word. In the usual
Temperley--Lieb and web categories, Jones--Wenzl projectors are subtle objects,
and their formulas involve quantum coefficients.  In the crystal setting the
formulas are much more rigid.  For types \(A\), \(B\), \(C\), and
\(\mathrm{G}_2\), and for suitably ordered tensor words, \autoref{JW theorem}
says that the Jones--Wenzl projector is simply a product of adjacent
two-strand Jones--Wenzl projectors. Type \(D\) has the expected additional
spin bookkeeping, but the same philosophy still applies.
We also define a poset structure on the set of distinguished bottom diagrams and explain how to explicitly compute the Jones-Wenzl projectors in terms of the M\"obius function on this poset in \autoref{Mobius JW}.

The final section spells out examples. For \(\mathfrak{sl}_2\) we recover the
\(q=0\) Temperley--Lieb category of \cite{Alqady2025coboundary}. We then give
explicit presentations for \(\mathfrak{sl}_3\), \(\mathfrak{sp}_4\), and
\(\mathrm{G}_2\), and discuss the related \(\mathrm{SO}_3\) and Motzkin
examples.  These examples should be read as crystal versions of familiar web
and diagram categories.  They also illustrate the main difference between the
generic \(q\)-world and the crystal world: much of the linear algebra
disappears, so e.g. the Jones--Wenzl projectors are much easier to obtain, but the price is that braidings are replaced by coboundary
structures and local rewriting rules are usually much more complicated.

\begin{Remark}
In fact, categories of crystals are so easy, that one can set up a computer program that outputs lists of generators and relations, and we did this; see \cite{HT} for the Python and SageMath code. For type $\mathrm{G}_2$, the list of relations is so long that we refer the reader to the computer output.
\end{Remark}

\noindent\textbf{Acknowledgments.}
We thank Mateusz Stroiński for helpful discussions and
generous explanations, and Abel Lacabanne and Pedro Vaz for conversations on a related project which also fed into this one.

DT was supported by the ARC Future Fellowship FT230100489, and notes that crystal bases exist, unlike any convincing exit strategy.

\section{Main results}\label{sec:main-results}

We start with the general theory. Our story is mostly self-contained, but background on crystals, monoidal categories and diagrammatics is required, see e.g. \cite{BuSc-crystal-bases,EtGeNiOs-tensor-categories,Tu-qt} for some references.

\subsection{Basic definitions}\label{subsec:basic-definitions}

Throughout, let $\kk$ be a fixed field. 

\begin{Remark}
A field is chosen for convenience. Everything can be done integrally (i.e. all scalars are in $\Z$), which is why crystals are much easier than the full representation categories.
\end{Remark}

Let \(\mathfrak{g}\) be a complex reductive Lie algebra with simple roots
\(\{\alpha_i\}_{i\in J}\) and simple coroots
\(\{\alpha_i^\vee\}_{i\in J}\). Write \(\Lambda\) for the weight lattice and
\(\Lambda_+\subset \Lambda\) for the set of dominant weights. We write
\(\omega_i\) for the \(i\)th fundamental weight, so that
$\langle \omega_i,\alpha_j^\vee\rangle=\delta_{i,j}$.
Thus every dominant integral weight is a nonnegative integral linear
combination of the \(\omega_i\). 

\begin{Notation}
For \(\lambda\in\Lambda_+\), let \(V(\lambda)\)
denote the simple \(\mathfrak{g}\)-module of highest weight \(\lambda\), and
let \(B(\lambda)\) denote its highest weight crystal (recalled in \autoref{crystal def}). When \(\mathfrak{g}\) is
simple, we usually identify \(J\) with \(\{1,\dots,n\}\), using the standard
numbering of the Dynkin diagram, and write
$B_i:=B(\omega_i)$.
We will often write weights in fundamental coordinates: for example, in type
\(A_2\) we write \(B(a,b)\) for \(B(a\omega_1+b\omega_2)\).
\end{Notation}

We first recall some definitions.

\begin{Definition} \label{crystal def}
A $\mathfrak{g}$-\emph{crystal} is a finite set $B$ together with functions $\wt: B\to \Lambda, e_i: B\to B\sqcup \{0\}, f_i: B\to B\sqcup \{0\}$ for each $i\in J$, such that
\begin{enumerate}
\item If $e_i(b)\neq 0$, then $\wt(e_i(b))=\wt(b)+\alpha_i$; if $f_i(b)\neq 0, \wt(f_i(b))=\wt(b)-\alpha_i$.
\item For $b,b'\in B$, we have $e_i(b)=b'$ if and only if $b=f_i(b)'$. 
\item Define $\varepsilon_i=\max\{n\ge 0: e_i^n(b)\neq 0\}, \varphi_i(b)=\max\{n\ge 0: f_i^n(b)\neq 0\}$. Then $\varphi_i(b)-\varepsilon_i(b)=\langle\wt(b),\alpha_i^\vee\rangle$.
\end{enumerate}
We call $e_i, f_i$ the raising and lowering Kashiwara operators. 
If $B, C$ are $\g$-crystals, a \emph{(strict) morphism of crystals} is a map $\psi: B\sqcup \{0\} \to C\sqcup \{0\}$ which sends $0$ to $0$, preserves weights, and commutes with $e_i$ and $f_i$ for each $i$. 
\end{Definition}

\begin{Definition}
If $B,C$ are $\g$-crystals, their tensor product $B\otimes C$ as a set is $B\times C$, with $\wt(b\otimes c)=\wt(b)+\wt(c)$. For the raising and lowering operators we use:
\begin{enumerate}
\item $e_i(b\otimes c)=\begin{cases} e_ib\otimes c & \text{ if } \varphi_i(b)\ge \varepsilon_i(c), \\
b\otimes e_ic & \text{ if } \varphi_i(b)< \varepsilon_i(c),
\end{cases}$
\item $f_i(b\otimes c)=\begin{cases} f_ib\otimes c & \text{ if } \varphi_i(b)> \varepsilon_i(c), \\
b\otimes f_ic & \text{ if } \varphi_i(b)\le \varepsilon_i(c).
\end{cases}$
\end{enumerate}
These are the conventions of \cite{kashiwara1990crystalizing}.
\end{Definition}

\begin{Remark}
We use the convention of \cite{kashiwara1990crystalizing}, which is the opposite of the one used by SageMath. 
\end{Remark}

We view a crystal as a colored graph: the vertices are the elements of $B$, and there is an edge colored by $i$ from $b$ to $b'$ if $f_i(b)=b'$. We say that $B$ is connected if this graph is connected. Recall that if $V(\lambda), \lambda \in \Lambda_+$ is a simple $\g$-module of highest weight $\lambda$, then there is a unique (up to isomorphism) connected crystal $B(\lambda)$ whose elements are obtained by applying the lowering operators $f_i$ to some \emph{highest weight element} $b_0$.

\begin{Example}
The crystal graphs $B(\omega_{1})$ in types $A_{4}$, $B_{4}$, $C_4$ and $D_{4}$ are as follows:
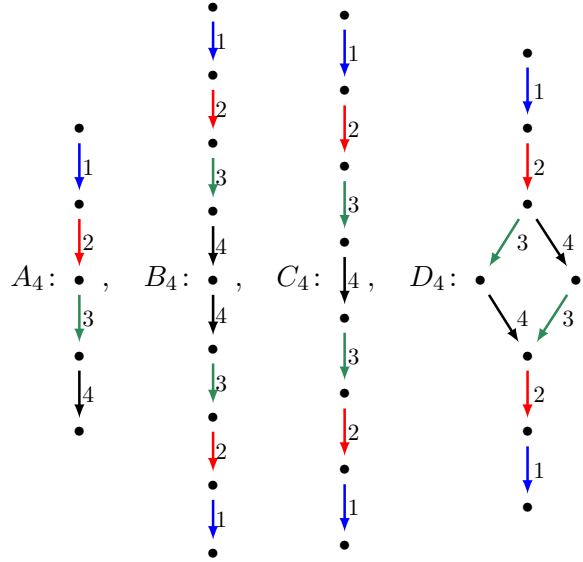
\begin{gather*}
A_{4}\colon
\scalebox{0.8}{$\begin{tikzpicture}[anchorbase,>=latex,line join=bevel,scale=0.5,
every path/.style={very thick}]
\node (node_0) at (27.5bp,150.5bp) [draw,draw=none] {$\bullet$};
\node (node_4) at (27.5bp,79.5bp) [draw,draw=none] {$\bullet$};
\node (node_1) at (27.5bp,221.5bp) [draw,draw=none] {$\bullet$};
\node (node_2) at (27.5bp,292.5bp) [draw,draw=none] {$\bullet$};
\node (node_3) at (27.5bp,8.5bp) [draw,draw=none] {$\bullet$};
\draw [spinach,->] (node_0) ..controls (27.5bp,131.44bp) and (27.5bp,112.5bp)  .. (node_4);
\draw (36.0bp,115.0bp) node {$3$};
\draw [red,->] (node_1) ..controls (27.5bp,202.44bp) and (27.5bp,183.5bp)  .. (node_0);
\draw (36.0bp,186.0bp) node {$2$};
\draw [blue,->] (node_2) ..controls (27.5bp,273.44bp) and (27.5bp,254.5bp)  .. (node_1);
\draw (36.0bp,257.0bp) node {$1$};
\draw [black,->] (node_4) ..controls (27.5bp,60.442bp) and (27.5bp,41.496bp)  .. (node_3);
\draw (36.0bp,44.0bp) node {$4$};
\end{tikzpicture}$}
,\quad
B_{4}\colon
\scalebox{0.8}{$\begin{tikzpicture}[anchorbase,>=latex,line join=bevel,scale=0.45,every path/.style={very thick}]
\node (node_0) at (53.5bp,435.5bp) [draw,draw=none] {$\bullet$};
\node (node_8) at (53.5bp,364.5bp) [draw,draw=none] {$\bullet$};
\node (node_1) at (53.5bp,506.5bp) [draw,draw=none] {$\bullet$};
\node (node_2) at (53.5bp,577.5bp) [draw,draw=none] {$\bullet$};
\node (node_3) at (53.5bp,150.5bp) [draw,draw=none] {$\bullet$};
\node (node_5) at (53.5bp,79.5bp) [draw,draw=none] {$\bullet$};
\node (node_4) at (53.5bp,293.0bp) [draw,draw=none] {$\bullet$};
\node (node_6) at (53.5bp,221.5bp) [draw,draw=none] {$\bullet$};
\node (node_7) at (53.5bp,8.5bp) [draw,draw=none] {$\bullet$};
\draw [spinach,->] (node_0) ..controls (53.5bp,416.44bp) and (53.5bp,397.5bp)  .. (node_8);
\draw (62.0bp,400.0bp) node {$3$};
\draw [red,->] (node_1) ..controls (53.5bp,487.44bp) and (53.5bp,468.5bp)  .. (node_0);
\draw (62.0bp,471.0bp) node {$2$};
\draw [blue,->] (node_2) ..controls (53.5bp,558.44bp) and (53.5bp,539.5bp)  .. (node_1);
\draw (62.0bp,542.0bp) node {$1$};
\draw [red,->] (node_3) ..controls (53.5bp,131.44bp) and (53.5bp,112.5bp)  .. (node_5);
\draw (62.0bp,115.0bp) node {$2$};
\draw [black,->] (node_4) ..controls (53.5bp,273.19bp) and (53.5bp,254.59bp)  .. (node_6);
\draw (62.0bp,257.0bp) node {$4$};
\draw [blue,->] (node_5) ..controls (53.5bp,60.442bp) and (53.5bp,41.496bp)  .. (node_7);
\draw (62.0bp,44.0bp) node {$1$};
\draw [spinach,->] (node_6) ..controls (53.5bp,202.44bp) and (53.5bp,183.5bp)  .. (node_3);
\draw (62.0bp,186.0bp) node {$3$};
\draw [black,->] (node_8) ..controls (53.5bp,345.42bp) and (53.5bp,326.63bp)  .. (node_4);
\draw (62.0bp,329.0bp) node {$4$};
\end{tikzpicture}$}
,\quad
C_{4}\colon
\scalebox{0.8}{$\begin{tikzpicture}[anchorbase,>=latex,line join=bevel,scale=0.5,
every path/.style={very thick}]
\node (node_0) at (27.5bp,150.5bp) [draw,draw=none] {$\bullet$};
\node (node_5) at (27.5bp,79.5bp) [draw,draw=none] {$\bullet$};
\node (node_1) at (27.5bp,221.5bp) [draw,draw=none] {$\bullet$};
\node (node_2) at (27.5bp,505.5bp) [draw,draw=none] {$\bullet$};
\node (node_7) at (27.5bp,434.5bp) [draw,draw=none] {$\bullet$};
\node (node_3) at (27.5bp,8.5bp) [draw,draw=none] {$\bullet$};
\node (node_4) at (27.5bp,363.5bp) [draw,draw=none] {$\bullet$};
\node (node_6) at (27.5bp,292.5bp) [draw,draw=none] {$\bullet$};
\draw [blue,->] (node_2) ..controls (27.5bp,486.44bp) and (27.5bp,467.5bp)  .. (node_7);
\draw (36.0bp,470.0bp) node {$1$};
\draw [red,->] (node_7) ..controls (27.5bp,415.44bp) and (27.5bp,396.5bp)  .. (node_4);
\draw (36.0bp,399.0bp) node {$2$};
\draw [spinach,->] (node_4) ..controls (27.5bp,344.44bp) and (27.5bp,325.5bp)  .. (node_6);
\draw (36.0bp,328.0bp) node {$3$};
\draw [black,->] (node_6) ..controls (27.5bp,273.44bp) and (27.5bp,254.5bp)  .. (node_1);
\draw (36.0bp,257.0bp) node {$4$};
\draw [spinach,->] (node_1) ..controls (27.5bp,202.44bp) and (27.5bp,183.5bp)  .. (node_0);
\draw (36.0bp,186.0bp) node {$3$};
\draw [red,->] (node_0) ..controls (27.5bp,131.44bp) and (27.5bp,112.5bp)  .. (node_5);
\draw (36.0bp,115.0bp) node {$2$};
\draw [blue,->] (node_5) ..controls (27.5bp,60.442bp) and (27.5bp,41.496bp)  .. (node_3);
\draw (36.0bp,44.0bp) node {$1$};
\end{tikzpicture}$}
,\quad
D_{4}\colon
\scalebox{0.8}{$\begin{tikzpicture}[anchorbase,>=latex,line join=bevel,xscale=0.65,yscale=0.5,
every path/.style={very thick}]
\node (node_0) at (61.5bp,150.5bp) [draw,draw=none] {$\bullet$};
\node (node_5) at (61.5bp,79.5bp) [draw,draw=none] {$\bullet$};
\node (node_1) at (61.5bp,363.5bp) [draw,draw=none] {$\bullet$};
\node (node_3) at (61.5bp,292.5bp) [draw,draw=none] {$\bullet$};
\node (node_2) at (61.5bp,434.5bp) [draw,draw=none] {$\bullet$};
\node (node_4) at (27.5bp,221.5bp) [draw,draw=none] {$\bullet$};
\node (node_6) at (96.5bp,221.5bp) [draw,draw=none] {$\bullet$};
\node (node_7) at (61.5bp,8.5bp) [draw,draw=none] {$\bullet$};
\draw [red,->] (node_0) ..controls (61.5bp,131.44bp) and (61.5bp,112.5bp)  .. (node_5);
\draw (70.0bp,115.0bp) node {$2$};
\draw [red,->] (node_1) ..controls (61.5bp,344.44bp) and (61.5bp,325.5bp)  .. (node_3);
\draw (70.0bp,328.0bp) node {$2$};
\draw [blue,->] (node_2) ..controls (61.5bp,415.44bp) and (61.5bp,396.5bp)  .. (node_1);
\draw (70.0bp,399.0bp) node {$1$};
\draw [spinach,->] (node_3) ..controls (52.5bp,273.24bp) and (42.895bp,253.74bp)  .. (node_4);
\draw (58.0bp,257.0bp) node {$3$};
\draw [black,->] (node_3) ..controls (70.764bp,273.24bp) and (80.653bp,253.74bp)  .. (node_6);
\draw (91.0bp,257.0bp) node {$4$};
\draw [black,->] (node_4) ..controls (36.5bp,202.24bp) and (46.105bp,182.74bp)  .. (node_0);
\draw (58.0bp,186.0bp) node {$4$};
\draw [blue,->] (node_5) ..controls (61.5bp,60.442bp) and (61.5bp,41.496bp)  .. (node_7);
\draw (70.0bp,44.0bp) node {$1$};
\draw [spinach,->] (node_6) ..controls (87.236bp,202.24bp) and (77.347bp,182.74bp)  .. (node_0);
\draw (91.0bp,186.0bp) node {$3$};
\end{tikzpicture}$}
\end{gather*}

These are small enough to see the rule: a vertex is an element of the
crystal, and an arrow labeled $i$ records the action of the lowering
operator $f_i$. (The example is taken from \cite{MaTu-klrw-crystal} 
where the reader can find code to create more examples.)
\end{Example}

\begin{Definition}
We let $\gcrys$ denote the (strictification of the) $\kk$-linear monoidal category whose indecomposable objects are the crystals $B(\lambda), \lambda \in \Lambda_+$, corresponding to simple $\g$-modules. Its morphisms are $\kk$-linear combinations of crystal morphisms, with tensor product given by the tensor product of crystals, and direct sum given by disjoint union.
\end{Definition}

We will see later that $\gcrys$ is semisimple and that $\{B(\lambda)\mid \lambda \in \Lambda_+\}$ is the set of (isomorphism classes of) simple objects in $\gcrys$.

\begin{Notation} 
For later use, if $\mathcal{A}$ is a $\kk$-linear category, then we denote by $\mathcal{A}^c$ its \emph{additive idempotent completion}. (We also call this Cauchy completion.)
\end{Notation}

\subsection{Generators}\label{subsec:generators}

Let \(\mathcal{C}\) be a $\kk$-linear additive monoidal subcategory of \(\gcrys\). Let $\cong_{\otimes}$ denote monoidal $\kk$-linear equivalence (which is automatically additive).

\begin{Definition}
Let \(B=\{B_1,B_2,\dots\}\) be a (countable) collection of objects of \(\mathcal{C}\).
We write
\[
\Fund(\mathcal{C},B)
\]
for the full monoidal subcategory of \(\mathcal{C}\) whose objects are tensor
words in the objects \(B_i\), including the empty tensor word \(\mathbbm{1}\) (representing the trivial crystal which is the monoidal unit).
Thus an object of \(\Fund(\mathcal{C},B)\) is of the form
\[
B_{\epsilon}=B_{\epsilon_1}\otimes\cdots\otimes B_{\epsilon_m},
\]
where \(\epsilon=(\epsilon_1,\dots,\epsilon_m)\) is a word in the indexing set
of \(B\). We say that \(B\) monoidally generates \(\mathcal{C}\) up to
Cauchy completion if
\[
\Fund(\mathcal{C},B)^c\cong_{\otimes} \mathcal{C}.
\]
When \(B\) is clear from the context, we write simply
\(\Fund(\mathcal{C})\).
\end{Definition}

Note that \(\Fund(\mathcal{C},B)\) is not additive: we allow tensor
products of the chosen generators, but we do not formally add direct sums or
idempotent summands. This is the point of the notation above. The category
\(\Fund(\mathcal{C},B)\) is the small monoidal category for which we seek a
generators-and-relations presentation; the additive and idempotent complete
category \(\mathcal{C}\) is then recovered by Cauchy completion.

\begin{Remark}
For \(\gcrys\), the most natural choice of \(B\) is the set of fundamental
crystals. Other choices are possible. For example, for \(\mathfrak{sl}_2\) one
could instead take the crystal of the direct sum of the vector representation
and the trivial representation (see \autoref{S:Motzkin} for such an example). 
\end{Remark}

For the rest of the paper we assume out of convenience that
\(\g\) is simple, \(\mathcal{C}=\gcrys\), and that
$B_i=B(\omega_i)$
is the crystal of the \(i\)th fundamental representation. The results that follow apply equally well to any monoidal subcategory of $\gcrys$ and we give an example of this in \autoref{S:SO3}.   

\begin{Definition}
By the \textbf{highest element} of $B_{\epsilon}$ we mean the highest weight element of the top summand of $B_{\epsilon}$. By the \textbf{weight} of $B_{\epsilon}$ we mean that of its highest element, or, equivalently, that of the top summand.
\end{Definition}

\begin{Definition} Let $f:B_\epsilon \to B_{\epsilon'}$ be a morphism in $\Fund(\mathcal{C})$.
\begin{enumerate}

\item $f$ is called a \textbf{basic morphism} if it is zero on all but a single connected component of $B_{\epsilon}$, say one of highest weight $\lambda$. In this case we say $f$ is of \textbf{weight} $\lambda$, $\wt(f)=\lambda$.

\item $f$ \textbf{decreases weight} (resp. \textbf{increases weight}) if the weight of $B_\epsilon'$ is less (resp. greater) than or equal to that of $B_\epsilon$.

\end{enumerate}
If $f$ is injective outside its kernel, write $\overline{f}$ for the map in the opposite direction, namely the inverse of $f$ on the image of $f$, and sending all other elements to 0. We extend this definition linearly to arbitrary morphisms. If $S$ is a set of morphisms, write $\overline{S}$ for $\{\overline{s}\mid s\in S\}$.
\end{Definition}

One has \(\overline{f\circ g}=\overline{g}\circ\overline{f}\), and we will use this without further comment.

\begin{Lemma}\label{lem:weight-factorization}
Every morphism \(f\colon B_\epsilon\to B_{\epsilon'}\) in
\(\Fund(\mathcal{C})\) can be written as
\[
	f=\sum c_i\, t_i\circ b_i,
\]
where \(c_i\in\kk\), the morphisms \(b_i\) decrease weight, and the morphisms
\(t_i\) increase weight. Moreover, the \(b_i\) and \(t_i\) may be chosen to be
basic morphisms.
\end{Lemma}

\begin{proof}
This is the standard projection-inclusion decomposition with respect to the
connected components of \(B_\epsilon\) and \(B_{\epsilon'}\). 
(And well-known, see e.g. \cite{AnStTu-cellular-tilting}.) Namely, a strict
crystal morphism sends a connected component either to zero or isomorphically
onto a connected component of the same highest weight. Hence a basis of the
Hom-space is given by first projecting onto one connected component and then
including it into the target. These projections decrease weight, the inclusions
increase weight, and both are basic morphisms.
\end{proof}

\begin{Lemma}\label{hw element lemma}
A basic morphism $f:B_\epsilon \to B_{\epsilon'}$ has weight equal to that of $B_{\epsilon'}$ if and only if the image of $f$ contains the highest element of $B_{\epsilon'}$. Dually, $f$ has weight equal to that of $B_{\epsilon}$ if and only if the highest element of $B_{\epsilon}$ is not sent to 0. 
\end{Lemma}

\begin{proof}
This is immediate from the definition of weight.
\end{proof}

We now describe the generating morphisms for $\Fund(\mathcal{C})$. Let $B(\lambda)\in \mathcal{C}$ denote a connected crystal of highest weight $\lambda$. For $\lambda\in \Lambda_+$ sufficiently deep in the dominant chamber, we have 
\begin{gather}\label{eqn: fund decomp}
B(\lambda)\otimes B(\omega_k)\cong \bigoplus_{b\in B_k} B\big(\lambda +\wt(b)\big),
\end{gather}
with the decomposition independent of $\lambda$. (If $h$ is the highest weight element of $B(\lambda)$, then each $h\otimes b$ is a highest weight element.) 

\begin{Definition}\label{atom}
Let \(I=\bigsqcup_{k=1}^n B_k\). For \(b\in B_k\), put
\(\lambda(b)=\sum_i\varepsilon_i(b)\omega_i\). Then
\(\lambda(b)+\wt(b)=\sum_i\varphi_i(b)\omega_i\), by the crystal axiom
\(\varphi_i(b)-\varepsilon_i(b)=\langle\wt(b),\alpha_i^\vee\rangle\).

Let \(B_{\epsilon(b)}\) be the tensor word
\(B_1^{\otimes\varepsilon_1(b)}\otimes\cdots\otimes
B_n^{\otimes\varepsilon_n(b)}\), and let \(B_{\delta(b)}\) be the tensor word
\(B_1^{\otimes\varphi_1(b)}\otimes\cdots\otimes
B_n^{\otimes\varphi_n(b)}\). Thus both tensor words have weakly increasing
labels, and their weights are \(\lambda(b)\) and \(\lambda(b)+\wt(b)\),
respectively.

Let \(c\) be the highest element of \(B_{\epsilon(b)}\). Then \(c\otimes b\)
is a highest weight element of \(B_{\epsilon(b)}\otimes B_k\). We define
\(f_b\colon B_{\epsilon(b)}\otimes B_k\to B_{\delta(b)}\) to be the basic
morphism which maps the connected component generated by \(c\otimes b\) onto
the top component of \(B_{\delta(b)}\), and is zero on all other connected
components.

The morphism \(f_b\) is called an \textbf{atom}. We write
\(\mathcal{A}:=\{f_b\mid b\in I\}\) for the set of atoms.
\end{Definition}

The point of the definition is that $f_b$ is the `minimal' realization of the component indexed by $b$ in the decomposition shown in \eqref{eqn: fund decomp}.

\begin{Example}\label{atom examples}
For this example, and for the examples below, the reader may find it useful to recall the basic tableau combinatorics; see, for instance, \cite{BuSc-crystal-bases}.

\begin{enumerate}

\item For $\g=\mathfrak{sl}_2$, being sufficiently deep in the dominant chamber means $\lambda > \omega_1$. In this case we have $B(\lambda)\otimes B_1\cong B(\lambda-\omega_1)\oplus B(\lambda+\omega_1)$, the elements of $I=B_1=\{-1,1\}$ are of weight $-\omega_1,\omega_1$, and $\varepsilon_1(-1)=1,\varepsilon_1(1)=0$. The map $f_{1}$ is just the identity map $B_1\to B_1$, and the map $f_{-1}$ is the `evaluation' map $B_1\otimes B_1\to \mathbf{1}$ which projects onto the trivial crystal.

\item For $\mathfrak{sl}_3$, the decomposition rules are (writing weights in fundamental coordinates):
\begin{gather*}
B_{(a,b)} \otimes B_1
\cong
B_{(a+1,b)}
\oplus \mathbf{1}_{a>0}\, B_{(a-1,b+1)}
\oplus \mathbf{1}_{b>0}\, B_{(a,b-1)},
\\
B_{(a,b)} \otimes B_2
\cong
B_{(a,b+1)}
\oplus \mathbf{1}_{b>0}\, B_{(a+1,b-1)}
\oplus \mathbf{1}_{a>0}\, B_{(a-1,b)},
\end{gather*}
In this case we have six atoms, which we draw as \[\upvert,\downvert,\trisplit,\trimergedown,\capL,\capR,\] where the arrow leaving the vertex labeled $B_1$ (resp. $B_2$) goes up (resp. down).

\item More generally, in type $A$ we may treat an element $c\in B_k$ in type $A$ as a column of increasing integers of height $k$, and an element of $B_{\epsilon}$, with labels weakly increasing, as a filling of a Young diagram of shape $\mu = (\mu_1,\dots, \mu_n)$, (meaning there are $\mu_1$ columns of height 1, etc.), such that $B_{\epsilon}$ contains $\mu_k$ factors of $B_k, 1\le k\le n$. (Using the convention of \cite{kashiwara1994crystal}, we read the columns from right to left). Let $f_b: B_{\epsilon}\otimes B_k\to B_{\delta}$ be an atom, and assume an element of $B_{\epsilon}$ (resp. $B_{\delta}$) corresponds to a filling of a Young diagram of shape $\mu$ (resp. $\mu'$). Then $\mu'$ is obtained from $\mu$ by adding a vertical strip, up to removing columns of length $n$ (this is the Pieri rule). If $A$ is a filling of a diagram of shape $\mu$, and $b$ is a column, let $[A,b]$ denote the filling obtained by inserting $b$ into $A$ via Schensted's algorithm. Then if $b, c\in B_k$ we have 
\[f_b(A\otimes c)=\begin{cases} [A,c] & \text{ if $[A,c]$ is of shape $\mu'$,} \\
0 &  \text{ else.}
\end{cases}\]
For example, if \(n\ge 4\), \(k=2\), \(b=\tableau{{2}\\{3}}\in B_2\), then \(\mu=(1)\), \(\mu'=(3)\), and we have
\[
f_b\!\left(\tableau{{1}}\otimes\tableau{{2}\\{3}}\right)=\tableau{{1}\\{2}\\{3}}, \qquad f_b\!\left(\tableau{{2}}\otimes\tableau{{2}\\{3}}\right)=0.
\]
A similar description of the atoms in terms of column insertion exists in types $B,C,D$, see \cite[\S 3.3]{lecouvey2003schensted} and \cite[\S 4]{lecouvey2002typeC} for a description of the insertion algorithms in these cases.
\end{enumerate}
We will revisit several of these in \autoref{S:Examples}.
\end{Example}

For $\g$ of higher rank, we will represent the atom $f_b$ diagrammatically by lines going in and out of a \emph{trapezium}, so that
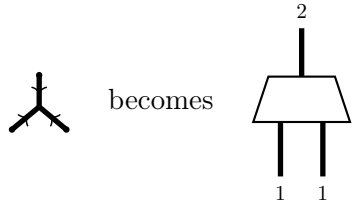
\[
\trimergedown
\quad\text{ becomes }\quad
\begin{tikzpicture}[
thick,
scale=0.8,
transform shape,
baseline={([yshift=-.5ex]current bounding box.center)}
]
\usetikzlibrary{calc}

\coordinate (BL) at (0.2,0);
\coordinate (BR) at (1.8,0);
\coordinate (TR) at (1.55,0.75);
\coordinate (TL) at (0.45,0.75);

\draw (BL) -- (BR) -- (TR) -- (TL) -- cycle;

\coordinate (BM) at ($(BL)!0.5!(BR)$);
\coordinate (TM) at ($(TL)!0.5!(TR)$);

\draw[usual] (0.65,-0.9) node[below] {\small $1$} -- (0.65,0);
\draw[usual] (1.35,-0.9) node[below] {\small $1$} -- (1.35,0);
\draw[usual] (1,0.75) -- (1,1.55) node[above] {\small $2$};
\end{tikzpicture}.\]
The box is wider at the bottom to indicate that the atom decreases weight; the flipped version will be drawn with the top wider.

\begin{Remark}\label{G2 atom dependence}
It is possible for there to be relations between atoms. For example, we have a relation of the following form in type $\mathrm{G}_2$. 
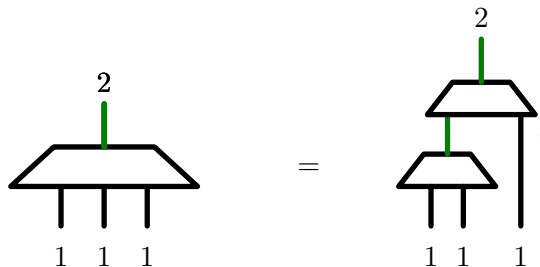
\begin{equation*}
\begin{tikzpicture}[
baseline=(current bounding box.center),
x=.95cm,y=.95cm,
line cap=round,
line join=round,
twostrand/.style={usual,green!50!black,line width=2pt}
]

\begin{scope}[shift={(0,0)}]
\draw[usual]
(0.00,0.55) -- (2.60,0.55) -- (2.00,1.10) -- (0.60,1.10) -- cycle;

\foreach \x in {0.70,1.30,1.90}{
\draw[usual] (\x,0.00) -- (\x,0.55);
\node at (\x,-0.42) {$1$};

\draw[twostrand] (1.30,1.10) -- (1.30,1.70);
\node at (1.30,2.00) {$2$};
}
\end{scope}

\node at (4.15,0.80) {$=$};

\begin{scope}[shift={(5.40,0)}]
\draw[usual]
(0.00,0.55) -- (1.35,0.55) -- (1.00,1.00) -- (0.35,1.00) -- cycle;

\draw[twostrand] (0.68,1.00) -- (0.68,1.55);

\draw[usual]
(0.40,1.55) -- (1.90,1.55) -- (1.55,2.00) -- (0.75,2.00) -- cycle;

\draw[twostrand] (1.15,2.00) -- (1.15,2.60);
\node at (1.15,2.90) {$2$};

\foreach \x in {0.45,0.90}{
\draw[usual] (\x,0.00) -- (\x,0.55);
\node at (\x,-0.42) {$1$};
}

\draw[usual] (1.70,0.00) -- (1.70,1.55);
\node at (1.70,-0.42) {$1$};
\end{scope}

\end{tikzpicture}.
\end{equation*}
We will see several other relations below.
\end{Remark}

Let \(\#B_i\) denote the cardinality of the finite crystal \(B_i\).

\begin{Proposition}\label{prop:atom-count}
The number of atoms is
\(\lvert\mathcal{A}\rvert=\sum_{i=1}^n\lvert B_i\rvert\). For the classical
types this gives
\[
\begin{array}{c|c}
\text{type} & \lvert\mathcal{A}\rvert \\ \hline
A_{n-1} & 2^n-2 \\
B_n & 4^n+2^n-1-\binom{2n+1}{n} \\
C_n & \binom{2n+1}{n}-1 \\
D_n & \displaystyle \sum_{i=1}^{n-2}\binom{2n}{i}+2^n
\end{array}
.
\]
Equivalently, the entry for \(D_n\) is
\(2^{2n-1}+2^n-1-\binom{2n}{n-1}-\frac12\binom{2n}{n}\).
For \(\mathrm{G}_2,F_4,E_6,E_7,E_8\), the corresponding numbers are
\[
21,\quad 1625,\quad 3759,\quad 404699,\quad 7054732527.
\]
\end{Proposition}

\begin{proof}
By definition, there is one atom \(f_b\) for each element
\(b\in\bigsqcup_i B_i\), hence
\(\lvert\mathcal{A}\rvert=\sum_i\lvert B_i\rvert\). Moreover
\(\lvert B_i\rvert=\dim V(\omega_i)\).

It remains only to recall the dimensions of the fundamental representations.
In type \(A_{n-1}\), these are \(\binom{n}{i}\), \(1\leq i\leq n-1\), whose
sum is \(2^n-2\). In type \(B_n\), they are \(\binom{2n+1}{i}\),
\(1\leq i<n\), together with the spin representation of dimension \(2^n\);
this gives the displayed formula. In type \(C_n\), they are
\(\binom{2n}{i}-\binom{2n}{i-2}\), \(1\leq i\leq n\), and the sum telescopes
to \(\binom{2n+1}{n}-1\). In type \(D_n\), they are
\(\binom{2n}{i}\), \(1\leq i\leq n-2\), together with the two half-spin
representations, both of dimension \(2^{n-1}\). This gives the stated
formula for \(D_n\). The exceptional values are obtained in the same way,
using the Weyl dimension formula.
\end{proof}

\begin{Definition}
The set of \textbf{crossings}, $\mathcal X=\{\beta_{i,j}:1\le i, j\le n, i\neq j\}$, consists of one choice of isomorphism $\beta_{i,j}: B_i\otimes B_j\to B_j\otimes B_i$ for each pair $(i,j), i\neq j$, which are compatible in the sense that $\beta_{i,j}^{-1}=\beta_{j,i}$.
\end{Definition}

We draw $\beta_{i,j}$ as a crossing 
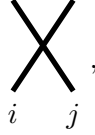
\[
\begin{tikzpicture}[thick, scale=0.8, baseline=(current bounding box.center)]
\draw[usual] (0,0) node[below] {$i$} -- (1,1.5);
\draw[usual] (1,0) node[below] {$j$} -- (0,1.5);
\end{tikzpicture},
\]
which is standard notation.

\begin{Remark}
In type $A,B,C,\mathrm{G}_2$ the tensor product decomposition of two fundamental representations is multiplicity-free and $\beta_{i,j}$ is unique; this is not true in type $D,E,F$. When the choice of $C$ is not unique, a natural choice is to pick the crystal commutors in the sense of \cite{Kamnitzcrystal}. (Here ``commutor'' is meant literally: this choice gives
a so-called \emph{coboundary structure}, not a braiding.) 
\end{Remark}

\begin{Example}\label{Type A crossing}
In type $A, C$, the unique crossings $\beta_{i,j}$ are realized by sequences of \emph{jeu de taquin} moves, i.e. applications of relations in the plactic monoid. (The symplectic jeu de taquin in type $C$ was introduced by \cite{sheats1999symplectic}, and proved to commute with crystal operators in \cite[\S 6]{lecouvey2002typeC}). There is also a jeu de taquin of type $B$, introduced in \cite[\S 3.5]{lecouvey2003schensted}, which gives $\beta_{i,j}$ when $i,j \neq n$; elements of $B_n$ are \emph{spin columns}, which introduce additional complexities.  

In type $A$, it is straightforward to translate jeu de taquin on two columns into the following explicit description: if the element in $B_k$ and $B_l$ are represented by subsets $I=\{i_1<\dots < i_k\}$ and $J=\{j_1<\dots <j_l\}$ respectively, and $k<l$, we obtain a set $P$ by processing the elements of $I$ in decreasing order: at the step of $i_r$, let $p_r$ be the largest yet-unused element of $J$ which is less than or equal to $i_r$; if no element of $J$ is less than or equal to $i_r$, pick the largest unused element of $J$. Let $P=\{p_1,\dots, p_k\}$. Then $I\otimes J$ is sent to $((J\setminus P)\cup I) \otimes P$. If $k>l$, then define $Q$ by processing elements of $J$ in increasing order, for $j_r$ picking the smallest unused element of $I$ which is greater than or equal to it, or the smallest unused element. Then $I\otimes J\mapsto Q\otimes ((I\setminus Q)\cup J)$. For example we have $\beta_{2,3}((1,2)\otimes (1,2,3))=(1,2,3)\otimes (1,2)$, corresponding to 
\[
\begin{array}{cc}
1&1\\
2&2\\
3&\Box
\end{array}
\quad\longrightarrow\quad
\begin{array}{cc}
1&1\\
2&2\\
\Box&3
\end{array}
\quad\longrightarrow\quad
\begin{array}{cc}
1&1\\
\Box&2\\
2&3
\end{array}
\quad\longrightarrow\quad
\begin{array}{cc}
\Box&1\\
1&2\\
2&3
\end{array}
.
\]
(We are using the convention of \cite{kashiwara1994crystal}, reading columns from right to left.) To go from $(123)\otimes (12)$, we apply the slide in the reverse direction, from inner corner to outer corner.
\end{Example}

The following is the famous Reidemeister 3 move.
\begin{gather*}
\begin{tikzpicture}[thick, scale=0.6, baseline=(current bounding box.center)]
\draw[usual] (0,0) node[below] {$i$} -- (1.5,2);
\draw[usual] (1.5,0) node[below] {$k$} -- (0,2);
\draw[usual] (0.75,0) node[below] {$j$} .. controls (0.75,0.4) and (-0.3,0.6) .. (-0.3,1) .. controls (-0.3,1.4) and (0.75,1.6) .. (0.75,2);
\end{tikzpicture}
=
\begin{tikzpicture}[thick, scale=0.6, baseline=(current bounding box.center)]
\draw[usual] (0,0) node[below] {$i$} -- (1.5,2);
\draw[usual] (1.5,0) node[below] {$k$} -- (0,2);
\draw[usual] (0.75,0) node[below] {$j$} .. controls (0.75,0.4) and (1.8,0.6) .. (1.8,1) .. controls (1.8,1.4) and (0.75,1.6) .. (0.75,2);
\end{tikzpicture}.
\end{gather*}

\begin{Lemma}
The Reidemeister 3 move holds in type $A$ and $C$, and type $B$, if $i,j,k \neq n$ (but not in general).
\end{Lemma}

\begin{proof}
By the confluence of jeu de taquin slides.
\end{proof}

\subsection{Two bases for hom-spaces}

We now construct bases.

\begin{Remark}
To get bases of hom-spaces we follow a standard strategy that has appeared many times. In fact, \autoref{wedderburn basis} says that the basis given therein is a sandwich basis in the sense of \cite{Tu-sandwich-cellular}. The basis in \autoref{bmt basis} is also a sandwich basis, but we do not need this here.
\end{Remark}

The atoms have a useful interpretation in terms of \emph{branching graphs} of tensor products. Let $B_{\epsilon}\in \Fund(\mathcal{C})$. Its connected components (although these components are not themselves objects in $\Fund(\mathcal{C})$) can all be reached by a sequence of elementary modifications of weights as in \eqref{eqn: fund decomp}. Explicitly, if $b_1\otimes b_2\otimes \dots \otimes b_m$ is the highest weight element of a component $C$ of $B_{\epsilon}$, then each prefix $b_1\otimes\dots\otimes b_k, k\le m$ is also highest weight, and $C$ corresponds to the sequence $(b_1,\dots,b_m)$. 

\begin{Example}\label{fig:sl3 weight branching}
For example, in the case of $\mathfrak{sl}_3$, the connected components of $B_1\otimes B_1\otimes B_2$ correspond to the sequences $\big((1,0),(1,0),(0,1)\big), \big((1,0),(1,0),(-1,0)\big),\big((1,0),(-1,1),(1,-1)\big)$ and $\big((1,0),(-1,1),(0,1)\big)$, as illustrated below.
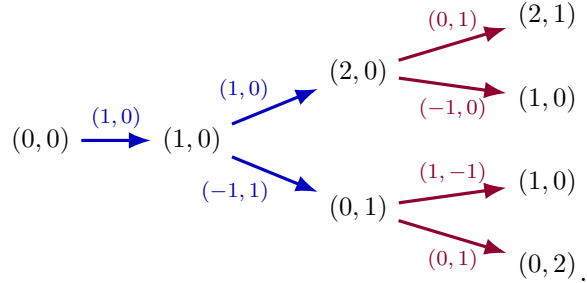
\[
\begin{tikzpicture}[
>=Latex,
font=\small,
v/.style={inner sep=1pt},
blueedge/.style={->, very thick, blue!75!black, shorten >=3pt, shorten <=3pt},
purpleedge/.style={->, very thick, purple!75!black, shorten >=3pt, shorten <=3pt},
edgelab/.style={font=\scriptsize, inner sep=1pt}
]

\node[v] (a) at (0,0)       {$(0,0)$};
\node[v] (b) at (2.0,0)     {$(1,0)$};

\node[v] (c) at (4.2,0.9)   {$(2,0)$};
\node[v] (d) at (4.2,-0.9)  {$(0,1)$};

\node[v] (e) at (6.7,1.65)  {$(2,1)$};
\node[v] (f) at (6.7,0.55)  {$(1,0)$};
\node[v] (g) at (6.7,-0.55) {$(1,0)$};
\node[v] (h) at (6.7,-1.65) {$(0,2)$};

\draw[blueedge] (a) -- node[edgelab, above=3pt] {$(1,0)$} (b);

\draw[blueedge] (b) -- node[edgelab, above left=1pt] {$(1,0)$} (c);
\draw[blueedge] (b) -- node[edgelab, below left=1pt] {$(-1,1)$} (d);

\draw[purpleedge] (c) -- node[edgelab, above=3pt] {$(0,1)$} (e);
\draw[purpleedge] (c) -- node[edgelab, below=3pt] {$(-1,0)$} (f);

\draw[purpleedge] (d) -- node[edgelab, above=3pt] {$(1,-1)$} (g);
\draw[purpleedge] (d) -- node[edgelab, below=3pt] {$(0,1)$} (h);

\end{tikzpicture}.
\]
This is the branching graph for $B_1\otimes B_1\otimes B_2$ (a cutoff is displayed). Here we are indexing the atoms by weight.
\end{Example}

Thus, there is a natural correspondence between paths in the branching diagram, and special morphisms built from the atoms $\mathcal{A}$ which are their diagrammatic incarnations. Fix now an object $B_{\epsilon}=B_{\epsilon_1}\otimes\dots\otimes B_{\epsilon_m}$ and its associated branching graph with $m$ levels, as in \autoref{fig:sl3 weight branching} above. We associate a morphism to each vertex of the graph inductively as follows:  
\begin{enumerate}

\item We associate the empty morphism to $(0,0)$. 

\item Assume we have assigned a morphism to each vertex at level $k\ge 0$, i.e. each component of $B_{\epsilon_1}\otimes\dots \otimes B_{\epsilon_k}$. Now let $C$ be the component of $B_{\epsilon_1}\otimes\dots \otimes B_{\epsilon_{k+1}}$ which is obtained from a component $C'$ at level $k$ by the branch labeled by $b\in B_{\epsilon_{k+1}}$. Let $D'$ be the morphism associated to $C'$. Morally, we want to define the morphism associated to $C$ to be $(\id\otimes f_b)\circ(D'\otimes \id_{B_{\epsilon_{k+1}}})$. However, this composite may not be type correct, so we need to first use crossings to permute the factors of $B_{\epsilon_1}\otimes\dots \otimes B_{\epsilon_k}$ in a minimal way, so that the rightmost factors, after shuffling, together with the new factor $B_{\epsilon_{k+1}}$, are exactly the domain of $f_b$. Call this permutation $\beta$. Then we define the morphism associated to $C$ to be $$D:=(\id\otimes f_b)\circ((\beta\circ D')\otimes \id_{B_{\epsilon_{k+1}}}).$$ 
\end{enumerate}

For example, for the sequence $\big((1,0),(-1,1),(1,0),(1,0),(0,-1)\big)$ corresponding to a component of weight $(2,0)$ in $B_{1}^{\otimes 5}$ for $\mathfrak{sl}_3$, the construction is illustrated below:
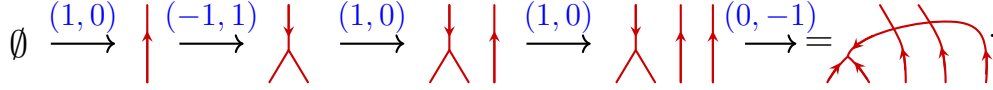
\[
\begin{tikzpicture}[scale=0.85, baseline=(current bounding box.center)]
\tikzset{
web/.style={red!80!black, thick, line cap=round, line join=round},
warrow/.style={-stealth, red!80!black, thick, line cap=round, line join=round},
blabel/.style={text=blue, font=\large}
}

\node at (0,0) {\Large $\emptyset$};

\draw[->, thick] (0.5,0) -- (1.5,0)
node[midway, above, blabel] {$(1,0)$};

\draw[web] (2.0,-0.6) -- (2.0,0.6);
\draw[warrow] (2.0,-0.6) -- (2.0,0.2);

\draw[->, thick] (2.5,0) -- (3.5,0)
node[midway, above, blabel] {$(-1,1)$};

\draw[web] (3.9,-0.6) -- (4.2,-0.1) -- (4.5,-0.6);
\draw[web] (4.2,-0.1) -- (4.2,0.6);
\draw[warrow] (4.2,0.6) -- (4.2,0.1);

\draw[->, thick] (5.0,0) -- (6.0,0)
node[midway, above, blabel] {$(1,0)$};

\draw[web] (6.4,-0.6) -- (6.7,-0.1) -- (7.0,-0.6);
\draw[web] (6.7,-0.1) -- (6.7,0.6);
\draw[warrow] (6.7,0.6) -- (6.7,0.1);
\draw[web] (7.4,-0.6) -- (7.4,0.6);
\draw[warrow] (7.4,-0.6) -- (7.4,0.2);

\draw[->, thick] (7.9,0) -- (8.9,0)
node[midway, above, blabel] {$(1,0)$};

\draw[web] (9.3,-0.6) -- (9.6,-0.1) -- (9.9,-0.6);
\draw[web] (9.6,-0.1) -- (9.6,0.6);
\draw[warrow] (9.6,0.6) -- (9.6,0.1);
\draw[web] (10.3,-0.6) -- (10.3,0.6);
\draw[warrow] (10.3,-0.6) -- (10.3,0.2);
\draw[web] (10.8,-0.6) -- (10.8,0.6);
\draw[warrow] (10.8,-0.6) -- (10.8,0.2);

\draw[->, thick] (11.3,0) -- (12.1,0)
node[midway, above, blabel] {$(0,-1)$};

\node at (12.45,0) {\Large $=$};

\begin{scope}[shift={(12.8,0)}]

\draw[web]
(-0.22,-0.60)
.. controls (-0.14,-0.47) and (-0.03,-0.31) .. (0.10,-0.20);
\draw[warrow]
(-0.20,-0.56)
.. controls (-0.15,-0.48) and (-0.08,-0.38) .. (-0.02,-0.30);

\draw[web]
(0.43,-0.60)
.. controls (0.35,-0.47) and (0.23,-0.31) .. (0.10,-0.20);
\draw[warrow]
(0.40,-0.55)
.. controls (0.34,-0.46) and (0.27,-0.36) .. (0.20,-0.28);

\draw[web]
(1.00,-0.60)
.. controls (1.00,-0.39) and (0.99,-0.19) .. (0.93,-0.01)
.. controls (0.86,0.18) and (0.72,0.39) .. (0.60,0.60);
\draw[warrow]
(0.99,-0.40)
.. controls (0.99,-0.29) and (0.97,-0.18) .. (0.94,-0.07);

\draw[web]
(1.62,-0.60)
.. controls (1.62,-0.38) and (1.60,-0.17) .. (1.54,0.01)
.. controls (1.47,0.20) and (1.32,0.42) .. (1.20,0.60);
\draw[warrow]
(1.61,-0.40)
.. controls (1.60,-0.28) and (1.58,-0.16) .. (1.55,-0.05);

\draw[web]
(2.25,-0.60)
.. controls (2.27,-0.39) and (2.27,-0.18) .. (2.22,0.02)
.. controls (2.17,0.22) and (2.02,0.32) .. (1.76,0.34)
.. controls (1.45,0.36) and (1.12,0.31) .. (0.82,0.23)
.. controls (0.55,0.16) and (0.33,0.06) .. (0.20,-0.05)
.. controls (0.13,-0.11) and (0.10,-0.17) .. (0.10,-0.20);

\draw[warrow]
(2.25,-0.47)
.. controls (2.26,-0.34) and (2.26,-0.22) .. (2.24,-0.09);

\draw[warrow]
(0.48,0.12)
.. controls (0.36,0.06) and (0.25,-0.01) .. (0.17,-0.09);

\end{scope}

\end{tikzpicture}.
\]
Notice that, in the last step, it was necessary to apply $\beta_{2,1}$ twice to route the first strand to the right before applying the cap.

\begin{Definition}\label{distinguished-bots}
A morphism obtained in the correspondence above is called a \textbf{distinguished bottom diagram}. Denote by $\mathcal{D}_{\epsilon}$ the set of distinguished bottom diagrams with domain $B_{\epsilon}$. Call the flips of such diagrams the \textbf{distinguished top diagrams}. The \textbf{weight} of a distinguished bottom (resp. top) diagram is the top weight of its codomain (resp. domain).
\end{Definition}

Note that $\D_{\epsilon}$ is generated by $\mathcal{A}\cup \mathcal X$ via composition and tensor products. 
\begin{Proposition}\label{bijection branching}
There is a bijection between  $\mathcal{D}_{\epsilon}$ and the set $C_{\epsilon}$ of connected components of $B_{\epsilon}$, such that if $D$ corresponds to $C$, then $D$ decreases weight and $\wt(D)$ is the highest weight of $C$. Moreover, if $\wt(D)=\lambda$, the restriction of $D$ to the components of $B_{\epsilon}$ of highest weight $\lambda$ is nonzero on only one of these components, namely the component $C$ corresponding to $D$. 
\end{Proposition}

\begin{proof}
The first statement is clear from the construction. For the second statement, note that $D$ is clearly injective outside its kernel, being built from atoms and crossings which have this property, so it suffices to show that $D$ maps $C$ onto the top component of its codomain, where $C$ is the component of $B_{\epsilon}$ corresponding to $D$. As above, let $C'$ be the component corresponding to $D'$, the morphism built in the previous step, whose codomain we denote by $Y$. By induction, the highest weight element $c'\in C'$ is mapped by $D'$ onto the top component $T$ of $Y$, with highest weight element $t$. It follows that if $b\in B_{\epsilon_{m+1}}$ is such that $c'\otimes b$ is highest weight, then $D'\otimes \id_{B_{\epsilon_{m+1}}}$ maps it onto $t\otimes b$. Postcomposing with some crossings on the $T$-factor sends it to $t'\otimes b$, where $t'$ is still the highest element, and postcomposing now with $\id\otimes f_b$ sends it to the highest element in the codomain of $D$. This completes the proof, since $c'\otimes b$ is the highest weight element of $C$, where $C$ is obtained from $C'$ by the atom $f_b$.
\end{proof}


For $\lambda\in \Lambda_+$, let $\D_{\epsilon, \lambda}:=\{D\in \D_{\epsilon}\mid \wt(D)=\lambda\}$ be the distinguished bottom diagrams with domain $B_{\epsilon}$ and weight $\lambda$. Given $f\in D_{\epsilon',\lambda}, g\in D_{\epsilon,\lambda}$, we form the composite $\overline{f}\circ P\circ g:B_{\epsilon}\to B_{\epsilon'}$ which we say is \textbf{of weight $\lambda$}, where $P$ is some permutation built from crossings. The next result does not depend on the particular choice of $P$; however, let us fix the convention that $P$ is such that whenever there are three strands with distinct labels, the crossing is `from the left' as illustrated below. (By \autoref{Type A crossing}, the choice is unique in types $A$ and $C$.)
\begin{gather}\label{left crossing convention}
\begin{tikzpicture}[thick, scale=0.6, baseline=(current bounding box.center)]
\draw[usual] (0,0) node[below] {$i$} -- (1.5,2);
\draw[usual] (1.5,0) node[below] {$k$} -- (0,2);
\draw[usual] (0.75,0) node[below] {$j$} .. controls (0.75,0.4) and (-0.3,0.6) .. (-0.3,1) .. controls (-0.3,1.4) and (0.75,1.6) .. (0.75,2);
\end{tikzpicture}.
\end{gather}
From now on, by a \emph{permutation} we always mean one in the above convention. 

\begin{Proposition}\label{bmt basis}
The set $\mathcal{B}(\epsilon,\epsilon')=\{\overline{f}\circ P\circ g\mid f\in \D_{\epsilon', \lambda}, g\in \D_{\epsilon,\lambda}, \lambda \in \Lambda_+\}$ is a basis of $\Hom_{\mathcal{C}}(B_{\epsilon},B_{\epsilon'})$.
\end{Proposition}
\begin{proof}
This set has the right cardinality by \autoref{bijection branching}. (Morally, $g$ is (up to lower components) the projection onto a component, and $\overline{f}$ is (up to lower components) the embedding of this component into $B_{\epsilon'}$). It is therefore enough to prove independence, but this follows from the second sentence in \autoref{bijection branching}: suppose for a contradiction that $\overline{f}\circ P\circ g=\sum \overline{f}_i\circ P_i\circ g_i$, where $\overline{f}\circ P\circ g$ does not appear on the right-hand side. By precomposing with the projection onto the unique component corresponding to $g$, and postcomposing with the projection onto the unique component corresponding to $f$, we get a contradiction.
\end{proof}

Write $S=\mathcal{A}\cup \overline{\mathcal{A}}\cup \mathcal X$. In particular, \autoref{bmt basis} tells us that $S$ generates all morphisms in $\Fund(\mathcal{C})$.



\begin{Definition}
A \emph{Jones--Wenzl projector} of type $\epsilon$, denoted $j_\epsilon$, is the endomorphism of $B_\epsilon$ which sends the top summand to itself and other summands to 0. 
\end{Definition}

If $\nu=\overline{f}\circ P\circ g\in \mathcal{B}(\epsilon,\epsilon')$,  define $\hat{\nu}:=\overline{f}\circ j_{\delta}\circ P\circ g$ where $P$ has codomain $B_{\delta}$.  
We have another basis, which is in many ways the more natural one:

\begin{Proposition}\label{wedderburn basis}
The set $\hat{\mathcal{B}}(\epsilon,\epsilon')=\{\hat{\nu}\mid \nu \in \mathcal{B}(\epsilon, \epsilon')\}$ is a matrix-unit basis of $\Hom_{\mathcal{C}}(B_{\epsilon},B_{\epsilon'})$ consisting of basic morphisms. Explicitly, if $\nu= \overline{f}\circ P\circ g \in \Hom_{\mathcal{C}}(B_{\epsilon}, B_{\epsilon'}), \, \mu = \overline{f'}\circ P'\circ g' \in \Hom_{\mathcal{C}}(B_{\epsilon'}, B_{\epsilon''})$, then writing $\eta= \overline{f'}\circ P''\circ g$ for any compatible permutation $P''$, we have \[\hat{\mu}\circ \hat{\nu} = \begin{cases} \hat{\eta} & \text{if \ } $g'=f$, \\  0 & \text{else}. \end{cases} \]
\end{Proposition}

\begin{proof}
This follows directly from the second statement in \autoref{bijection branching} and the definition of the Jones--Wenzl projectors.
\end{proof}

\begin{Remark}\label{JW remark}
The idea behind \autoref{wedderburn basis} is that if $ f\in \D_{\epsilon},f:B_{\epsilon}\to B_{\epsilon'}$, then $j_{\epsilon'}\circ f$ gives the projection onto the top summand of $B_{\epsilon}'$, and dually $\overline{f}\circ j_{\epsilon'}$ gives the embedding of this summand into $B_{\epsilon}$, cf. \cite[Lemma 4.16]{Alqady2025coboundary}. Of course, if $P:B_{\delta}\to B_{\delta'}$ is invertible, then $P\circ j_{\delta}=j_{\delta'}\circ P$, so there are no ambiguities when inserting Jones--Wenzl projectors. 
\end{Remark}

\subsection{Relations}\label{subsec:relations}

Having fixed the generators for $\Fund(\mathcal{C})$, we now describe the relations.
In the relations
below, we write all scalar coefficients as \(c_i\) to avoid overloading the
notation; these scalars are not meant to be the same from one relation to the
next.

First we have relations of the following types, which we call \emph{$H=I$ relations}.
\begin{gather}\tag{HI1}
\label{eq:HI1}
\begin{tikzpicture}[
baseline=(current bounding box.center),
x=1cm,y=1cm,
line width=0.8pt,
line cap=butt,
line join=miter,
lab/.style={inner sep=0pt,font=\large},
dots/.style={inner sep=0pt,font=\normalsize}
]
\def\s{0.16}
\def\Wrhs{1.70}
\begin{scope}[shift={(0,0)}]
\draw[usual] ({1.40+\s},0.55) -- ({3.05-\s},0.55) -- (3.05,1.20) -- (1.40,1.20) -- cycle;
\draw[usual] (0.00,1.55) -- (2.25,1.55) -- ({2.25-\s},2.20) -- ({0.00+\s},2.20) -- cycle;
\draw[usual] (1.90,-0.20) -- (1.90,0.55);
\draw[usual] (2.65,-0.20) -- (2.65,0.55);
\node[dots] at (2.275,0.175) {$\cdots$};
\draw[usual] (1.50,1.20) -- (1.50,1.55);
\draw[usual] (2.15,1.20) -- (2.15,1.55);
\node[dots] at (1.825,1.375) {$\cdots$};
\draw[usual] (0.60,2.20) -- (0.60,2.85);
\draw[usual] (1.65,2.20) -- (1.65,2.85);
\node[dots] at (1.125,2.525) {$\cdots$};
\draw[usual] (0.00,-0.20) -- (0.00,1.55);
\draw[usual] (0.35,-0.20) -- (0.35,1.55);
\draw[usual] (2.85,1.20) -- (2.85,2.85);
\draw[usual] (3.05,1.20) -- (3.05,2.85);
\node[dots] at (2.55,2.025) {$\cdots$};
\node[dots] at (0.85,0.675) {$\cdots$};
\node[lab] at (2.225,0.875) {$\bar{\bm{g}}$};
\node[lab] at (1.125,1.875) {$\bm{f}$};
\end{scope}
\node[lab] at (3.75,1.325) {$=$};
\node[lab] at (4.45,1.325) {$\displaystyle\sum c_i$};
\begin{scope}[shift={(5.35,0.60)}]
\draw[usual] (\s,0.9667) -- ({\Wrhs-\s},0.9667) -- (\Wrhs,1.45) -- (0,1.45) -- cycle;
\draw[usual] (\s,0.4833) rectangle ({\Wrhs-\s},0.9667);
\draw[usual] (0,0) -- (\Wrhs,0) -- ({\Wrhs-\s},0.4833) -- (\s,0.4833) -- cycle;
\node[lab, font=\small] at ({\Wrhs/2},1.208) {$\bar{\bm{f_i}}$};
\node[lab, font=\small] at ({\Wrhs/2},0.725) {$\bm{P_i}$};
\node[lab, font=\small] at ({\Wrhs/2},0.242) {$\bm{g_i}$};
\end{scope}
\end{tikzpicture}.
\end{gather}
\begin{gather}
\tag{HI2}\label{eq:HI2}
\begin{tikzpicture}[
  baseline=(current bounding box.center),
  x=1cm,y=1cm,
  line width=0.8pt,
  line cap=butt,
  line join=miter,
  lab/.style={inner sep=0pt,font=\large},
  dots/.style={inner sep=0pt,font=\normalsize}
]
\begin{scope}[shift={(0,0)}]
  \draw[usual] (-0.70,-0.20) -- (-0.70,1.55);
  \draw[usual] (-0.70,2.15) -- (-0.70,3.75);
  \draw[usual] (-0.20,-0.20) -- (-0.20,1.55);
  \draw[usual] (-0.20,2.15) -- (-0.20,3.75);
  \draw[usual,fill=white] (-0.90,1.55) rectangle (2.00,2.15);
  \node at (0.55,1.85) {$\bm{P}$};
  \draw[usual] (0.80,-0.20) -- (0.80,1.55);
  \draw[usual] (0.80,2.15) -- (0.80,2.50);
  \draw[usual] (1.90,-0.20) -- (1.90,0.55);
  \draw[usual] (3.30,-0.20) -- (3.30,0.55);
  \node[dots] at (2.60,0.175) {$\cdots$};
  \draw[usual] (1.70,0.55) -- (3.50,0.55) -- (4.10,1.20) -- (1.10,1.20) -- cycle;
  \node[lab] at (2.60,0.875) {$\bm{\overline{g}}$};
  \draw[usual] (1.80,1.20) -- (1.80,1.55);
  \draw[usual] (1.80,2.15) -- (1.80,2.50);
  \draw[usual] (2.40,1.20) -- (2.40,2.50);
  \draw[usual] (2.90,1.20) -- (2.90,3.75);
  \draw[usual] (3.40,1.20) -- (3.40,3.75);
  \node[dots] at (-0.45,1.375) {$\cdots$};
  \node[dots] at (0.30,1.375) {$\cdots$};
  \node[dots] at (1.30,1.375) {$\cdots$};
  \node[dots] at (-0.45,2.325) {$\cdots$};
  \node[dots] at (0.30,2.325) {$\cdots$};
  \node[dots] at (1.30,2.325) {$\cdots$};
  \node[dots] at (3.15,2.475) {$\cdots$};
  \draw[usual] (0.30,2.50) -- (2.70,2.50) -- (2.20,3.15) -- (0.80,3.15) -- cycle;
  \node[lab] at (1.50,2.825) {$\bm{f}$};
  \draw[usual] (1.10,3.15) -- (1.10,3.75);
  \draw[usual] (1.70,3.15) -- (1.70,3.75);
  \node[dots] at (1.40,3.37) {$\cdots$};
  \draw[usual] (2.90,3.15) -- (2.90,3.75);
  \draw[usual] (3.40,3.15) -- (3.40,3.75);
\end{scope}
\node[lab] at (5.10,1.775) {$=$};
\node[lab] at (5.85,1.775) {$\displaystyle\sum c_i$};
\begin{scope}[shift={(7.25,0.675)}]
  \draw[usual] (0.00,2.10) -- (1.80,2.10) -- (1.50,1.45) -- (0.30,1.45) -- cycle;
  \draw[usual] (0.30,0.78) rectangle (1.50,1.45);
  \draw[usual] (0.30,0.78) -- (1.50,0.78) -- (1.80,0.10) -- (0.00,0.10) -- cycle;
  \node[lab] at (0.90,1.78) {$\bar{\bm{f_i}}$};
  \node[lab] at (0.90,1.115) {$\bm{P_i}$};
  \node[lab] at (0.90,0.43) {$\bm{g_i}$};
\end{scope}
\end{tikzpicture}.
\end{gather}
\begin{gather}
\tag{HI3}\label{eq:HI3}
\begin{tikzpicture}[
baseline=(current bounding box.center),
x=1cm,y=1cm,
line width=0.8pt,
line cap=butt,
line join=miter,
lab/.style={inner sep=0pt,font=\large},
rhslab/.style={inner sep=0pt,font=\normalsize},
dots/.style={inner sep=0pt,font=\normalsize}
]
\def\W{2.10}
\def\w{1.40}
\def\H{0.58}
\def\xoff{0.42}
\def\pole{0.42}
\def\yf{1.58}
\def\yg{0.28}
\pgfmathsetmacro{\fbot}{\yf-\H/2}
\pgfmathsetmacro{\ftop}{\yf+\H/2}
\pgfmathsetmacro{\gbot}{\yg-\H/2}
\pgfmathsetmacro{\gtop}{\yg+\H/2}
\pgfmathsetmacro{\yeq}{(\yf+\yg)/2}
\def\RW{1.72}
\def\rw{1.28}
\def\RH{0.48}
\pgfmathsetmacro{\yrhs}{\yeq-1.5*\RH}
\draw[usual] ({-\W/2},\fbot) -- ({\W/2},\fbot)
-- ({\w/2},\ftop) -- ({-\w/2},\ftop) -- cycle;
\draw[usual] ({-\W/2},\gbot) rectangle ({\W/2},\gtop);
\draw[usual] (-\xoff,\gtop) -- (-\xoff,\fbot);
\draw[usual] ( \xoff,\gtop) -- ( \xoff,\fbot);
\node[dots] at (0,\yeq) {$\cdots$};
\draw[usual] (-\xoff,\ftop) -- (-\xoff,{\ftop+\pole});
\draw[usual] ( \xoff,\ftop) -- ( \xoff,{\ftop+\pole});
\node[dots] at (0,{\ftop+0.17}) {$\cdots$};
\draw[usual] (-\xoff,\gbot) -- (-\xoff,{\gbot-\pole});
\draw[usual] ( \xoff,\gbot) -- ( \xoff,{\gbot-\pole});
\node[dots] at (0,{\gbot-0.17}) {$\cdots$};
\node[lab] at (0,\yf) {$\bm{f}$};
\node[lab] at (0,\yg) {$\bm{P}$};
\node[lab] at (2.35,\yeq) {$=$};
\node[lab] at (3.20,\yeq) {$\displaystyle\sum c_i$};
\begin{scope}[shift={(4.10,\yrhs)}]
\draw[usual] ({(\RW-\rw)/2},{2*\RH}) -- ({(\RW+\rw)/2},{2*\RH})
-- (\RW,{3*\RH}) -- (0,{3*\RH}) -- cycle;
\draw[usual] ({(\RW-\rw)/2},\RH) rectangle ({(\RW+\rw)/2},{2*\RH});
\draw[usual] (0,0) -- (\RW,0)
-- ({(\RW+\rw)/2},\RH) -- ({(\RW-\rw)/2},\RH) -- cycle;
\node[rhslab] at ({\RW/2},{2.5*\RH}) {$\overline{\bm{f_i}}$};
\node[rhslab] at ({\RW/2},{1.5*\RH}) {$\bm{P_i}$};
\node[rhslab] at ({\RW/2},{0.5*\RH}) {$\bm{g_i}$};
\end{scope}
\end{tikzpicture}.
\end{gather}
Here \(P_i\) and \(P\) are permutations, possibly identities, \(f\) and \(g\)
are atoms, and \(f_i\) and \(g_i\) are distinguished bottom diagrams. The
left-hand side of \eqref{eq:HI1} represents an arbitrary interaction of an
atom \(f\) with the flip of an atom \(g\), in a composite of the form
\((\id\otimes f\otimes\id)\circ(\id\otimes\overline{g}\otimes\id)\). Thus one
atom may overhang the other on either side and we allow the horizontal
reflection of the displayed diagram. 

The left hand side of \eqref{eq:HI2} is the analogue of that of \eqref{eq:HI1}, but where we allow a permutation $P$ to be sandwiched in the middle: we assume that every strand in the domain (resp. codomain) of $P$ that is to the left of the domain of $\overline{g}$ (resp. the codomain of $f$) is connected to a strand in the domain of $f$ (resp. codomain of $\overline{g}$), and that moreover $P$ does not involve the rightmost strand in the domain of $f$ or the codomain of $\overline{g}$. Similar to in \eqref{eq:HI1}, we allow overhang by either atom on either side, though $P$ is always inserted from the left and does not reach the rightmost strand of $f$ or $\overline{g}$.

Finally, in
\eqref{eq:HI3}, we assume \(P\) is an arbitrary permutation that \textit{involves the last strand}. (An atom precomposed by a permutation missing the last strand is still a distinguished bottom diagram.)

By \autoref{bmt basis}, all such relations hold in \(\Fund(\mathcal{C})\), and
their coefficients can be obtained by direct computation. Their role is to
move flipped atoms past atoms, so that diagrams can be rewritten in sandwich
(bottom-permutation-top) form.

The second kind are the \emph{orientation-reversing relations}:
\begin{gather}
\tag{OR1}\label{eq:OR1}
\begin{tikzpicture}[
baseline=(current bounding box.center),
x=1cm,y=1cm,
line width=0.8pt,
line cap=butt,
line join=miter,
lab/.style={inner sep=0pt,font=\large},
dots/.style={inner sep=0pt,font=\normalsize}
]
\begin{scope}[shift={(0,0)}]
\draw[usual] (0.00,0.55) -- (3.30,0.55) -- (2.70,1.22) -- (0.60,1.22) -- cycle;
\node[lab] at (1.65,0.885) {$\bm{f}$};
\draw[usual] (1.30,1.22) -- (1.30,2.02);
\draw[usual] (2.00,1.22) -- (2.00,2.02);
\node[dots] at (1.65,1.62) {$\cdots$};
\draw[usual] (1.22,-0.20) -- (1.22,0.55);
\draw[usual] (2.08,-0.20) -- (2.08,0.55);
\node[dots] at (1.65,0.175) {$\cdots$};
\draw[usual] (2.85,0.55) -- (4.25,-0.20);
\draw[usual] (2.85,-0.20) -- (4.25,0.55) -- (4.25,2.02);
\end{scope}
\node[lab] at (5.20,0.95) {$=$};
\node[lab] at (5.95,0.95) {$\displaystyle\sum c_i$};
\begin{scope}[shift={(7.30,-0.15)}]
\draw[usual] (0.00,2.10) -- (1.80,2.10) -- (1.50,1.45) -- (0.30,1.45) -- cycle;
\draw[usual] (0.30,0.78) rectangle (1.50,1.45);
\draw[usual] (0.30,0.78) -- (1.50,0.78) -- (1.80,0.10) -- (0.00,0.10) -- cycle;
\node[lab] at (0.90,1.78) {$\bar{\bm{f_i}}$};
\node[lab] at (0.90,1.115) {$\bm{P_i}$};
\node[lab] at (0.90,0.43) {$\bm{g_i}$};
\end{scope}
\end{tikzpicture}.
\end{gather}
\begin{gather}
\tag{OR2}\label{eq:OR2}
\begin{tikzpicture}[
  baseline=(current bounding box.center),
  x=1cm,y=1cm,
  line width=0.8pt,
  line cap=butt,
  line join=miter,
  lab/.style={inner sep=0pt,font=\large},
  dots/.style={inner sep=0pt,font=\normalsize}
]
\begin{scope}[shift={(0,0)}]
  \draw[usual] (-0.70,-0.20) -- (-0.70,1.55);
  \draw[usual] (-0.70,2.15) -- (-0.70,3.75);
  \draw[usual] (-0.20,-0.20) -- (-0.20,1.55);
  \draw[usual] (-0.20,2.15) -- (-0.20,3.75);
  \draw[usual,fill=white] (-0.90,1.55) rectangle (2.00,2.15);
  \node at (0.55,1.85) {$\bm{P}$};
  \draw[usual] (0.80,-0.20) -- (0.80,1.55);
  \draw[usual] (0.80,2.15) -- (0.80,2.50);
  \draw[usual] (1.90,-0.20) -- (1.90,0.55);
  \draw[usual] (3.30,-0.20) -- (3.30,0.55);
  \node[dots] at (2.60,0.175) {$\cdots$};
  \draw[usual] (1.10,0.55) -- (4.10,0.55) -- (3.50,1.20) -- (1.70,1.20) -- cycle;
  \node[lab] at (2.60,0.875) {$\bm{g}$};
  \draw[usual] (1.80,1.20) -- (1.80,1.55);
  \draw[usual] (1.80,2.15) -- (1.80,2.50);
  \draw[usual] (2.40,1.20) -- (2.40,2.50);
  \draw[usual] (2.90,1.20) -- (2.90,3.75);
  \draw[usual] (3.40,1.20) -- (3.40,3.75);
  \node[dots] at (-0.45,1.375) {$\cdots$};
  \node[dots] at (0.30,1.375) {$\cdots$};
  \node[dots] at (1.30,1.375) {$\cdots$};
  \node[dots] at (-0.45,2.325) {$\cdots$};
  \node[dots] at (0.30,2.325) {$\cdots$};
  \node[dots] at (1.30,2.325) {$\cdots$};
  \node[dots] at (3.15,2.475) {$\cdots$};
  \draw[usual] (0.30,2.50) -- (2.70,2.50) -- (2.20,3.15) -- (0.80,3.15) -- cycle;
  \node[lab] at (1.50,2.825) {$\bm{f}$};
  \draw[usual] (1.10,3.15) -- (1.10,3.75);
  \draw[usual] (1.70,3.15) -- (1.70,3.75);
  \node[dots] at (1.40,3.37) {$\cdots$};
  \draw[usual] (2.90,3.15) -- (2.90,3.75);
  \draw[usual] (3.40,3.15) -- (3.40,3.75);
\end{scope}
\node[lab] at (5.10,1.775) {$=$};
\node[lab] at (5.85,1.70) {$\displaystyle\sum_i c_i$};
\begin{scope}[shift={(7.25,0.675)}]
  \draw[usual] (0.00,2.10) -- (1.80,2.10) -- (1.50,1.45) -- (0.30,1.45) -- cycle;
  \draw[usual] (0.30,0.78) rectangle (1.50,1.45);
  \draw[usual] (0.30,0.78) -- (1.50,0.78) -- (1.80,0.10) -- (0.00,0.10) -- cycle;
  \node[lab] at (0.90,1.78) {$\bar{\bm{f_i}}$};
  \node[lab] at (0.90,1.115) {$\bm{P_i}$};
  \node[lab] at (0.90,0.43) {$\bm{g_i}$};
\end{scope}
\end{tikzpicture}.
\end{gather}

\begin{gather}
\tag{OR3}\label{eq:OR3}
\begin{tikzpicture}[
baseline=(current bounding box.center),
x=1cm,y=1cm,
line width=0.8pt,
line cap=butt,
line join=miter,
lab/.style={inner sep=0pt,font=\large},
dots/.style={inner sep=0pt,font=\normalsize}
]
\begin{scope}[shift={(0,0)}]
\draw[usual] (1.15,2.45) -- (4.15,2.45) -- (3.55,3.15) -- (1.75,3.15) -- cycle;
\node[lab] at (2.65,2.80) {$\bm{f}$};
\draw[usual] (2.25,3.15) -- (2.25,3.95);
\draw[usual] (3.05,3.15) -- (3.05,3.95);
\node[dots] at (2.65,3.55) {$\cdots$};
\draw[usual] (0.70,3.95) -- (0.70,1.55)
-- (2.00,0.85) -- (2.00,0.55)
-- (3.30,-0.15) -- (3.30,-0.65);
\draw[usual] (2.00,2.45) -- (2.00,1.55)
-- (0.70,0.85) -- (0.70,-0.65);
\node[dots] at (2.65,2.00) {$\cdots$};
\draw[usual] (3.30,2.45) -- (3.30,0.55)
-- (2.00,-0.15) -- (2.00,-0.65);
\end{scope}
\node[lab] at (4.95,1.55) {$=$};
\node[lab] at (5.70,1.55) {$\displaystyle\sum c_i$};
\begin{scope}[shift={(7.05,0.45)}]
\draw[usual] (0.00,2.10) -- (1.80,2.10) -- (1.50,1.45) -- (0.30,1.45) -- cycle;
\draw[usual] (0.30,0.78) rectangle (1.50,1.45);
\draw[usual] (0.30,0.78) -- (1.50,0.78) -- (1.80,0.10) -- (0.00,0.10) -- cycle;
\node[lab] at (0.90,1.78) {$\bar{\bm{f_i}}$};
\node[lab] at (0.90,1.115) {$\bm{P_i}$};
\node[lab] at (0.90,0.43) {$\bm{g_i}$};
\end{scope}
\end{tikzpicture}.
\end{gather}
\begin{gather}
\tag{OR4}\label{eq:OR4}
\begin{tikzpicture}[
baseline=(current bounding box.center),
x=1cm,y=1cm,
lab/.style={inner sep=0pt,font=\large},
boxlab/.style={inner sep=1.2pt,fill=white,font=\Large},
box/.style={usual}
]
\begin{scope}[shift={(0,0)}]
\draw[usual] (0,0) node[below] {$i$} -- (1.5,2);
\draw[usual] (1.5,0) node[below] {$k$} -- (0,2);
\draw[usual] (0.75,0) node[below] {$j$}
.. controls (0.75,0.4) and (1.8,0.6) .. (1.8,1)
.. controls (1.8,1.4) and (0.75,1.6) .. (0.75,2);
\end{scope}
\node[lab] at (3.10,1.00) {$=$};
\node[lab] at (3.85,1.00) {$\displaystyle\sum c_i$};
\begin{scope}[shift={(5.20,-0.05)}]
\draw[box] (0.00,2.10) -- (1.80,2.10) -- (1.50,1.45) -- (0.30,1.45) -- cycle;
\draw[box] (0.30,0.78) rectangle (1.50,1.45);
\draw[box] (0.30,0.78) -- (1.50,0.78) -- (1.80,0.10) -- (0.00,0.10) -- cycle;
\node[boxlab] at (0.90,1.78) {$\bm{\bar f_i}$};
\node[boxlab] at (0.90,1.115) {$\bm{P_i}$};
\node[boxlab] at (0.90,0.43) {$\bm{g_i}$};
\end{scope}
\end{tikzpicture}.
\end{gather}

\begin{Remark}
The term `orientation-reversing' reflects the idea of flipping diagrams
over. The terminology \(H=I\) will be justified in \autoref{S:Examples}.
\end{Remark}

Again, \(f\) and \(g\) are arbitrary atoms, $P$ is a permutation (possibly identity) and the coefficients are determined
by direct computation in \(\Fund(\mathcal{C})\). The left hand side of \eqref{eq:OR2} is the analogue of that of \eqref{eq:HI2}: here every strand in the domain (resp. codomain) of $P$ to the left of that of $g$ (resp. $f$) is joined to a strand in the domain of $f$ (resp. codomain of $g$), and $P$ does not involve the rightmost strand of $f$. Here $f$ is not allowed to overhang $g$ on the right. By construction, none of
the four left-hand sides above can occur as a subdiagram of a distinguished
bottom diagram.

\begin{Lemma}\label{reduced diagrams are distinguished}
Let $D$ be a diagram generated by $\mathcal{A}\cup \mathcal X$. If $D$ contains none of the left-hand sides of \eqref{eq:OR1}, \eqref{eq:OR2}, 
\eqref{eq:OR3} and \eqref{eq:OR4} as a subdiagram, then $D$ is a distinguished bottom diagram, up to postcomposition by a permutation.

Dually, let $D$ be generated by $\overline{\mathcal{A}}\cup C$. If $D$ contains none of the vertical flips of these diagrams as a subdiagram, then $D$ is a distinguished top diagram, up to precomposition by a permutation.
\end{Lemma}

\begin{proof}
It is enough to prove the first statement. We induct on the number of strands in the domain. Since \eqref{eq:OR4} is excluded, every diagram made only of crossings is one of our chosen permutations.

Remove the last strand of $D$, together with all atoms and crossings attached to it, and call the resulting diagram $D_0$. By induction, $D_0$ is a distinguished bottom diagram up to a final permutation. If the last strand meets $D_0$ only by crossings, this final permutation just changes, and there is nothing to prove.

Otherwise the last strand meets $D_0$ through an atom. The exclusions \eqref{eq:OR1}, \eqref{eq:OR2},
and \eqref{eq:OR3} force this atom to be attached to the final strands in the codomain of $D_0$, exactly as in the recursive construction of distinguished bottom diagrams. Hence $D$ is again a distinguished bottom diagram up to postcomposition by a permutation.
\end{proof}

Let $\mathcal{F}$ be the $\kk$-linear monoidal category generated by the objects $B_1,\dots,B_n$ and the morphisms in $S$, modulo the relations $\beta_{i,j}\circ\beta_{j,i}=\id$, the $H=I$ relations \eqref{eq:HI1}, \eqref{eq:HI2}, \eqref{eq:HI3}  
the orientation-reversing relations \eqref{eq:OR1}, \eqref{eq:OR2}, 
\eqref{eq:OR3}, \eqref{eq:OR4}, and the vertical reflections of all these relations. Note that some of the relations include the horizontal reflections of the pictures, see the discussions in the paragraphs below them for details.

\begin{Theorem}\label{iso of categories}
The natural $\kk$-linear monoidal functor $\mathcal{F}\to \Fund(\mathcal{C})$ is an isomorphism. Consequently, after Cauchy completion, this gives a presentation of $\gcrys$, which is semisimple with simple objects $B(\lambda)$, $\lambda\in\Lambda_+$.
\end{Theorem}

\begin{proof}
The functor sends $B_i$ to the corresponding crystal and sends each generator in $S$ to the corresponding morphism in $\Fund(\mathcal{C})$. All defining relations of $\mathcal{F}$ hold in $\Fund(\mathcal{C})$ by construction. It remains to compare Hom-spaces.

The diagrams in $\mathcal{B}(\epsilon,\epsilon')$, viewed as morphisms in $\mathcal{F}$, are linearly independent, since their images form part of the basis from \autoref{bmt basis}. Thus it remains to prove that they span.

Let $h$ be a diagram in $\mathcal{F}$. Using the orientation-reversing relations and their vertical reflections, we may assume that $h$ contains none of the left-hand sides of \eqref{eq:OR1}, \eqref{eq:OR2},
\eqref{eq:OR3}, \eqref{eq:OR4}, nor their vertical reflections, as a subdiagram. By \autoref{reduced diagrams are distinguished}, the remaining pieces are distinguished bottom diagrams, distinguished top diagrams, permutations, or local configurations of the kind appearing in the left-hand sides of \eqref{eq:HI1}, \eqref{eq:HI2}, \eqref{eq:HI3}, up to padding with identities (the bottommost layer of a distinguished bottom diagram is always a padded atom precomposed with some permutation missing the last strand of the atom part).

The $H=I$ relations move flipped atoms past atoms. Repeating this finitely many times, $h$ becomes a sum of diagrams of the form $\overline{f}\circ P\circ g$, where $f$ and $g$ are distinguished bottom diagrams and $P$ is a permutation. These are precisely the diagrams in $\mathcal{B}(\epsilon,\epsilon')$. Hence they span, and the functor is an isomorphism.

The final statement follows from the standing identification $\Fund(\mathcal C)^c\cong_\otimes\mathcal C=\gcrys$ and the definition of $\gcrys$ as the semisimple crystal category with simple objects $B(\lambda)$, $\lambda\in\Lambda_+$.
\end{proof}

\begin{Example}\label{E:SL2}
In the case of $\g=\mathfrak{sl}_2$, $\Fund(\mathcal{C})$ is generated by a single object $B_1$, three morphisms which are $\id_{B_1}$, the cup, and the cap. There are no orientation-reversing relations, and the only $H=I$ relations are the `zigzag relation' and circle evaluation, both of type \eqref{eq:HI1}:
\begin{equation*}
\begin{tikzpicture}[
baseline=(current bounding box.center),
x=.48cm,y=.48cm,
line width=.85pt,
line cap=round,
line join=round
]

\begin{scope}[shift={(0,0)}]
\draw[usual]
(0,1.6) -- (0,-.35)
to[out=-90,in=180] (.5,-.75)
to[out=0,in=-90] (1,-.35)
-- (1,.35)
to[out=90,in=180] (1.5,.75)
to[out=0,in=90] (2,.35)
-- (2,-1.6);
\end{scope}

\node at (3,0) {$=$};

\node at (4,0) {$0$};

\node at (5,0) {$=$};

\begin{scope}[shift={(6,0)}]
\draw[usual]
(0,-1.6) -- (0,.35)
to[out=90,in=180] (.5,.75)
to[out=0,in=90] (1,.35)
-- (1,-.35)
to[out=-90,in=180] (1.5,-.75)
to[out=0,in=-90] (2,-.35)
-- (2,1.6);
\end{scope}

\end{tikzpicture}.
\end{equation*}
\begin{equation*}
\begin{tikzpicture}[baseline=(current bounding box.center)]
\draw[usual] (0,0) circle [radius=0.5];
\node at (1.15,0) {$=1$};
\end{tikzpicture}.
\end{equation*}
This recovers \cite[Theorem 4.7]{Alqady2025coboundary}.
\end{Example}

\subsection{A smaller set of generating morphisms}\label{subsec:smaller-generators}

We saw in \autoref{G2 atom dependence} that there are relations between atoms. In fact, the example in \autoref{G2 atom dependence} generalizes and every atom can be obtained from \emph{atoms on two strands}, i.e. atoms whose domain is of the form $B_i\otimes B_j$. Let $\mathcal{A}^*$ denote the set of all two-strand atoms with domain $B_i\otimes B_j, i\le j$. Denote $\mathcal{A}^*\cup \overline{\mathcal{A}^*}\cup \mathcal X$ by $S^*$.

\begin{Lemma} \label{two strand highest element switch lemma} Let $B,C\in \CC$, where $u_B, u_C$ denote the highest elements of $B,C$, respectively. Let $\beta: B\otimes C\to C\otimes B$ be any isomorphism. Then if $u_B\otimes c$ lies in the top component of $B\otimes C$, its image under $\beta$ is $c\otimes u_B$.
\end{Lemma}

\begin{proof}
We have $u_B\otimes c=f_{i_r}\dots f_{i_{1}}(u_B\otimes u_C)$ for a sequence of lowering operators $f_{i_j}$ which all act \emph{on the second factor}. Thus $\beta(u_B\otimes c) = f_{i_r}\dots f_{i_{1}}(u_C\otimes u_B)$, where the lowering operators now all must act on the first factor since $u_B$ is highest. This gives the result.  
\end{proof}

\begin{Lemma} \label{pairwise jw lemma} Suppose the labels of $B_{\epsilon}$ are weakly increasing, $B_{\epsilon}=B_{a_1}^{\otimes r_1}\otimes\dots\otimes B_{a_s}^{\otimes r_s}, a_1<\dots <a_s$. Let $B_{a_r}^*$ denote the leftmost factor of $B_{a_r}$, say at position $r^*$. For $i<r^*$, let $\alpha_{i,r}$ be the composite of the sequence of crossings which moves $B_{a_r}^*$ leftwards until it is immediately to the right of $B_{\epsilon_i}$. Let $j_{i,r}:= \id\otimes j_{(\epsilon_i,a_r)}\otimes \id$ be the Jones--Wenzl projector applied to the now adjacent copy of $B_{\epsilon_i}\otimes B_{a_r}$ padded by identity, and define $\gamma_{i,r}=\alpha_{i,r}^{-1}\circ j_{i,r}\circ \alpha_{i,r}$. If $\epsilon_i=\epsilon_{i-1}$, also define $j_{i-1,i}$ to be the two-strand Jones--Wenzl projector applied to $B_{\epsilon_{i-1}}\otimes B_{\epsilon_i}$, padded by identity. 

Then, we have that $b$ lies in the top component of $B_{\epsilon}$ if and only if $\gamma_{i,r}(b)\neq 0$ for $ 1<r\le s, i<r^*$ and $j_{i-1,i}(b)\neq 0$ for all $i$ for which $j_{i-1,i}$ is defined.
\end{Lemma}

The projector $\gamma_{i,r}$ is illustrated below:
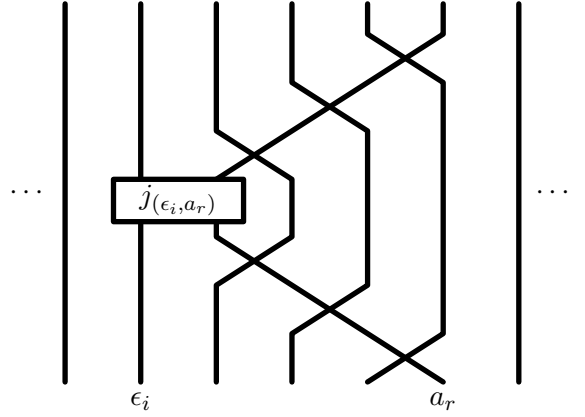
\[
\begin{tikzpicture}[
thick,
scale=0.8,
baseline=(current bounding box.center),
line cap=round,
line join=round
]
\def\xL{-1.25}
\def\xi{0}
\def\xA{1.25}
\def\xB{2.50}
\def\xC{3.75}
\def\xr{5.00}
\def\xR{6.25}

\def\yb{0}
\def\yone{0.80}
\def\ytwo{1.60}
\def\ythree{2.40}
\def\yboxb{2.65}
\def\yboxt{3.35}
\def\yfour{4.15}
\def\yfive{4.95}
\def\ysix{5.75}
\def\ytop{6.25}

\node at (-1.85,3.125) {$\cdots$};
\node at (6.85,3.125) {$\cdots$};

\draw[usual] (\xL,\yb) -- (\xL,\ytop);

\draw[usual] (\xi,\yb) -- (\xi,\yboxb);
\draw[usual] (\xi,\yboxt) -- (\xi,\ytop);

\draw[usual]
(\xA,\yb) -- (\xA,\ytwo) -- (\xB,\ythree)
-- (\xB,\yboxt) -- (\xA,\yfour) -- (\xA,\ytop);

\draw[usual]
(\xB,\yb) -- (\xB,\yone) -- (\xC,\ytwo)
-- (\xC,\yfour) -- (\xB,\yfive) -- (\xB,\ytop);

\draw[usual]
(\xC,\yb) -- (\xr,\yone)
-- (\xr,\yfive) -- (\xC,\ysix) -- (\xC,\ytop);

\draw[usual]
(\xr,\yb) -- (\xC,\yone) -- (\xB,\ytwo) -- (\xA,\ythree)
-- (\xA,\yboxb);

\draw[usual]
(\xA,\yboxt) -- (\xB,\yfour) -- (\xC,\yfive)
-- (\xr,\ysix) -- (\xr,\ytop);

\draw[usual] (\xR,\yb) -- (\xR,\ytop);

\draw[usual, fill=white] (-0.45,\yboxb) rectangle (1.70,\yboxt);
\node at (0.625,3.00) {$j_{(\epsilon_i,a_r)}$};

\node[below] at (\xi,\yb) {$\epsilon_i$};
\node[below] at (\xr,\yb) {$a_r$};
\end{tikzpicture}
.
\]

\begin{proof}
Since all projectors under consideration are crystal morphisms and act by
identity or zero on connected components, it is enough to check the claim on
highest weight elements of connected components.

If $b$ is in the top component, it certainly survives all projectors as if we write $\Top(X)$ for the top component of $X$, we have $\Top(X\otimes Y)= \Top(\Top(X)\otimes \Top(Y))$. It thus suffices to show that if $b=b_1\otimes\dots b_m$ is a highest weight element such that $\gamma_{i,r}(b)\neq 0, j_{i-1,i}(b)\neq 0$ for all $r,i$, then it is the highest element $u_1\otimes\dots \otimes u_m$ of $B_{\epsilon}$. Clearly $b_1=u_1$, as every prefix of highest weight element is highest weight. Inductively assume that we have $b_1=u_1,\dots, b_{i-1}=u_{i-1}$. We have two cases. If $i$ lies in the same constant label block as $i-1$, say $\epsilon_{i-1}=\epsilon_{i}=k$ then the pair $u_{i-1}\otimes b_i$ lies in the top component of a copy of $B_k\otimes B_k$. However, if we write $u$ for the highest element of $B_k$, we have $B(2\omega_k)\cap (u\otimes B_k)=\{u\otimes u\}$, since lowering operators applied to $u\otimes u$ can only act on the first factor. This shows $b_i=u_i$. It remains to consider the case where $i$ belongs to a new block. In this case, apply crossings to move $b_i$ until it is in the second position, so the word is 
\[u_1\otimes b_i\otimes u_2\otimes \dots \otimes u_{i-1}\otimes \dots \]
Here we have repeatedly used \autoref{two strand highest element switch lemma}, which is permitted because by assumption $u_j\otimes b_i, j<i$ always lie in the top component.
Now $u_1\otimes b_i$ is a prefix of a highest weight element so it is of highest weight, and by assumption it lies in the top component of $B_{\epsilon_1}\otimes B_{\epsilon_{i}}$. Thus, we must have $b_i=u_i$. This completes the induction.
\end{proof}

\begin{Proposition}\label{efficient pres prop}
The category $\Fund(\mathcal{C})$ is generated by the objects $B_1,\dots, B_n$ and the morphisms $S^*$. 
\end{Proposition}

\begin{proof}
It is enough to prove that any atom can be generated by $\mathcal{A}^*\cup \mathcal X$. Let $f$ be an atom with domain $B_\epsilon$, which by precomposing with crossings we may assume has labels weakly increasing. 
We may assume $f$ is non-identity, so that it kills the top summand of $B_{\epsilon}$. Let $h_c$ be the highest weight element of the image of $f$ (which is unique by \autoref{bijection branching}). 
We have $j_\epsilon (h_c)=0$, so by \autoref{pairwise jw lemma}, we either have $\gamma_{i,r}(h_c)=0$ for some $r$, or $j_{i-1,i}(h_c)=0$ for some $i$. In either case, applying crossings if necessary, we obtain that there is some adjacent pair of factors $B_{\epsilon_{i-1}}\otimes B_{\epsilon_i}$ such that applying the identity-padded Jones--Wenzl projector on these two factors kills $h_c$.
It follows that $f$ can be written in the form $\sum_{\alpha} g_\alpha \circ (\id\otimes \alpha\otimes \id)$, where $\alpha$ is an atom on the  two strands $i-1,i$, and $g_{\alpha}$ has domain of weight strictly less than that of $B_{\epsilon}$. After using crossings to make the domain of $g_{\alpha}$ have weakly increasing labels, we may iterate the process above, and
after finitely many steps we obtain $f$ from two-strand atoms and crossings. Finally, observe that any atom with domain $B_i\otimes B_j, i>j$ is an atom with domain $B_j\otimes B_i$ precomposed with $\beta_{i,j}$.     
\end{proof}

\begin{Lemma}\label{lem:type-a-two-strand-atom-count}
In type \(A_{n-1}\), the number of two-strand atoms is
\(\binom{n+1}{3}=n(n^2-1)/6\). If we only keep those whose domain labels are
weakly increasing, this number becomes
\[
	\left\lfloor \tfrac{n(2n^2+3n-2)}{24}\right\rfloor .
\]
\end{Lemma}

\begin{proof}
We use the usual model in which the elements of \(B_j\) are \(j\)-element
subsets \(S\subset\{1,\dots,n\}\). Such an element has weight
\[
	\wt(S)=
	\sum_{\substack{1\le r\le n-1\\ r\in S,\ r+1\notin S}}\omega_r
	-
	\sum_{\substack{1\le r\le n-1\\ r\notin S,\ r+1\in S}}\omega_r .
\]
The atom \(f_S\) is two-strand precisely when the minimal dominant weight
\(\lambda\) such that \(\lambda+\wt(S)\) is dominant is a single fundamental
weight. By the displayed formula, this is equivalent to having exactly one
index \(r\) with \(r\notin S\) and \(r+1\in S\).

Thus \(S\) is an interval with one gap:
\(S=\{1,\dots,a\}\cup\{r+1,\dots,c\}\), where
\(0\le a<r\le c\le n\) where $r$ is the unique index with $r\in S, r+1\neq S$. Hence choosing \(S\) is the same as choosing the triple
\(a,b,c\) from the \(n+1\) possible values \(0,\dots,n\), giving
\(\binom{n+1}{3}\). The weakly increasing count is obtained by imposing the
additional condition on the two domain labels, which gives
\(\left\lfloor n(2n^2+3n-2)/24\right\rfloor\).
\end{proof}

\begin{Example}
Thus \autoref{efficient pres prop} replaces an exponential number of atoms by
a cubic number of two-strand atoms. For comparison, in type \(A_{n-1}\) the
total number of atoms is \(2^n-2\), and the number of non-identity atoms is
\(2^n-n-1\). The first few values are:
\[
\begin{array}{c|c|c|c}
& \text{non-identity atoms}
& \text{all two-strand atoms}
& \text{weakly increasing two-strand atoms}
\\ \hline
A_2 & 4   & 4  & 3  \\
A_3 & 11  & 10 & 7  \\
A_4 & 26  & 20 & 13 \\
A_5 & 57  & 35 & 22 \\
A_6 & 120 & 56 & 34 \\
A_7 & 247 & 84 & 50
\end{array},
\]
as one easily checks.
\end{Example}

\subsection{Weak Rigidity and Tensor Ideals}

Like the $\mathfrak{sl}_2$ case, the categories $\mathfrak g$-Crys are not rigid. Indeed, one easily shows this by mimicking the proof of \cite[Proposition 3.23]{Alqady2025coboundary}: first, one shows that if $B(\lambda)$ had a dual, say for $\lambda=\omega_1$, it would have to be $B(-w_0 \lambda)$, where $w_0$ is the longest element of the Weyl group of $\mathfrak g$; then, one shows that an evaluation and a coevaluation for $B(\lambda)$ satisfying the zig-zag relation cannot possibly exist by simply examining the corresponding spaces of morphisms.

\begin{Remark}
By the main result	 of \cite{EtiPen2026}, the non-rigidity of $\mathfrak g$-Crys also proves that these categories do not admit any braidings, which was also proved directly in \cite{Kamnitzcrystal}.
\end{Remark}

However, the categories $\mathfrak g$-Crys are weakly rigid in the following sense of Bakalov and Kirillov Jr., which allows us to generalize Proposition 4.3 of \cite{Alqady2025coboundary} on the nonexistence of nontrivial tensor ideals.

\begin{Definition}\cite[Section 5.3]{BakKir2001}
	An object $X$ in a monoidal category is said to be \emph{weakly rigid} if the functor $\Hom(\Id, X\otimes -)$ is representable; its representing object $X^\vee$ is called the \emph{weak dual} of $X$. Thus, we have a bijection $\Hom(\Id,X\otimes Y)\cong \Hom(X^\vee, Y)$	natural in $Y$ whenever $X^\vee$ exists.
\end{Definition}

\begin{Remark}
	Up to taking the opposite monoidal structure, this notion of weak rigidity coincides with the notion of an $r$-category from \cite{BoyDr2013}. 
\end{Remark}

Given a weakly rigid object $X\noindent$, we get a canonical morphism $\eta_X:\Id\to X\otimes X^\vee$ corresponding to $\id_{X^\vee}$.
Suppose $X$ and $Y$ are weakly rigid objects in $\mathcal C$, and $f:X\to Y$ is a morphism. Then, we get a transformation
\[\begin{tikzcd}
	{\Hom(X^\vee, Z)} & {\Hom(\Id, X\otimes Z)} && {\Hom(\Id, Y\otimes Z)} & {\Hom(Y^\vee, Z)}
	\arrow["\sim", from=1-1, to=1-2]
	\arrow["{(f\otimes \id_Z)\circ -}", from=1-2, to=1-4]
	\arrow["\sim", from=1-4, to=1-5]
\end{tikzcd}\]
natural in $Z$. 
By Yoneda's Lemma, this corresponds to morphism $f^\vee:Y^\vee\to X^\vee$. 
Equivalently, the morphism $f^\vee\in\Hom(Y^\vee,X^\vee)$ can be defined as the canonical morphism corresponding to the map $(f\otimes \id_{X^\vee})\circ \eta_X\in\Hom(\Id, Y\otimes X^\vee)$.

\begin{Definition}\cite[Definition 5.3.4]{BakKir2001}
	A monoidal category $\mathcal C$ is called weakly rigid if every object in $\mathcal C$ is weakly rigid and the functor $-^\vee:\mathcal C\to\mathcal C^\mathrm{op}$ is an equivalence of categories.
\end{Definition}

\begin{Proposition}\label{no-tensor-ideals}
	A semisimple weakly rigid monoidal category with a simple monoidal unit contains no non-trivial tensor ideals.
\end{Proposition}

\begin{proof}
	Suppose $\mathcal I$ is a tensor ideal containing some nonzero morphism $f$ in such a category $\mathcal C$.
	Since $\mathcal C$ is semisimple, by composing with projections and inclusions, we may assume that $f$ is a morphism between simple objects. 
	By Schur's Lemma, we may further assume that $f=\id_S$ for some simple object $S$. 
	Clearly, the canonical morphism $\eta_S:\Id\to S\otimes S^\vee$ is nonzero.
	Since $\Id$ is assumed to be simple, $\eta_S$ is a monomorphism, so that $\Id$ is a direct summand of $S\otimes S^\vee$ by semisimplicity; therefore, $\id_\Id$ factors through $\eta_S$. 
	But now, $\id_{S\otimes S^\vee}= \id_S\otimes \id_{S^\vee}\in \mathcal I$, so $\eta_S\in\mathcal I$, and hence $\id_\Id\in\mathcal I$. 
	Therefore, $\mathcal I= \mathcal C$.
\end{proof}

\begin{Proposition} \label{weak-rigidity-of-g-crys}
	The category $\mathfrak g$-Crys is weakly rigid but not rigid.
\end{Proposition}

\begin{proof}
	For each $\lambda\in\Lambda_+$, let $B(\lambda)^\vee := B(-w_0\lambda)$, where $w_0$ is the longest element of the Weyl group of $\mathfrak g$. 
	Then, $\Id \cong B(0)$ is a summand of $B(\lambda)\otimes B(\lambda)^\vee$ of multiplicity $1$, so that we have a unique up to scalar morphism $\eta_\lambda:\Id\to B(\lambda)\otimes B(\lambda)^\vee$. 
	More generally, for $X\cong \bigoplus_{i=1}^n B(\lambda_i)$, let $X^\vee := \bigoplus_{i=1}^n B(\lambda_i)^\vee$ and let $\eta_X: \Id\to X\otimes X^\vee$ be given by $\eta_X = \sum_{i=1}^n \eta_{\lambda_i}$. 
	
	Define $\Phi_{X,Y}:\Hom(X^\vee, Y)\to \Hom(\Id, X\otimes Y)$ by $g\mapsto (\id_X\otimes g)\circ \eta_X$. 
	It is easy to see by Schur's Lemma that $\Phi_{X,Y}$ is a natural isomorphism when $X$ and $Y$ are simple objects. Thus, by semisimplicity, $\Phi_{X,Y}$ is a direct sum of isomorphisms and hence an isomorphism. 
	Therefore, $X^\vee$ is a weak dual of $X$ for every object $X$ in $\mathfrak g$-Crys. 
	Finally, it is straight-forward to check that $(-^\vee)^\vee$ is naturally isomorphic to the identity functor so, in particular, $-^\vee$ is an equivalence.
\end{proof}

\begin{Corollary}
	$\mathfrak g$-Crys has no nontrivial tensor ideals.
\end{Corollary}

\begin{proof}
	This is now immediate from \autoref{no-tensor-ideals} and \autoref{weak-rigidity-of-g-crys}. 
\end{proof}

\section{Jones--Wenzl projectors}

\subsection{A Recurrence Relation Approach}

As observed in \autoref{JW remark}, if $P:B_{\epsilon}\to B_{\epsilon'}$ is an isomorphism, then $j_{\epsilon'}= P^{-1}\circ j_{\epsilon}\circ P$, so it suffices to determine $j_{\epsilon}$ for $\epsilon$ with labels occurring in some fixed order, e.g. weakly increasing. 

We first note that \autoref{pairwise jw lemma} implies a general recursive formula for the Jones--Wenzl projector.

\begin{Lemma}\label{2 strand JW}
If $\epsilon=(\epsilon_1,\epsilon_2)$, the Jones--Wenzl projector is $\id_{B_{\epsilon}}-\sum_{b} \overline{f_b}\circ f_b$ where the sum is over all two-strand atoms with domain $B_{\epsilon_1}\otimes B_{\epsilon_2}$.
\end{Lemma}

\begin{proof}
This is immediate. 
\end{proof}

\begin{Proposition}
Assume the labels of $B_{\epsilon}$ are weakly increasing. Then we have \[j_{\epsilon}= \prod_{\substack{1 < r \le s \\ i < r^*}} \gamma_{i,r}  \prod_{k} j_{k-1,k}\] where we use the notation of \autoref{pairwise jw lemma} so that $j_{k-1,k}$ is an identity-padded two-strand Jones--Wenzl projector, and $\gamma_{i,r}$ is built from crossings and a two-strand Jones--Wenzl projector. We use product notation for composition, noting that all of the projectors involved commute.
\end{Proposition}

\begin{proof}
This is just a restatement of \autoref{pairwise jw lemma}.
\end{proof}

We now obtain simpler formulas for the Jones--Wenzl projectors in type $A,B,C,D,\mathrm{G}_2$ which are easier to apply. For these types, we index $B_1,\dots, B_n$ according to the usual labeling of the vertices of the associated Dynkin diagram, so $B_n$ corresponds to the spin representation in type $B$, and $B_{n-1}, B_n$ correspond to the half-spin representations in type $D$.   


\begin{Theorem}\label{JW theorem} 
For type $A,B,C,\mathrm{G}_2$, assume the strand labels are weakly increasing or decreasing. Then we have 
\[j_{\epsilon}= \prod_{k=1}^{m-1} j_{k,k+1}\] where as before $j_{k,k+1} =\id\otimes j_{\epsilon_{k}, \epsilon_{k+1}}\otimes \id$ is the Jones--Wenzl projector on two adjacent strands. Equivalently, we have the recursive formula
\begin{equation*}j_m=(j_{m-1}\otimes \id)\circ (\id\otimes j_2)=(\id\otimes j_2)\circ (j_{m-1}\otimes \id).
\end{equation*} 
\[
\begin{tikzpicture}[
baseline=(current bounding box.center),
x=1cm,y=1cm,
scale=0.8,
transform shape,
jbox/.style={
draw,
line width=1.2pt,
fill=white,
inner sep=2pt,
minimum height=0.78cm
}
]

\begin{scope}[shift={(0,0)}]
\draw[usual] (0.35,-1.85) -- (0.35,1.85);
\draw[usual] (1.75,-1.85) -- (1.75,1.85);

\node[jbox,minimum width=2.35cm] at (1.05,0) {\(\bm{j_m}\)};
\node at (1.05, 1.08) {\(\cdots\)};
\node at (1.05,-1.08) {\(\cdots\)};
\end{scope}

\node at (3.05,0) {\(=\)};

\begin{scope}[shift={(4.05,0)}]
\draw[usual] (0.35,-1.95) -- (0.35,1.95);
\draw[usual] (1.75,-1.95) -- (1.75,1.95);
\draw[usual] (3.15,-1.95) -- (3.15,1.95);

\node[jbox,minimum width=2.55cm] at (1.05,-0.60) {\(\bm{j_{m-1}}\)};
\node[jbox,minimum width=2.30cm] at (2.45, 1.10) {\(\bm{j_2}\)};

\node at (1.00, 0.38) {\(\cdots\)};
\node at (1.05,-1.32) {\(\cdots\)};
\end{scope}

\node at (8.55,0) {\(=\)};

\begin{scope}[shift={(9.55,0)}]
\draw[usual] (0.35,-1.95) -- (0.35,1.95);
\draw[usual] (1.75,-1.95) -- (1.75,1.95);
\draw[usual] (3.15,-1.95) -- (3.15,1.95);

\node[jbox,minimum width=2.55cm] at (1.05, 0.90) {\(\bm{j_{m-1}}\)};
\node[jbox,minimum width=2.30cm] at (2.45,-0.78) {\(\bm{j_2}\)};

\node at (1.05,0.10) {\(\cdots\)};
\end{scope}

\end{tikzpicture}.
\]
where $j_m$ denotes the Jones--Wenzl projector on the first $m$ strands. 
\end{Theorem}

\begin{Remark}
The formula also applies to type $D$ if we remove (one of) $B_{n-1},B_{n}$. 
\end{Remark}

\begin{proof}[Proof of \autoref{JW theorem} for type $A, C, \mathrm{G}_2$:]
We use the tableau model of crystals as described in \cite{kashiwara1994crystal}. Without loss of generality assume the labels are weakly increasing, so that e.g. in type $A$, an element in $B_{\epsilon}$ is a filling of a Young diagram and it is in the top component if and only if it is a semistandard tableau. The element $x$ survives the projector $j_{m-1}\otimes \id$ if and only if its last $m-1$ columns are semistandard (recall the columns are read from right to left), and it survives $(\id\otimes j_2)$ if and only if its first two columns are semistandard. Thus it survives both projectors if and only if it is semistandard, and the formula holds. The argument for type $C$ is similar, but we have to use the appropriate notion of `semistandardness' as explained in \cite[\S 4.5]{kashiwara1994crystal}. The case of $\mathrm{G}_2$ similarly follows using the tableau model described in \cite[Theorem 2.7]{kang1994crystal}.
\end{proof}

For type $B$ we need some notation. We refer to \cite[\S 5.7]{kashiwara1994crystal} and \cite{lecouvey2003schensted} for background. Let $\omega_1,\dots, \omega_n$ be the fundamental weights, and set $\Omega_i:=\omega_i, 1\le i<n, \, \Omega_n:=2\omega_n$. Write $B(\lambda)$ for the connected crystal of highest weight $\lambda$, then the elements of $B(\Omega_i), 1\le i \le n$ correspond to the admissible columns of type $B$ of height $i$, whereas the elements of $B(\omega_n)$ are the so-called \emph{spin columns} of height $n$. An element in $B(\Omega_1)^{\otimes i_1}\otimes \dots B(\Omega_n)^{\otimes i_n}$ corresponds to a Young diagram with $i_j$ columns of height $j$, with a filling by letters in \[
\mathcal{B}_n
=
\{1 \prec 2 \prec \cdots \prec n \prec 0
\prec \bar n \prec \overline{n-1} \prec \cdots \prec \bar 2 \prec \bar 1\}
\] such that each column is \emph{admissible}, see \cite[\S3.1.1]{lecouvey2003schensted}. Call such an element a \emph{tabloid} (of type $B$). An admissible column naturally corresponds to a pair $(lC,rC)$ of columns which do not contain pairs $(z,\overline{z})$ of barred and unbarred letters; the map $C\mapsto (lC, rC)$ corresponds to the unique map $S_2: B(\Omega_i)\to B(2\Omega_i) \subseteq B(\Omega_i)^{\otimes 2}$ that scales weights and strings by a factor of 2, see \cite[Proposition 3.1.9]{lecouvey2003schensted}. A tabloid $T$ is an \emph{orthogonal tableau} (of type $B$) if it lies in the top component, which may be checked as follows: write $T= C_1\dots C_r$ where $C_1,\dots, C_r$ are the columns (from left to right), then the split form $\operatorname{spl}(T)$ is $(lC_1rC_1)\dots (lC_r rC_r)$. Write $C_1\unlhd C_2$ if $rC_1\preceq lC_2$, i.e. $rC_1$ is taller than or equal to $lC_2$ and entries weakly increase along rows. Then $T$ is a tableau if and only if $\operatorname{spl}(T)$ is, if and only if $C_i\unlhd C_{i+1}$ for each $i$, see \cite[Theorem 3.1.18]{lecouvey2003schensted}. Finally, any $\lambda$ that cannot be written as $\sum_{j=1}^n i_j\Omega_j$ can be written as $\lambda=\lambda'+ \omega_n$, where $\lambda'$ is of the given form. If $T$ is a tabloid, and $\mathfrak{C}$ is a spin column, a spin tabloid $[\mathfrak{C},T]$ is obtained by placing $\mathfrak{C}$ in front of $T$. A spin tabloid corresponds to an element of $B(\Omega_1)^{\otimes i_1}\otimes \dots B(\Omega_n)^{\otimes i_n}\otimes B(\omega_n)$, and it is called a \emph{spin tableau} if it is in the top component. Let $T=C_1\dots C_r$. Then $[\mathfrak{C},T]$ is a spin tableau if and only if $T$ is a tableau and rows of $[\mathfrak{C}, lC_1]$ weakly increase, see \cite[Theorem 4.2.2]{lecouvey2003schensted}. Moreover, there is a crystal isomorphism $S^B$ which associates a (non-spin) column to the tensor product of two spin columns; in the case where $C$ has length $n$, this is simply the map $C\mapsto rC\otimes lC$, see \cite[Lemma 4.1.1]{lecouvey2003schensted}.   

\begin{Lemma}\label{type B standard check}
Let $A\in B(\Omega_i), 1\le i < n, D,E\in B(\Omega_n), S\in B(\omega_n)$, viewed as columns. We also view tensor products of columns as tabloids, but read from right to left. 
\begin{enumerate}
\item $A\otimes D$ is a tableau if and only if $A\otimes rD$ is.
\item $D\otimes E$ is a tableau if and only if $lD\otimes rE$ is.
\item $D\otimes S$ is a spin tableau if and only if $lD\otimes S$ is a spin tableau.
\end{enumerate}
\end{Lemma}
\begin{proof}
These follow from \cite[Theorem 3.1.18, Theorem 4.2.2]{lecouvey2003schensted} as explained above.
\end{proof}

\begin{proof}[Proof of \autoref{JW theorem} for type B] Consider an element $x:=C_1\otimes\dots C_p\otimes S_1\otimes \dots \otimes S_r$ of $B_{\epsilon}=B_{\epsilon_1}\otimes\dots\otimes B_{\epsilon_{p}}\otimes B_n^{\otimes r}$, where $1\le \epsilon_i < n$ for $1\le i \le p$. We claim that it lies inside the top component if and only if each pair of adjacent elements lies in the top component of the corresponding tensor product, which implies our formula. The `only if' direction is clear, so assume all adjacent elements lie in top components. In particular, this is true of each $S_{2k-1}\otimes S_{2k}$, which under the isomorphism $S^B$ described above we may identify with a column $D_k\in B(\Omega_n)$ with $lD_k=S_{2k}, rD_k=S_{2k-1}$ ($D_k$ has height $n$ by weight considerations). If $r$ is even, we can identify $x$ with an orthogonal tabloid, and if $r$ is odd, we can identify it with a spin tabloid. By \autoref{type B standard check}, $C_p\otimes D_1$ is a tableau if and only if $C_p\otimes rD_1=C_p\otimes S_1$ is, which is true by assumption. Similarly, the other items of \autoref{type B standard check} tell us that $D_k\otimes D_{k+1}$, and (if $r=2m+1$) $D_m\otimes S_r$ is a (spin) tableau. This proves the theorem. 
\end{proof}

\begin{Example}
In the case of $\mathfrak{sl}_2$, the two-strand Jones--Wenzl projector is \[
\begin{tikzpicture}[
baseline=(current bounding box.center),
x=0.55cm,y=0.55cm,
line width=0.8pt
]
\draw[usual] (0,-1.0) -- (0,1.0);
\draw[usual] (1,-1.0) -- (1,1.0);
\end{tikzpicture}
\;-\;
\begin{tikzpicture}[
baseline=(current bounding box.center),
x=0.55cm,y=0.55cm,
line width=0.8pt
]
\draw[usual] (0,0.95) .. controls (0,0.55) and (1,0.55) .. (1,0.95);
\draw[usual] (0,-0.95) .. controls (0,-0.55) and (1,-0.55) .. (1,-0.95);
\end{tikzpicture}
\]
and the recursive formula in \autoref{JW theorem} is equivalent to the one given in \cite[Proposition 3.20]{Alqady2025coboundary}:
\[
\begin{tikzpicture}[
baseline=(current bounding box.center),
x=1cm,y=1cm, scale = 0.8,
jbox/.style={
draw,
line width=1.2pt,
fill=white,
inner sep=2pt,
minimum width=1.55cm,
minimum height=0.62cm
}
]

\begin{scope}[shift={(0,0)}]
\draw[usual] (0,-1.85) -- (0,1.85);
\draw[usual] (1.20,-1.85) -- (1.20,1.85);

\node[jbox] at (0.60,0) {\(\bm{j_n}\)};
\node at (0.60, 1.10) {\(\cdots\)};
\node at (0.60,-1.10) {\(\cdots\)};
\end{scope}

\node at (1.95,0) {\(=\)};

\begin{scope}[shift={(2.75,0)}]
\draw[usual] (0,-1.85) -- (0,1.85);
\draw[usual] (1.20,-1.85) -- (1.20,1.85);
\draw[usual] (2.00,-1.85) -- (2.00,1.85);

\node[jbox] at (0.60,0) {\(\bm{j_{n-1}}\)};
\node at (0.60, 1.10) {\(\cdots\)};
\node at (0.60,-1.10) {\(\cdots\)};
\end{scope}

\node at (5.45,0) {\(-\)};

\begin{scope}[shift={(6.25,0)}]
\draw[usual] (0.00,-1.85) -- (0.00,1.85);
\draw[usual] (0.85,-1.85) -- (0.85,1.85);

\draw[usual] (1.45,1.85) -- (1.45,0.62);
\draw[usual] (1.45,-0.62) -- (1.45,-1.85);

\draw[usual] (2.15,1.85) -- (2.15,0.62);
\draw[usual] (2.15,-0.62) -- (2.15,-1.85);

\draw[usual]
(1.45,0.62)
.. controls (1.45,0.28) and (2.15,0.28) ..
(2.15,0.62);

\draw[usual]
(1.45,-0.62)
.. controls (1.45,-0.28) and (2.15,-0.28) ..
(2.15,-0.62);

\node[jbox,minimum width=2.00cm] at (0.72, 1.05) {\(\bm{j_{n-1}}\)};
\node[jbox,minimum width=2.00cm] at (0.72,-1.05) {\(\bm{j_{n-1}}\)};

\node at (0.50, 1.58) {\(\cdots\)};
\node at (0.50, 0.00) {\(\cdots\)};
\node at (0.50,-1.58) {\(\cdots\)};
\end{scope}

\end{tikzpicture}.
\]
This is in stark contrast to the full representation category: the
Jones--Wenzl recursion is coefficient-free, with only signs appearing.
\end{Example}

For type $D$ let us set $\Omega_i:=\omega_i$ for $1\le i \le n-2$, $\Omega_{n-1}:=\omega_{n-1}+\omega_n$, $\Omega_n:=2\omega_n$, and $\overline{\Omega}_n:=2\omega_{n-1}$. Denote by $\Omega_+$ the set of dominant weights generated by $\Omega_1,\dots,\Omega_n$ and $\overline{\Omega}_n$. Then any dominant weight $\lambda$ generated by the fundamental weights $\{\omega_i\mid 1\le i \le n\}$ which is not in $\Omega_+$ can be written uniquely as either $\lambda = \omega_n+\lambda'$, with $\lambda'\in \Omega_+$ and $\overline{\Omega}_n$ not appearing in $\lambda'$, or $\lambda = \omega_{n-1}+\lambda'$, with $\lambda' \in \Omega_+$ and $\Omega_n$ not appearing in $\lambda'$. Suppose first that $\lambda \in \Omega_+$, then writing it as the corresponding tensor product of the crystals $B(\Omega_i)$ and $B(\overline{\Omega}_n)$, we can associate an element in this crystal with an \emph{orthogonal tabloid} (of type $D$) with $i_j$ columns of height $j$, as before (see \cite[\S 6.7]{kashiwara1994crystal}, \cite[\S 3.1.2]{lecouvey2003schensted}). As before, this tabloid is a \emph{tableau} (lies in the top component) if it satisfies admissibility conditions which can be checked on adjacent columns. As for type $B$, there is a procedure that associates a tableau $T=C_1\dots C_r$ with its split form $\operatorname{spl}(T)=lC_1rC_1\dots lC_rrC_r$, and $T$ is a tabloid if and only if $\operatorname{spl}(T)$ is. An element of $B(\omega_{n-1})$ or $B(\omega_{n})$ corresponds to a spin column $\mathfrak{C}$, which has height $n$, and as before it can be appended in front of a tabloid $T$ to form a spin tabloid $[\mathfrak{C}, T]$. A spin tabloid is a \emph{spin tableau} if it lies in the top component; this can be checked by conditions on adjacent columns of $T=C_1C_2\dots C_n$ and $\mathfrak{C}lC_1$.
There is a crystal isomorphism $S^D$ which associates a (non-spin) column to the tensor product of two spin columns; when the column has height $n$, this is simply $C\mapsto rC\otimes lC$, see \cite[Lemma 4.1.2]{lecouvey2003schensted}.

Assume now that the labels of $B_{\epsilon}$ are arranged as follows. Let $a,b$ denote the number of times $n-1$ and $n$ occur. If $b\ge a$, write $b-a=2q+r, r\in \{0,1\}$, and write \[B_{\epsilon}=B_{\epsilon_1}\otimes\dots \otimes B_{\epsilon_s}\otimes (B_{n-1}\otimes B_{n})^{\otimes a} \otimes (B_{n}\otimes B_n)^{\otimes q}\otimes B_n^{\otimes r},\] where $\epsilon_1\le \dots \le \epsilon_s \le n-2$. If $a>b$, write $a-b=2q+r, r\in \{0,1\}$, and write 
\[B_{\epsilon}=B_{\epsilon_1}\otimes\dots \otimes B_{\epsilon_s}\otimes (B_{n-1}\otimes B_{n})^{\otimes b} \otimes (B_{n-1}\otimes B_{n-1})^{\otimes q}\otimes B_{n-1}^{\otimes r},\] where $\epsilon_1\le \dots \le \epsilon_s \le n-2$. Using shorthands $x,y,z$ for the tensor product pairs respectively, $\epsilon$ corresponds to a word of the form $\epsilon_1\dots \epsilon_s x^ay^qn^r$ or $\epsilon_1\dots \epsilon_s x^bz^q(n-1)^r$.

Assume the labels of $B_{\epsilon}$ are arranged as above. Write $\sigma_{i}$ for the identity-padded Jones--Wenzl projector applied to the tensor product corresponding to the $i$-th two-letter adjacent subword $uv, u,v\in \{1,\dots n, x,y,z\}$ in the unique form described above. Note that if $\epsilon_{i}, \epsilon_{i+1}\le n-2$, this is just the adjacent two-strand Jones--Wenzl projector $j_{i,i+1}$ from before. If $u,v \in \{x,y,z\}$, then this corresponds to a four-strand projector, and if $u\in \{x,y,z\}, v\in \{n-1,n\}$ or $u\in \{1,\dots,n-2\}, v\in \{x,y,z\}$ this is a three-strand projector.   
\begin{Theorem}
In type $D$, assume the labels of $B_{\epsilon}$ are arranged as above. Then we have \[j_\epsilon = \prod_i \sigma_i.\]
\end{Theorem}

\begin{proof}
This is similar to the proof of \autoref{JW theorem} for type $B$. By the isomorphism $S^D$, we may identify a tableau in $B_{n-1}\otimes B_n$ with a tableau in $B(\Omega_{n-1})$, a tableau in $B_{n-1}\otimes B_{n-1}$ with one in $B(\overline{\Omega}_n)$, and a tableau in $B_n\otimes B_n$ with one in $B(\Omega_n)$. Surviving the projectors $\sigma_i$ then is equivalent to the two adjacent columns being a (spin) tableau.
\end{proof}

\subsection{A M\"obius Inversion Approach}

Similar to \cite[\S 3.4]{Alqady2025coboundary}, Jones--Wenzl projectors may be computed using M\"obius inversion, but in the general case more care is needed. 
For example, unlike for $\mathfrak{sl}_2$, the endomorphism algebras of $\Fund(\mathfrak{sl}_3)$ are not in general monoid algebras with respect to the basis $\mathcal{B}(\epsilon,\epsilon)=\{\overline{f}\circ P\circ g \}$. 
In particular, composition of such basis diagrams may be a linear combination of diagrams (see \eqref{eq:SL3-9} for an example).
However, we can directly consider the poset of idempotents without worrying about a monoid structure for the bases. 
The following lemma shows that we have a bijection between idempotent basis elements and distinguished bottom diagrams.

\begin{Lemma}
A basis element $e\in \mathcal B(\epsilon, \epsilon)$ is an idempotent if and only if $e = \overline{f}\circ f$ for some $f\in \mathcal D_{\epsilon}$.
\end{Lemma}

\begin{proof}
For the ``if'' direction, note that $f\circ \overline{f}\circ f = f$ since $\overline f$ is the inverse of $f$ on its image. Thus, $$e^2 = \overline{f}\circ f\circ \overline f \circ f = \overline{f}\circ f = e.$$ 
Now suppose that $e=\overline f \circ P \circ g\in \mathcal B(\epsilon, \epsilon)$ is an idempotent.
If $e$ is nonzero on some component $C\subset B_{\epsilon}$,
then $e$ sends $C$ to an isomorphic component $C'\subset B_{\epsilon}$.
But $e^2 = e$ only if $e(C') = C' = e(C)$, which implies that $C=C'$ since $e$ is injective outside its kernel by construction.
It follows that $e$ is just the identity outside its kernel which coincides with the projection $\overline f\circ f$.
\end{proof}

Thus, instead of defining a poset structure on the set of idempotents, we define it for distinguished bottom diagrams for simplicity.
The following definition simply formalizes the notion of the building blocks of distinguished bottom diagrams from \autoref{distinguished-bots}.

\begin{Definition}
	A \emph{bonded atom} $a$ (of type $\epsilon$) is defined to be a diagram with domain $\epsilon=(\epsilon_1, \epsilon_2, \dots, \epsilon_k)$ which consists of a single identity-padded atom $f_b$, which is obtained from a branching labelled $b\in B_{\epsilon_i}$ at $\epsilon_i$, precomposed with a minimal permutation $P$ of $B_{\epsilon_1}\otimes B_{\epsilon_2}\otimes \dots B_{\epsilon_{i-1}}$ to guarantee $f_b$ has a domain with the correct labels. In other words,
	$a = (\id\otimes f_b\otimes \id) \circ (P \otimes \id),$
	where $P$ is a minimal permutation (possibly, the identity) matching the domain of $f_b$ and the leftmost strands  of $B_{\epsilon_1}\otimes \dots\otimes   B_{\epsilon_i}$.
\end{Definition}

\begin{Definition}\label{poset def}
	For $f,g\in \mathcal{D}_\epsilon$, we say $f\preceq g$ if there is a sequence of $s\geq 0$ weight decreasing (i.e. non-identity) bonded atoms $a_1, a_2, \dots, a_s$ such that $f = a_s \circ \dots \circ a_2\circ a_1 \circ g$ in $\mathfrak g$-Crys.
\end{Definition}

\begin{Proposition}
	$(\mathcal D_\epsilon, \preceq)$ is a poset.  
\end{Proposition}

\begin{proof}
	Reflexivity and transitivity are clear. To see that $\preceq$ is antisymmetric, suppose that $f\preceq g$ and $g\preceq f$. 
	Thus, $f = a_s \dots  a_1 g$ and $g=a_{s'}' \dots a_1' f$ for weight-decreasing atoms $a_1, \dots, a_s$ and $a_1',\dots, a_{s'}'$.
	Then $\wt(f) \leq \wt(g)$ and $\wt(g)\leq \wt(f)$ since all $a_i$'s are weight-decreasing, by assumption.
	Therefore, $\wt(f)=\wt(g)$ whence $s=s'=0$; in other words, $f=g$.
\end{proof}

\begin{Remark}
	It is not in general true that if $f = a_s \dots a_2 a_1 g$ then the shortest path in $\mathcal D_\epsilon$ is of length $s$. Composition with even a single bonded atom might result in a diagram which is not in $\mathcal D_\epsilon$, but which is equivalent via some $\mathfrak g$-Crys relations to another diagram $a_s\dots a_1 f$ where $a_r\dots a_1f$ and $a_s\dots a_1 f$ are both in $\mathcal D_\epsilon$ for some $s>r\geq 1$. 
	See the poset in \autoref{mobius-example} for an example.
\end{Remark}
 
Note that by construction, every diagram in $\mathcal D_\epsilon$ can be obtained from $\id_\epsilon$ by composing with a sequence of thick atoms. Thus, $\id_\epsilon$ is the unique maximum element of the poset $\mathcal D_\epsilon$. 
Define $\mu(f) := \mu(f, \id_\epsilon)$, where $\mu(x,y)$ denotes the M\"obius function of the poset $(\mathcal D_\epsilon, \preceq)$.
The function $\mu(f)$ may be computed on all $f\in \mathcal D_{\epsilon}$ recursively via
$
\mu(\id_\epsilon) = 1$ and $\mu(f) = - \sum_{g\succ f} \mu(g)$.

\begin{Proposition}\label{Mobius JW}
The Jones-Wenzl Idempotents $j_\epsilon$ are given by the following formula.
$$j_\epsilon = \sum_{f\in \mathcal D_\epsilon} \mu(f)\; \overline f\circ f.$$
\end{Proposition}

\begin{proof}
Consider the two functions $\mathbb{I}: \mathcal D_\epsilon \to \End(B_\epsilon)$ and $\mathbb{J}: \mathcal D_\epsilon \to \End(B_\epsilon)$ defined by $\mathbb{I}(f) = \overline f\circ f$ and $\mathbb{J}(f) = \overline f \circ j_\epsilon \circ f$.
By \autoref{bijection branching}, every component of $B_\epsilon$ appears as a top component in a unique $f\in \mathcal D_\epsilon$. Thus, we have $\id_\epsilon = \mathbb{I}(\id_\epsilon) = \sum_{f\,\preceq\, \id_\epsilon} \mathbb{J}(f)$.
The claim thus follows from the last equality by the M\"obius Inversion Formula for $\mathcal D_\epsilon$.
\end{proof}

Note that \autoref{Mobius JW} is valid for any Lie type and any $\epsilon$ (not necessarily ordered). Essentially, the role of M\"obius inversion here is simply performing an inclusion-exclusion computation on the components appearing in the image of each idempotent, and thereby isolating the top one. 

\begin{Example}\label{mobius-example} We use the presentation of $\Fund(\mathfrak{sl}_3)$ computed below. 
Consider the $\mathfrak{sl}_3$ crystal $B_{\epsilon} = B_1\otimes B_2\otimes B_1\otimes B_2$.
Then, $B_\epsilon$ has the following branching graph and corresponding distinguished bottom diagrams.
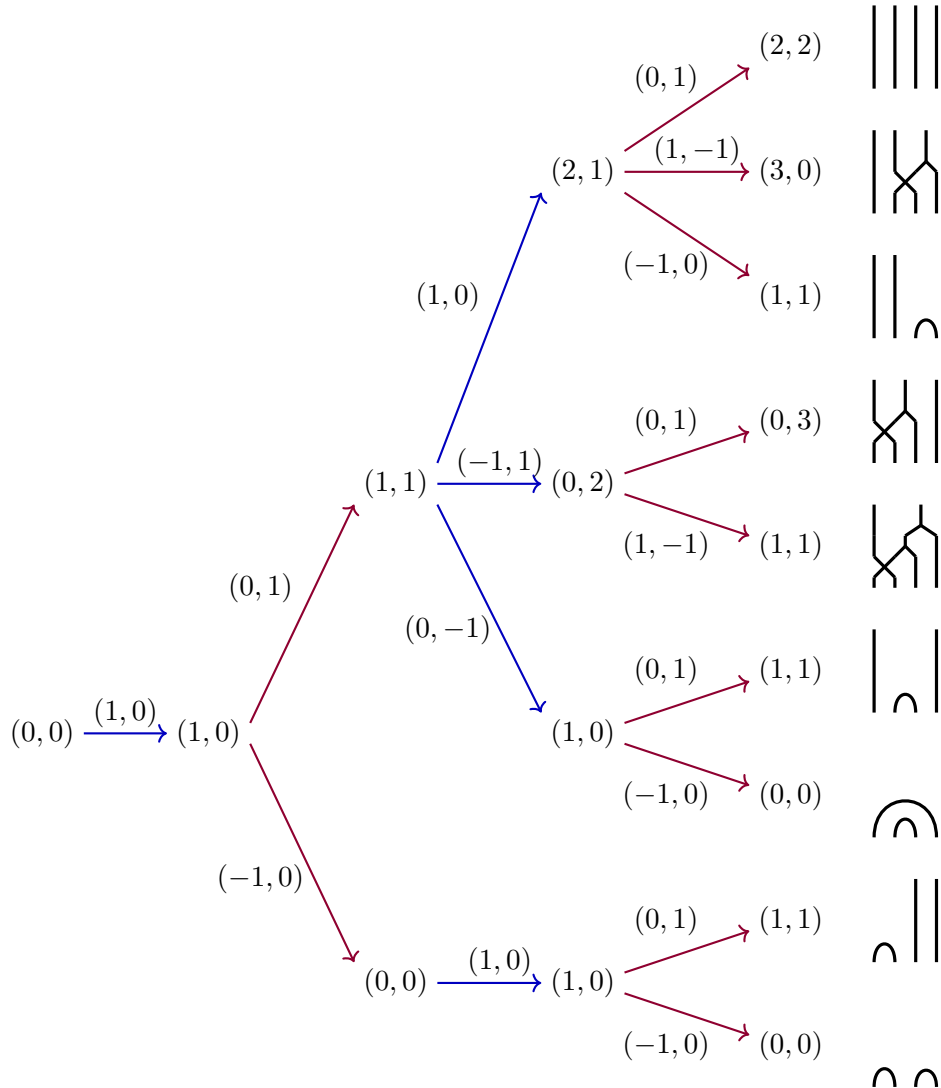
\[\begin{tikzpicture}[scale=0.55]
	\begin{pgfonlayer}{nodelayer}
		\node [style=none] (72) at (10, 12) {$(2,2)$};
		\node [style=none] (73) at (10, 9) {$(3,0)$};
		\node [style=none] (74) at (10, 6) {$(1,1)$};
		\node [style=none] (75) at (10, 3) {$(0,3)$};
		\node [style=none] (76) at (10, 0) {$(1,1)$};
		\node [style=none] (77) at (10, -3) {$(1,1)$};
		\node [style=none] (78) at (10, -6) {$(0,0)$};
		\node [style=none] (79) at (10, -9) {$(1,1)$};
		\node [style=none] (80) at (10, -12) {$(0,0)$};
		\node [style=none] (81) at (5, 9) {$(2,1)$};
		\node [style=none] (82) at (5, 1.5) {$(0,2)$};
		\node [style=none] (83) at (5, -4.5) {$(1,0)$};
		\node [style=none] (84) at (5, -10.5) {$(1,0)$};
		\node [style=none] (85) at (0.5, 1.5) {$(1,1)$};
		\node [style=none] (86) at (0.5, -10.5) {$(0,0)$};
		\node [style=none] (87) at (-4, -4.5) {$(1,0)$};
		\node [style=none] (88) at (-8, -4.5) {$(0,0)$};
		\node [style=none] (89) at (-7, -4.5) {};
		\node [style=none] (90) at (-5, -4.5) {};
		\node [style=none] (91) at (1.5, 2) {};
		\node [style=none] (92) at (4, 8.5) {};
		\node [style=none] (93) at (1.5, 1.5) {};
		\node [style=none] (94) at (4, 1.5) {};
		\node [style=none] (95) at (1.5, 1) {};
		\node [style=none] (96) at (4, -4) {};
		\node [style=none] (97) at (1.5, -10.5) {};
		\node [style=none] (98) at (4, -10.5) {};
		\node [style=none] (99) at (6, 9) {};
		\node [style=none] (100) at (6, 9.5) {};
		\node [style=none] (101) at (6, 8.5) {};
		\node [style=none] (102) at (9, 11.5) {};
		\node [style=none] (103) at (9, 9) {};
		\node [style=none] (104) at (9, 6.5) {};
		\node [style=none] (105) at (6, 1.75) {};
		\node [style=none] (106) at (6, 1.25) {};
		\node [style=none] (107) at (9, 2.75) {};
		\node [style=none] (108) at (9, 0.25) {};
		\node [style=none] (109) at (6, -4.25) {};
		\node [style=none] (110) at (6, -4.75) {};
		\node [style=none] (111) at (9, -3.25) {};
		\node [style=none] (112) at (9, -5.75) {};
		\node [style=none] (113) at (6, -10.25) {};
		\node [style=none] (114) at (6, -10.75) {};
		\node [style=none] (115) at (9, -9.25) {};
		\node [style=none] (116) at (9, -11.75) {};
		\node [style=none] (117) at (-3, -4.25) {};
		\node [style=none] (118) at (-3, -4.75) {};
		\node [style=none] (119) at (-0.5, 1) {};
		\node [style=none] (120) at (-0.5, -10) {};
		\node [style=none] (140) at (12, 0.25) {};
		\node [style=none] (141) at (12.75, 0) {};
		\node [style=none] (142) at (12, -1) {};
		\node [style=none] (143) at (12.5, -1) {};
		\node [style=none] (144) at (13, -1) {};
		\node [style=none] (145) at (12.5, -0.75) {};
		\node [style=none] (146) at (12, -0.75) {};
		\node [style=none] (147) at (12, -0.25) {};
		\node [style=none] (148) at (12.5, -0.25) {};
		\node [style=none] (149) at (13, -0.25) {};
		\node [style=none] (150) at (12.75, 0.25) {};
		\node [style=none] (151) at (13.5, 0.25) {};
		\node [style=none] (152) at (13.5, -1) {};
		\node [style=none] (153) at (13.125, 0.5) {};
		\node [style=none] (154) at (13.125, 1) {};
		\node [style=none] (155) at (12, 1) {};
		\node [style=none] (156) at (12, 4) {};
		\node [style=none] (157) at (12.75, 3.25) {};
		\node [style=none] (158) at (12, 2) {};
		\node [style=none] (159) at (12.5, 2) {};
		\node [style=none] (160) at (13, 2) {};
		\node [style=none] (161) at (12.5, 2.5) {};
		\node [style=none] (162) at (12, 2.5) {};
		\node [style=none] (163) at (12, 3) {};
		\node [style=none] (164) at (12.5, 3) {};
		\node [style=none] (165) at (13, 3) {};
		\node [style=none] (166) at (12.75, 4) {};
		\node [style=none] (167) at (13.5, 4) {};
		\node [style=none] (168) at (13.5, 2) {};
		\node [style=none] (169) at (12.5, 10) {};
		\node [style=none] (170) at (13.25, 9.25) {};
		\node [style=none] (171) at (12.5, 8) {};
		\node [style=none] (172) at (13, 8) {};
		\node [style=none] (173) at (13.5, 8) {};
		\node [style=none] (174) at (13, 8.5) {};
		\node [style=none] (175) at (12.5, 8.5) {};
		\node [style=none] (176) at (12.5, 9) {};
		\node [style=none] (177) at (13, 9) {};
		\node [style=none] (178) at (13.5, 9) {};
		\node [style=none] (179) at (13.25, 10) {};
		\node [style=none] (180) at (12, 10) {};
		\node [style=none] (181) at (12, 8) {};
		\node [style=none] (182) at (12, 13) {};
		\node [style=none] (183) at (12.5, 13) {};
		\node [style=none] (184) at (13, 13) {};
		\node [style=none] (185) at (13.5, 13) {};
		\node [style=none] (186) at (12, 11) {};
		\node [style=none] (187) at (12.5, 11) {};
		\node [style=none] (188) at (13, 11) {};
		\node [style=none] (189) at (13.5, 11) {};
		\node [style=none] (190) at (12, 7) {};
		\node [style=none] (191) at (12.5, 7) {};
		\node [style=none] (194) at (12, 5) {};
		\node [style=none] (195) at (12.5, 5) {};
		\node [style=none] (196) at (13, 5) {};
		\node [style=none] (197) at (13.5, 5) {};
		\node [style=none] (198) at (12, -2) {};
		\node [style=none] (199) at (13.5, -2) {};
		\node [style=none] (200) at (12, -4) {};
		\node [style=none] (201) at (13.5, -4) {};
		\node [style=none] (202) at (12.5, -4) {};
		\node [style=none] (203) at (13, -4) {};
		\node [style=none] (204) at (13.5, -8) {};
		\node [style=none] (205) at (13, -8) {};
		\node [style=none] (206) at (13.5, -10) {};
		\node [style=none] (207) at (13, -10) {};
		\node [style=none] (208) at (12.5, -10) {};
		\node [style=none] (209) at (12, -10) {};
		\node [style=none] (212) at (13.5, -13) {};
		\node [style=none] (213) at (13, -13) {};
		\node [style=none] (214) at (12.5, -13) {};
		\node [style=none] (215) at (12, -13) {};
		\node [style=none] (218) at (12, -7) {};
		\node [style=none] (219) at (13.5, -7) {};
		\node [style=none] (220) at (13, -7) {};
		\node [style=none] (221) at (12.5, -7) {};
		\node [style=none] (222) at (-6, -4) {$(1,0)$};
		\node [style=none] (223) at (1.75, 6) {$(1,0)$};
		\node [style=none] (224) at (3, -10) {$(1,0)$};
		\node [style=none] (225) at (3, 2) {$(-1,1)$};
		\node [style=none] (226) at (1.75, -2) {$(0,-1)$};
		\node [style=none] (227) at (-2.75, -1) {$(0,1)$};
		\node [style=none] (228) at (-2.75, -8) {$(-1,0)$};
		\node [style=none] (229) at (7, -9) {$(0,1)$};
		\node [style=none] (230) at (7, -12) {$(-1,0)$};
		\node [style=none] (231) at (7, -6) {$(-1,0)$};
		\node [style=none] (232) at (7, -3) {$(0,1)$};
		\node [style=none] (233) at (7, 3) {$(0,1)$};
		\node [style=none] (234) at (7, 0) {$(1,-1)$};
		\node [style=none] (235) at (7, 11.25) {$(0,1)$};
		\node [style=none] (236) at (7, 6.75) {$(-1,0)$};
		\node [style=none] (237) at (7.75, 9.5) {$(1,-1)$};
	\end{pgfonlayer}
	\begin{pgfonlayer}{edgelayer}
		\draw [style=redarrow] (100.center) to (102.center);
		\draw [style=redarrow] (99.center) to (103.center);
		\draw [style=redarrow] (101.center) to (104.center);
		\draw [style=redarrow] (105.center) to (107.center);
		\draw [style=redarrow] (106.center) to (108.center);
		\draw [style=redarrow] (109.center) to (111.center);
		\draw [style=redarrow] (110.center) to (112.center);
		\draw [style=redarrow] (113.center) to (115.center);
		\draw [style=redarrow] (114.center) to (116.center);
		\draw [style=bluearrow] (91.center) to (92.center);
		\draw [style=bluearrow] (93.center) to (94.center);
		\draw [style=bluearrow] (95.center) to (96.center);
		\draw [style=bluearrow] (97.center) to (98.center);
		\draw [style=bluearrow] (89.center) to (90.center);
		\draw [style=redarrow] (117.center) to (119.center);
		\draw [style=redarrow] (118.center) to (120.center);
		\draw [style=thickstrand] (142.center) to (146.center);
		\draw [style=thickstrand](146.center) to (148.center);
		\draw [style=thickstrand](143.center) to (145.center);
		\draw [style=thickstrand](145.center) to (147.center);
		\draw [style=thickstrand](147.center) to (140.center);
		\draw [style=thickstrand](144.center) to (149.center);
		\draw [style=thickstrand](149.center) to (141.center);
		\draw [style=thickstrand](148.center) to (141.center);
		\draw [style=thickstrand](141.center) to (150.center);
		\draw [style=thickstrand](151.center) to (152.center);
		\draw [style=thickstrand](150.center) to (153.center);
		\draw [style=thickstrand](151.center) to (153.center);
		\draw [style=thickstrand](153.center) to (154.center);
		\draw [style=thickstrand](140.center) to (155.center);
		\draw [style=thickstrand](158.center) to (162.center);
		\draw [style=thickstrand](162.center) to (164.center);
		\draw [style=thickstrand](159.center) to (161.center);
		\draw [style=thickstrand](161.center) to (163.center);
		\draw [style=thickstrand](163.center) to (156.center);
		\draw [style=thickstrand](160.center) to (165.center);
		\draw [style=thickstrand](165.center) to (157.center);
		\draw [style=thickstrand](164.center) to (157.center);
		\draw [style=thickstrand](157.center) to (166.center);
		\draw [style=thickstrand](167.center) to (168.center);
		\draw [style=thickstrand](171.center) to (175.center);
		\draw [style=thickstrand](175.center) to (177.center);
		\draw [style=thickstrand](172.center) to (174.center);
		\draw [style=thickstrand](174.center) to (176.center);
		\draw [style=thickstrand](176.center) to (169.center);
		\draw [style=thickstrand](173.center) to (178.center);
		\draw [style=thickstrand](178.center) to (170.center);
		\draw [style=thickstrand](177.center) to (170.center);
		\draw [style=thickstrand](170.center) to (179.center);
		\draw [style=thickstrand](180.center) to (181.center);
		\draw [style=thickstrand](182.center) to (186.center);
		\draw [style=thickstrand](183.center) to (187.center);
		\draw [style=thickstrand](184.center) to (188.center);
		\draw [style=thickstrand](185.center) to (189.center);
		\draw [style=thickstrand](190.center) to (194.center);
		\draw [style=thickstrand](191.center) to (195.center);
		\draw [style=thickstrand, bend left=90, looseness=3.00] (196.center) to (197.center);
		\draw [style=thickstrand](198.center) to (200.center);
		\draw [style=thickstrand](199.center) to (201.center);
		\draw [style=thickstrand, bend left=90, looseness=3.00] (202.center) to (203.center);
		\draw [style=thickstrand](204.center) to (206.center);
		\draw [style=thickstrand](205.center) to (207.center);
		\draw [style=thickstrand, bend right=90, looseness=3.00] (208.center) to (209.center);
		\draw [style=thickstrand, bend right=90, looseness=3.00] (214.center) to (215.center);
		\draw [style=thickstrand, bend left=90, looseness=2.75] (213.center) to (212.center);
		\draw [style=thickstrand, bend right=90, looseness=3.00] (220.center) to (221.center);
		\draw [style=thickstrand, bend left=90, looseness=2.00] (218.center) to (219.center);
	\end{pgfonlayer}
\end{tikzpicture}\]
The poset $(\mathcal D_\epsilon, \preceq)$ has the following Hasse diagram and M\"obius function values.
\[\begin{tikzpicture}[scale=0.45]
	\begin{pgfonlayer}{nodelayer}
		\node [style=none] (0) at (-0.75, 1.25) {};
		\node [style=none] (1) at (0, 1) {};
		\node [style=none] (2) at (-0.75, 0) {};
		\node [style=none] (3) at (-0.25, 0) {};
		\node [style=none] (4) at (0.25, 0) {};
		\node [style=none] (5) at (-0.25, 0.25) {};
		\node [style=none] (6) at (-0.75, 0.25) {};
		\node [style=none] (7) at (-0.75, 0.75) {};
		\node [style=none] (8) at (-0.25, 0.75) {};
		\node [style=none] (9) at (0.25, 0.75) {};
		\node [style=none] (10) at (0, 1.25) {};
		\node [style=none] (11) at (0.75, 1.25) {};
		\node [style=none] (12) at (0.75, 0) {};
		\node [style=none] (13) at (0.375, 1.5) {};
		\node [style=none] (14) at (0.375, 2) {};
		\node [style=none] (15) at (-0.75, 2) {};
		\node [style=none] (16) at (-3, 7.5) {};
		\node [style=none] (17) at (-2.25, 6.75) {};
		\node [style=none] (18) at (-3, 5.5) {};
		\node [style=none] (19) at (-2.5, 5.5) {};
		\node [style=none] (20) at (-2, 5.5) {};
		\node [style=none] (21) at (-2.5, 6) {};
		\node [style=none] (22) at (-3, 6) {};
		\node [style=none] (23) at (-3, 6.5) {};
		\node [style=none] (24) at (-2.5, 6.5) {};
		\node [style=none] (25) at (-2, 6.5) {};
		\node [style=none] (26) at (-2.25, 7.5) {};
		\node [style=none] (27) at (-1.5, 7.5) {};
		\node [style=none] (28) at (-1.5, 5.5) {};
		\node [style=none] (29) at (2, 7.5) {};
		\node [style=none] (30) at (2.75, 6.75) {};
		\node [style=none] (31) at (2, 5.5) {};
		\node [style=none] (32) at (2.5, 5.5) {};
		\node [style=none] (33) at (3, 5.5) {};
		\node [style=none] (34) at (2.5, 6) {};
		\node [style=none] (35) at (2, 6) {};
		\node [style=none] (36) at (2, 6.5) {};
		\node [style=none] (37) at (2.5, 6.5) {};
		\node [style=none] (38) at (3, 6.5) {};
		\node [style=none] (39) at (2.75, 7.5) {};
		\node [style=none] (40) at (1.5, 7.5) {};
		\node [style=none] (41) at (1.5, 5.5) {};
		\node [style=none] (42) at (1.5, 13) {};
		\node [style=none] (43) at (2, 13) {};
		\node [style=none] (44) at (2.5, 13) {};
		\node [style=none] (45) at (3, 13) {};
		\node [style=none] (46) at (1.5, 11) {};
		\node [style=none] (47) at (2, 11) {};
		\node [style=none] (48) at (2.5, 11) {};
		\node [style=none] (49) at (3, 11) {};
		\node [style=none] (50) at (3.75, 2) {};
		\node [style=none] (51) at (4.25, 2) {};
		\node [style=none] (52) at (3.75, 0) {};
		\node [style=none] (53) at (4.25, 0) {};
		\node [style=none] (54) at (4.75, 0) {};
		\node [style=none] (55) at (5.25, 0) {};
		\node [style=none] (56) at (-5.5, 2) {};
		\node [style=none] (57) at (-4, 2) {};
		\node [style=none] (58) at (-5.5, 0) {};
		\node [style=none] (59) at (-4, 0) {};
		\node [style=none] (60) at (-5, 0) {};
		\node [style=none] (61) at (-4.5, 0) {};
		\node [style=none] (62) at (9.75, 2) {};
		\node [style=none] (63) at (9.25, 2) {};
		\node [style=none] (64) at (9.75, 0) {};
		\node [style=none] (65) at (9.25, 0) {};
		\node [style=none] (66) at (8.75, 0) {};
		\node [style=none] (67) at (8.25, 0) {};
		\node [style=none] (68) at (5.75, -5.5) {};
		\node [style=none] (69) at (6.25, -5.5) {};
		\node [style=none] (70) at (6.75, -5.5) {};
		\node [style=none] (71) at (7.25, -5.5) {};
		\node [style=none] (72) at (-3.25, -5.5) {};
		\node [style=none] (73) at (-1.75, -5.5) {};
		\node [style=none] (74) at (-2.25, -5.5) {};
		\node [style=none] (75) at (-2.75, -5.5) {};
		\node [style=none] (76) at (2.25, 10.25) {};
		\node [style=none] (77) at (-2.25, 8.25) {};
		\node [style=none] (78) at (2.25, 8.25) {};
		\node [style=none] (79) at (-2.25, 4.75) {};
		\node [style=none] (80) at (2.25, 4.75) {};
		\node [style=none] (81) at (-0.5, 2.5) {};
		\node [style=none] (82) at (0.25, 2.5) {};
		\node [style=none] (83) at (-4.75, 2.5) {};
		\node [style=none] (84) at (4.5, 2.5) {};
		\node [style=none] (85) at (9, 2.5) {};
		\node [style=none] (86) at (-2.75, -3.5) {};
		\node [style=none] (87) at (-2.25, -3.5) {};
		\node [style=none] (88) at (6.25, -3.5) {};
		\node [style=none] (89) at (6.75, -3.5) {};
		\node [style=none] (90) at (-4.75, -0.5) {};
		\node [style=none] (91) at (0, -0.5) {};
		\node [style=none] (92) at (4.5, -0.5) {};
		\node [style=none] (93) at (9, -0.5) {};
		\node [style=none] (95) at (2.25, 10.25) {};
		\node [style=none] (96) at (0.5, 12) {{\textcolor{spinach}{$1$}}};
		\node [style=none] (97) at (-3.75, 6.5) {{\textcolor{spinach}{$-1$}}};
		\node [style=none] (98) at (0.75, 6.5) {{\textcolor{spinach}{$-1$}}};
		\node [style=none] (99) at (-6, 1) {{\textcolor{spinach}{$0$}}};
		\node [style=none] (100) at (-1.25, 1) {{\textcolor{spinach}{$1$}}};
		\node [style=none] (101) at (3.25, 1) {{\textcolor{spinach}{$0$}}};
		\node [style=none] (102) at (7.5, 1) {{\textcolor{spinach}{$-1$}}};
		\node [style=none] (103) at (-4, -5) {{\textcolor{spinach}{$0$}}};
		\node [style=none] (104) at (5, -5) {{\textcolor{spinach}{$1$}}};
	\end{pgfonlayer}
	\begin{pgfonlayer}{edgelayer}
		\draw [style=thickstrand] (2.center) to (6.center);
		\draw [style=thickstrand] (6.center) to (8.center);
		\draw [style=thickstrand] (3.center) to (5.center);
		\draw [style=thickstrand] (5.center) to (7.center);
		\draw [style=thickstrand] (7.center) to (0.center);
		\draw [style=thickstrand] (4.center) to (9.center);
		\draw [style=thickstrand] (9.center) to (1.center);
		\draw [style=thickstrand] (8.center) to (1.center);
		\draw [style=thickstrand] (1.center) to (10.center);
		\draw [style=thickstrand] (11.center) to (12.center);
		\draw [style=thickstrand] (10.center) to (13.center);
		\draw [style=thickstrand] (11.center) to (13.center);
		\draw [style=thickstrand] (13.center) to (14.center);
		\draw [style=thickstrand] (0.center) to (15.center);
		\draw [style=thickstrand] (18.center) to (22.center);
		\draw [style=thickstrand] (22.center) to (24.center);
		\draw [style=thickstrand] (19.center) to (21.center);
		\draw [style=thickstrand] (21.center) to (23.center);
		\draw [style=thickstrand] (23.center) to (16.center);
		\draw [style=thickstrand] (20.center) to (25.center);
		\draw [style=thickstrand] (25.center) to (17.center);
		\draw [style=thickstrand] (24.center) to (17.center);
		\draw [style=thickstrand] (17.center) to (26.center);
		\draw [style=thickstrand] (27.center) to (28.center);
		\draw [style=thickstrand] (31.center) to (35.center);
		\draw [style=thickstrand] (35.center) to (37.center);
		\draw [style=thickstrand] (32.center) to (34.center);
		\draw [style=thickstrand] (34.center) to (36.center);
		\draw [style=thickstrand] (36.center) to (29.center);
		\draw [style=thickstrand] (33.center) to (38.center);
		\draw [style=thickstrand] (38.center) to (30.center);
		\draw [style=thickstrand] (37.center) to (30.center);
		\draw [style=thickstrand] (30.center) to (39.center);
		\draw [style=thickstrand] (40.center) to (41.center);
		\draw [style=thickstrand] (42.center) to (46.center);
		\draw [style=thickstrand] (43.center) to (47.center);
		\draw [style=thickstrand] (44.center) to (48.center);
		\draw [style=thickstrand] (45.center) to (49.center);
		\draw [style=thickstrand] (50.center) to (52.center);
		\draw [style=thickstrand] (51.center) to (53.center);
		\draw [style=thickstrand, bend left=90, looseness=3.00] (54.center) to (55.center);
		\draw [style=thickstrand] (56.center) to (58.center);
		\draw [style=thickstrand] (57.center) to (59.center);
		\draw [style=thickstrand, bend left=90, looseness=3.00] (60.center) to (61.center);
		\draw [style=thickstrand] (62.center) to (64.center);
		\draw [style=thickstrand] (63.center) to (65.center);
		\draw [style=thickstrand, bend right=90, looseness=3.00] (66.center) to (67.center);
		\draw [style=thickstrand, bend left=90, looseness=3.00] (70.center) to (71.center);
		\draw [style=thickstrand, bend right=90, looseness=2.75] (69.center) to (68.center);
		\draw [style=thickstrand, bend right=90, looseness=3.00] (74.center) to (75.center);
		\draw [style=thickstrand, bend left=90, looseness=2.00] (72.center) to (73.center);
		\draw [style=arrow] (76.center) to (77.center);
		\draw [style=arrow] (76.center) to (78.center);
		\draw [style=arrow] (79.center) to (83.center);
		\draw [style=arrow] (79.center) to (81.center);
		\draw [style=arrow] (80.center) to (82.center);
		\draw [style=arrow] (80.center) to (84.center);
		\draw [style=arrow] (76.center) to (85.center);
		\draw [style=arrow] (90.center) to (86.center);
		\draw [style=arrow] (91.center) to (87.center);
		\draw [style=arrow] (92.center) to (88.center);
		\draw [style=arrow] (93.center) to (89.center);
	\end{pgfonlayer}
\end{tikzpicture}\]

Thus, we have 
\[\begin{tikzpicture}[scale=0.55]
	\begin{pgfonlayer}{nodelayer}
		\node [style=none] (0) at (7, -0.5) {};
		\node [style=none] (1) at (7.75, -0.75) {};
		\node [style=none] (2) at (7, -1.75) {};
		\node [style=none] (3) at (7.5, -1.75) {};
		\node [style=none] (4) at (8, -1.75) {};
		\node [style=none] (5) at (7.5, -1.5) {};
		\node [style=none] (6) at (7, -1.5) {};
		\node [style=none] (7) at (7, -1) {};
		\node [style=none] (8) at (7.5, -1) {};
		\node [style=none] (9) at (8, -1) {};
		\node [style=none] (10) at (7.75, -0.5) {};
		\node [style=none] (11) at (8.5, -0.5) {};
		\node [style=none] (12) at (8.5, -1.75) {};
		\node [style=none] (13) at (8.125, -0.25) {};
		\node [style=none] (14) at (8.125, 0.25) {};
		\node [style=none] (15) at (7, 0.25) {};
		\node [style=none] (16) at (-3.5, 0.25) {};
		\node [style=none] (17) at (-2.75, -0.5) {};
		\node [style=none] (18) at (-3.5, -1.75) {};
		\node [style=none] (19) at (-3, -1.75) {};
		\node [style=none] (20) at (-2.5, -1.75) {};
		\node [style=none] (21) at (-3, -1.25) {};
		\node [style=none] (22) at (-3.5, -1.25) {};
		\node [style=none] (23) at (-3.5, -0.75) {};
		\node [style=none] (24) at (-3, -0.75) {};
		\node [style=none] (25) at (-2.5, -0.75) {};
		\node [style=none] (26) at (-2.75, 0.25) {};
		\node [style=none] (27) at (-2, 0.25) {};
		\node [style=none] (28) at (-2, -1.75) {};
		\node [style=none] (29) at (0.5, 0.25) {};
		\node [style=none] (30) at (1.25, -0.5) {};
		\node [style=none] (31) at (0.5, -1.75) {};
		\node [style=none] (32) at (1, -1.75) {};
		\node [style=none] (33) at (1.5, -1.75) {};
		\node [style=none] (34) at (1, -1.25) {};
		\node [style=none] (35) at (0.5, -1.25) {};
		\node [style=none] (36) at (0.5, -0.75) {};
		\node [style=none] (37) at (1, -0.75) {};
		\node [style=none] (38) at (1.5, -0.75) {};
		\node [style=none] (39) at (1.25, 0.25) {};
		\node [style=none] (40) at (0, 0.25) {};
		\node [style=none] (41) at (0, -1.75) {};
		\node [style=none] (42) at (-7, 1.75) {};
		\node [style=none] (43) at (-6.5, 1.75) {};
		\node [style=none] (44) at (-6, 1.75) {};
		\node [style=none] (45) at (-5.5, 1.75) {};
		\node [style=none] (46) at (-7, -1.75) {};
		\node [style=none] (47) at (-6.5, -1.75) {};
		\node [style=none] (48) at (-6, -1.75) {};
		\node [style=none] (49) at (-5.5, -1.75) {};
		\node [style=none] (62) at (5, 1.75) {};
		\node [style=none] (63) at (4.5, 1.75) {};
		\node [style=none] (64) at (5, -1.75) {};
		\node [style=none] (65) at (4.5, -1.75) {};
		\node [style=none] (66) at (4, -1.75) {};
		\node [style=none] (67) at (3.5, -1.75) {};
		\node [style=none] (68) at (10.5, -2) {};
		\node [style=none] (69) at (11, -2) {};
		\node [style=none] (70) at (11.5, -2) {};
		\node [style=none] (71) at (12, -2) {};
		\node [style=none] (79) at (-4.5, 0) {$-$};
		\node [style=none] (80) at (-1, 0) {$-$};
		\node [style=none] (90) at (2.5, 0) {$-$};
		\node [style=none] (91) at (6, 0) {$+$};
		\node [style=none] (93) at (9.5, 0) {$+$};
		\node [style=none] (95) at (4, 1.75) {};
		\node [style=none] (96) at (3.5, 1.75) {};
		\node [style=none] (97) at (7, 0.5) {};
		\node [style=none] (98) at (7.75, 0.75) {};
		\node [style=none] (99) at (7, 1.75) {};
		\node [style=none] (100) at (7.5, 1.75) {};
		\node [style=none] (101) at (8, 1.75) {};
		\node [style=none] (102) at (7.5, 1.5) {};
		\node [style=none] (103) at (7, 1.5) {};
		\node [style=none] (104) at (7, 1) {};
		\node [style=none] (105) at (7.5, 1) {};
		\node [style=none] (106) at (8, 1) {};
		\node [style=none] (107) at (7.75, 0.5) {};
		\node [style=none] (108) at (8.5, 0.5) {};
		\node [style=none] (109) at (8.5, 1.75) {};
		\node [style=none] (110) at (8.125, 0.25) {};
		\node [style=none] (111) at (8.125, -0.25) {};
		\node [style=none] (112) at (7, -0.25) {};
		\node [style=none] (113) at (-3.5, -0.5) {};
		\node [style=none] (114) at (-2.75, 0.25) {};
		\node [style=none] (115) at (-3.5, 1.5) {};
		\node [style=none] (116) at (-3, 1.5) {};
		\node [style=none] (117) at (-2.5, 1.5) {};
		\node [style=none] (118) at (-3, 1) {};
		\node [style=none] (119) at (-3.5, 1) {};
		\node [style=none] (120) at (-3.5, 0.5) {};
		\node [style=none] (121) at (-3, 0.5) {};
		\node [style=none] (122) at (-2.5, 0.5) {};
		\node [style=none] (123) at (-2.75, -0.5) {};
		\node [style=none] (124) at (-2, -0.5) {};
		\node [style=none] (125) at (-2, 1.5) {};
		\node [style=none] (126) at (0.5, -0.5) {};
		\node [style=none] (127) at (1.25, 0.25) {};
		\node [style=none] (128) at (0.5, 1.5) {};
		\node [style=none] (129) at (1, 1.5) {};
		\node [style=none] (130) at (1.5, 1.5) {};
		\node [style=none] (131) at (1, 1) {};
		\node [style=none] (132) at (0.5, 1) {};
		\node [style=none] (133) at (0.5, 0.5) {};
		\node [style=none] (134) at (1, 0.5) {};
		\node [style=none] (135) at (1.5, 0.5) {};
		\node [style=none] (136) at (1.25, -0.5) {};
		\node [style=none] (137) at (0, -0.5) {};
		\node [style=none] (138) at (0, 1.5) {};
		\node [style=none] (139) at (10.5, 2) {};
		\node [style=none] (140) at (11, 2) {};
		\node [style=none] (141) at (11.5, 2) {};
		\node [style=none] (142) at (12, 2) {};
		\node [style=none] (143) at (-9, 0) {$j_\epsilon=$};
	\end{pgfonlayer}
	\begin{pgfonlayer}{edgelayer}
		\draw [style=thickstrand] (2.center) to (6.center);
		\draw [style=thickstrand] (6.center) to (8.center);
		\draw [style=thickstrand] (3.center) to (5.center);
		\draw [style=thickstrand] (5.center) to (7.center);
		\draw [style=thickstrand] (7.center) to (0.center);
		\draw [style=thickstrand] (4.center) to (9.center);
		\draw [style=thickstrand] (9.center) to (1.center);
		\draw [style=thickstrand] (8.center) to (1.center);
		\draw [style=thickstrand] (1.center) to (10.center);
		\draw [style=thickstrand] (11.center) to (12.center);
		\draw [style=thickstrand] (10.center) to (13.center);
		\draw [style=thickstrand] (11.center) to (13.center);
		\draw [style=thickstrand] (13.center) to (14.center);
		\draw [style=thickstrand] (0.center) to (15.center);
		\draw [style=thickstrand] (18.center) to (22.center);
		\draw [style=thickstrand] (22.center) to (24.center);
		\draw [style=thickstrand] (19.center) to (21.center);
		\draw [style=thickstrand] (21.center) to (23.center);
		\draw [style=thickstrand] (23.center) to (16.center);
		\draw [style=thickstrand] (20.center) to (25.center);
		\draw [style=thickstrand] (25.center) to (17.center);
		\draw [style=thickstrand] (24.center) to (17.center);
		\draw [style=thickstrand] (17.center) to (26.center);
		\draw [style=thickstrand] (27.center) to (28.center);
		\draw [style=thickstrand] (31.center) to (35.center);
		\draw [style=thickstrand] (35.center) to (37.center);
		\draw [style=thickstrand] (32.center) to (34.center);
		\draw [style=thickstrand] (34.center) to (36.center);
		\draw [style=thickstrand] (36.center) to (29.center);
		\draw [style=thickstrand] (33.center) to (38.center);
		\draw [style=thickstrand] (38.center) to (30.center);
		\draw [style=thickstrand] (37.center) to (30.center);
		\draw [style=thickstrand] (30.center) to (39.center);
		\draw [style=thickstrand] (40.center) to (41.center);
		\draw [style=thickstrand] (42.center) to (46.center);
		\draw [style=thickstrand] (43.center) to (47.center);
		\draw [style=thickstrand] (44.center) to (48.center);
		\draw [style=thickstrand] (45.center) to (49.center);
		\draw [style=thickstrand] (62.center) to (64.center);
		\draw [style=thickstrand] (63.center) to (65.center);
		\draw [style=thickstrand, bend right=90, looseness=3.00] (66.center) to (67.center);
		\draw [style=thickstrand, bend left=90, looseness=3.00] (70.center) to (71.center);
		\draw [style=thickstrand, bend right=90, looseness=2.75] (69.center) to (68.center);
		\draw [style=thickstrand, bend left=90, looseness=3.00] (95.center) to (96.center);
		\draw [style=thickstrand] (99.center) to (103.center);
		\draw [style=thickstrand] (103.center) to (105.center);
		\draw [style=thickstrand] (100.center) to (102.center);
		\draw [style=thickstrand] (102.center) to (104.center);
		\draw [style=thickstrand] (104.center) to (97.center);
		\draw [style=thickstrand] (101.center) to (106.center);
		\draw [style=thickstrand] (106.center) to (98.center);
		\draw [style=thickstrand] (105.center) to (98.center);
		\draw [style=thickstrand] (98.center) to (107.center);
		\draw [style=thickstrand] (108.center) to (109.center);
		\draw [style=thickstrand] (107.center) to (110.center);
		\draw [style=thickstrand] (108.center) to (110.center);
		\draw [style=thickstrand] (110.center) to (111.center);
		\draw [style=thickstrand] (97.center) to (112.center);
		\draw [style=thickstrand] (115.center) to (119.center);
		\draw [style=thickstrand] (119.center) to (121.center);
		\draw [style=thickstrand] (116.center) to (118.center);
		\draw [style=thickstrand] (118.center) to (120.center);
		\draw [style=thickstrand] (120.center) to (113.center);
		\draw [style=thickstrand] (117.center) to (122.center);
		\draw [style=thickstrand] (122.center) to (114.center);
		\draw [style=thickstrand] (121.center) to (114.center);
		\draw [style=thickstrand] (114.center) to (123.center);
		\draw [style=thickstrand] (124.center) to (125.center);
		\draw [style=thickstrand] (128.center) to (132.center);
		\draw [style=thickstrand] (132.center) to (134.center);
		\draw [style=thickstrand] (129.center) to (131.center);
		\draw [style=thickstrand] (131.center) to (133.center);
		\draw [style=thickstrand] (133.center) to (126.center);
		\draw [style=thickstrand] (130.center) to (135.center);
		\draw [style=thickstrand] (135.center) to (127.center);
		\draw [style=thickstrand] (134.center) to (127.center);
		\draw [style=thickstrand] (127.center) to (136.center);
		\draw [style=thickstrand] (137.center) to (138.center);
		\draw [style=thickstrand, bend right=90, looseness=3.00] (141.center) to (142.center);
		\draw [style=thickstrand, bend left=90, looseness=2.75] (140.center) to (139.center);
	\end{pgfonlayer}
\end{tikzpicture}\;\;.
\]
\end{Example}

\begin{Remark}
The definition of the partial order $\preceq$ on $\mathcal D_\epsilon$ is not canonical, but rather is chosen to make computing the poset relations and M\"obius function algorithmically easiest.
The only required condition of the partial order $\preceq$ to guarantee that \autoref{Mobius JW} holds is that $\mathbb I(\id_\epsilon) = \sum_{f\in\mathcal D_\epsilon} \mathbb J(f)$ or equivalently, by \autoref{bijection branching}, that $\id_\epsilon$ is the unique maximum element in $\mathcal{D}_\epsilon$. 
For example, a more natural choice is to define $f\preceq g$ to mean that $\mathrm{Im}(f)\subseteq \mathrm{Im}(g)$, which is a strictly coarser partial order than that of \autoref{poset def}; however, it seems to be a computationally expensive task to decide whether $\mathrm{Im}(f)\subseteq \mathrm{Im}(g)$, or equivalently whether $f\overline{g}g=f$, for $f,g\in\mathcal D_\epsilon$ in general.
\end{Remark}

\section{Examples}\label{S:Examples}

We now apply \autoref{iso of categories} to give complete presentations when the rank of $\mathfrak{g}$ is one or two. Recall that we have already seen the Temperley--Lieb case in \autoref{E:SL2}, so we start with $\mathfrak{sl}_3$. Below, isomorphic means $\cong_\otimes$.

The examples below should be thought of as \(q=0\) degenerations, or
\emph{crystal versions}, of familiar diagrammatic categories:
\begin{enumerate}

\item The Temperley--Lieb case in \autoref{E:SL2} is a crystal version of
\cite{RuTeWe-sl2,TeLi-the-tl-paper}.

\item The rank-two cases are crystal versions of \cite{Ku-spiders-rank-2}.

\item The \(\mathrm{SO}_3\) case is a crystal version of \cite{Ya-invariant-graphs}.

\item The Motzkin case is a crystal version of \cite{benkart2014motzkin}.

\end{enumerate}

\begin{Remark}
The crystal version of $\mathfrak{gl}_2$-webs (see, e.g., \cite{BeHoPuWe-sl2gl2,Tu-web-reps}) is straightforward to obtain and is left to the reader. A more interesting direction would be to find clean crystal presentations for higher ranks, crystallizing, for instance, the webs in \cite{CaKaMo-webs-skew-howe}, or for other monoidal categories (such as those considered in \cite{KhSiTu-monoidal-cryptography}, most notably the Brauer calculus \cite{Br-brauer-algebra-original}).
\end{Remark}

\subsection{\texorpdfstring{$\mathfrak{sl}_3$}{sl3}}

We use the notation introduced before in \autoref{atom examples} for $\Fund(\mathfrak{sl}_3\text{-}\mathrm{Crys})$.

\begin{Proposition}
The category $\Fund(\mathfrak{sl}_3\text{-}\mathrm{Crys})$ is isomorphic to the $\kk$-linear monoidal category generated by $B_1, B_2$ with generating morphisms
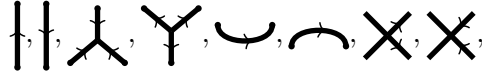
\[
\upvert,\downvert,\trimergedown,  \vtrimergedown, \cupL,\capL,\pcrossing,\mcrossing,
\]
subject to the following relations and their vertical and horizontal reflections, for all orientations that make sense, where \eqref{eq:SL3-0} is in particular to be taken as a definition for the two non-generator atoms denoted \trisplit, \capR. 
\tikzset{slthreerel/.style={baseline=(current bounding box.center),line cap=round,line join=round,usual/.style={line width=2pt,color=black}},slpunct/.style={inner sep=0pt,anchor=base,font=\normalsize}}
\begin{equation*}\tag{SL3.0}\label{eq:SL3-0}
\begin{tikzpicture}[slthreerel,x=.50cm,y=.50cm]
\path[use as bounding box] (-1.0,-.05) rectangle (13.2,3.1);
\coordinate (t) at (0,2.95);
\coordinate (v) at (0,1.45);
\coordinate (bl) at (-.75,.25);
\coordinate (br) at (.75,.25);
\draw[usual] (t) -- (v);
\draw[usual] (v) -- (bl);
\draw[usual] (v) -- (br);
\node at (1.8,1.55) {$=$};
\begin{scope}[shift={(3.0,0)}]
\coordinate (L0) at (0,.25);
\coordinate (L1) at (0,1.20);
\coordinate (L2) at (0,2.05);
\coordinate (T) at (0,2.95);
\coordinate (V) at (1.50,.58);
\coordinate (R0) at (1.50,.25);
\coordinate (A1) at (.92,1.20);
\coordinate (B1) at (2.08,1.20);
\coordinate (A2) at (.92,2.05);
\coordinate (B2) at (2.08,2.05);
\coordinate (C1) at (.92,2.50);
\coordinate (C2) at (2.08,2.50);
\draw[usual] (L0) -- (L1);
\draw[usual] (R0) -- (V);
\draw[usual] (V) -- (A1);
\draw[usual] (V) -- (B1);
\draw[usual] (L1) -- (A2);
\draw[usual] (A1) -- (L2);
\draw[usual] (B1) -- (B2);
\draw[usual] (L2) -- (T);
\draw[usual] (A2) -- (C1);
\draw[usual] (B2) -- (C2);
\draw[usual] (C1) to[out=90,in=180] (1.50,2.85);
\draw[usual] (1.50,2.85) to[out=0,in=90] (C2);
\end{scope}
\node[slpunct] at (7.35,.25) {$,\qquad$};
\draw[usual] (8.15,.25) -- (8.15,2.20);
\draw[usual] (9.65,.25) -- (9.65,2.20);
\draw[usual] (8.15,2.20) to[out=90,in=180] (8.90,2.95);
\draw[usual] (8.90,2.95) to[out=0,in=90] (9.65,2.20);
\node at (10.35,1.35) {$=$};
\begin{scope}[shift={(11.3,0)}]
\draw[usual] (0,.25) -- (1.50,2.20);
\draw[usual] (1.50,.25) -- (0,2.20);
\draw[usual] (0,2.20) to[out=90,in=180] (.75,2.95);
\draw[usual] (.75,2.95) to[out=0,in=90] (1.50,2.20);
\end{scope}
\node[slpunct] at (13.0,.25) {$,$};
\end{tikzpicture}
\end{equation*}
\begin{equation*}\tag{SL3.1}\label{eq:SL3-1}
\begin{tikzpicture}[slthreerel,x=.58cm,y=.58cm]
\path[use as bounding box] (-.35,-.15) rectangle (5.05,3.0);
\draw[usual] (0,0) -- (1.6,1.4);
\draw[usual] (1.6,1.4) -- (0,2.8);
\draw[usual] (1.6,0) -- (0,1.4);
\draw[usual] (0,1.4) -- (1.6,2.8);
\node at (2.6,1.4) {$=$};
\begin{scope}[shift={(3.35,0)}]
\draw[usual] (0,0) -- (0,2.8);
\draw[usual] (1.2,0) -- (1.2,2.8);
\end{scope}
\node[slpunct] at (4.85,0) {$.$};
\end{tikzpicture}
\end{equation*}

\begin{equation*}\tag{SL3.2}\label{eq:SL3-2}
\begin{tikzpicture}[slthreerel,x=.72cm,y=.72cm]
\path[use as bounding box] (-.75,-1.05) rectangle (3.65,1.05);
\draw[usual] (0,0) circle [radius=.65];
\node at (1.45,0) {$=$};
\node at (2.35,0) {$1$};
\node[slpunct] at (2.8,-0.2) {$.$};
\end{tikzpicture}
\end{equation*}

\begin{equation*}\tag{SL3.3}\label{eq:SL3-3}
\begin{tikzpicture}[slthreerel,x=.74cm,y=.74cm]
\path[use as bounding box] (-.85,-.15) rectangle (3.25,2.95);

\coordinate (m) at (0,1.05);
\coordinate (u) at (0,1.75);
\draw[usual] (-.55,1.4) -- (m);
\draw[usual] (.55,1.4) -- (m);
\draw[usual] (0,0) -- (m);
\draw[usual] (-.55,1.4) -- (u);
\draw[usual] (.55,1.4) -- (u);
\draw[usual] (0,2.8) -- (u);
\node at (1.55,1.4) {$=$};
\begin{scope}[shift={(2.55,0)}]
\draw[usual] (0,0) -- (0,2.8);
\end{scope}
\node[slpunct] at (3.05,0) {$.$};
\end{tikzpicture}
\end{equation*}

\begin{equation*}\tag{SL3.4}\label{eq:SL3-4}
\begin{tikzpicture}[slthreerel,x=.82cm,y=.82cm]
\path[use as bounding box] (-.65,-.75) rectangle (5.95,1.75);

\begin{scope}[shift={(0,0)}]
\coordinate (v) at (0,.35);
\coordinate (l) at (-.45,1.05);
\coordinate (r) at (.45,1.05);
\coordinate (fr) at (.95,1.05);
\draw[usual] (l) -- (v);
\draw[usual] (l) -- (-.45,1.55);
\draw[usual] (r) -- (v);
\draw[usual] (0,-.55) -- (v);
\draw[usual] (r) to[out=90,in=180] (.70,1.55);
\draw[usual] (.70,1.55) to[out=0,in=90] (fr);
\draw[usual] (fr) -- (.95,-.55);
\end{scope}

\node at (1.6,.50) {$=$};

\begin{scope}[shift={(2.35,0)}]
\draw[usual] (0,1.25) -- (0,.25);
\draw[usual] (0,.25) to[out=270,in=180] (.45,-.25);
\draw[usual] (.45,-.25) to[out=0,in=270] (.9,.25);
\draw[usual] (.9,.25) -- (.9,.75);
\draw[usual] (.9,.75) to[out=90,in=180] (1.35,1.25);
\draw[usual] (1.35,1.25) to[out=0,in=90] (1.8,.75);
\draw[usual] (1.8,.75) -- (1.8,-.55);
\end{scope}

\node at (4.75,.50) {$=$};
\node at (5.30,.50) {$0.$};

\end{tikzpicture}
\end{equation*}

\begin{equation*}\tag{SL3.5}\label{eq:SL3-5}
\begin{tikzpicture}[slthreerel,x=.86cm,y=.86cm]
\path[use as bounding box] (-.35,-.15) rectangle (4.75,1.95);
\draw[usual] (-.10,0) -- (.30,.9);
\draw[usual] (.30,.9) -- (-.10,1.8);
\draw[usual] (1.45,0) -- (1.05,.9);
\draw[usual] (1.05,.9) -- (1.45,1.8);
\draw[usual] (.30,.9) -- (1.05,.9);
\node at (2.15,.9) {$=$};
\begin{scope}[shift={(3.10,0)}]
\draw[usual] (0,0) to[out=90,in=180] (.575,.45);
\draw[usual] (.575,.45) to[out=0,in=90] (1.15,0);
\draw[usual] (0,1.8) to[out=270,in=180] (.575,1.35);
\draw[usual] (.575,1.35) to[out=0,in=270] (1.15,1.8);
\end{scope}
\node[slpunct] at (4.55,0) {$.$};
\end{tikzpicture}
\end{equation*}

\begin{equation*}\tag{SL3.6}\label{eq:SL3-6}
\begin{tikzpicture}[slthreerel,x=.60cm,y=.60cm]
\path[use as bounding box] (-.15,-.1) rectangle (8.45,3.05);
\draw[usual] (0,2.2) to[out=90,in=180] (1.0,2.9);
\draw[usual] (1.0,2.9) to[out=0,in=90] (2.0,2.2);
\draw[usual] (0,2.2) -- (0,0);
\draw[usual] (2.0,2.2) -- (2.0,1.15);
\draw[usual] (2.0,1.15) -- (1.15,0);
\draw[usual] (2.0,1.15) -- (2.85,0);
\node at (4.15,1.4) {$=$};
\begin{scope}[shift={(5.6,0)}]
\draw[usual] (0,2.2) to[out=90,in=180] (1.0,2.9);
\draw[usual] (1.0,2.9) to[out=0,in=90] (2.0,2.2);
\draw[usual] (0,2.2) -- (0,1.15);
\draw[usual] (0,1.15) -- (-.85,0);
\draw[usual] (0,1.15) -- (.85,0);
\draw[usual] (2.0,2.2) -- (2.0,0);
\end{scope}
\node[slpunct] at (8.15,0) {$.$};
\end{tikzpicture}
\end{equation*}

\begin{equation*}\tag{SL3.7}\label{eq:SL3-7}
\begin{tikzpicture}[slthreerel,x=.52cm,y=.52cm]
\path[use as bounding box] (-.2,-.2) rectangle (18.95,3.55);
\draw[usual] (0,0) -- (0,.8);
\draw[usual] (1,0) -- (2,.8);
\draw[usual] (2,0) -- (1,.8);
\draw[usual] (0,.8) -- (1,1.6);
\draw[usual] (1,.8) -- (0,1.6);
\draw[usual] (2,.8) -- (2,1.6);
\draw[usual] (1,1.6) -- (1.5,2.4);
\draw[usual] (2,1.6) -- (1.5,2.4);
\draw[usual] (0,1.6) -- (0,3.2);
\draw[usual] (1.5,2.4) -- (1.5,3.2);
\node at (3.65,1.6) {$=$};
\begin{scope}[shift={(5.2,0)}]
\draw[usual] (0,0) -- (.55,1.6);
\draw[usual] (1.1,0) -- (.55,1.6);
\draw[usual] (.55,1.6) -- (.55,3.2);
\draw[usual] (2.1,0) -- (2.1,3.2);
\end{scope}
\node at (8.7,1.6) {$+$};
\begin{scope}[shift={(10.2,0)}]
\draw[usual] (.55,0) -- (.55,1.6);
\draw[usual] (.55,1.6) -- (0,3.2);
\draw[usual] (.55,1.6) -- (1.1,3.2);
\draw[usual] (2.0,0) to[out=90,in=180] (2.55,.40);
\draw[usual] (2.55,.40) to[out=0,in=90] (3.1,0);
\end{scope}
\node at (14.65,1.6) {$-$};
\begin{scope}[shift={(16.15,0)}]
\coordinate (u) at (1.20,1.65);
\coordinate (a) at (.45,1.10);
\coordinate (b) at (1.95,1.10);
\coordinate (l) at (.45,.52);
\coordinate (s) at (1.20,2.40);
\draw[usual] (u) -- (a);
\draw[usual] (u) -- (b);
\draw[usual] (b) -- (1.95,0);
\draw[usual] (a) -- (l);
\draw[usual] (l) -- (-.25,0);
\draw[usual] (l) -- (1.15,0);
\draw[usual] (u) -- (s);
\draw[usual] (s) -- (.35,3.2);
\draw[usual] (s) -- (2.05,3.2);
\end{scope}
\node[slpunct] at (18.65,0) {$.$};
\end{tikzpicture}
\end{equation*}

\begin{equation*}\tag{SL3.8}\label{eq:SL3-8}
\begin{tikzpicture}[slthreerel,x=.58cm,y=.58cm]
\path[use as bounding box] (-.1,-.1) rectangle (8.95,3.05);
\draw[usual] (0,0) -- (0,2.0);
\draw[usual] (0,2.0) to[out=90,in=180] (.95,2.7);
\draw[usual] (.95,2.7) to[out=0,in=90] (1.9,2.0);
\draw[usual] (1.9,2.0) -- (3.25,0);
\draw[usual] (1.9,0) -- (3.25,2.0);
\draw[usual] (3.25,2.0) -- (3.25,2.8);
\node at (4.65,1.4) {$=$};
\begin{scope}[shift={(6.35,0)}]
\coordinate (u) at (1.2,2.05);
\coordinate (a) at (.45,1.35);
\coordinate (b) at (1.95,1.35);
\coordinate (l) at (.45,.65);
\draw[usual] (1.2,2.8) -- (u);
\draw[usual] (u) -- (a);
\draw[usual] (u) -- (b);
\draw[usual] (b) -- (1.95,0);
\draw[usual] (a) -- (l);
\draw[usual] (l) -- (-.25,0);
\draw[usual] (l) -- (1.15,0);
\end{scope}
\node[slpunct] at (8.65,0) {$.$};
\end{tikzpicture}
\end{equation*}

\begin{equation*}\tag{SL3.9}\label{eq:SL3-9}
\begin{tikzpicture}[slthreerel,x=.44cm,y=.44cm]
\path[use as bounding box] (-1.25,-.2) rectangle (23.75,4.25);
\coordinate (M) at (0,3);
\coordinate (ML) at (-1,2);
\coordinate (MR) at (1,2);
\draw[usual] (0,4) -- (M);
\draw[usual] (3,4) -- (3,2);
\draw[usual] (M) -- (ML);
\draw[usual] (M) -- (MR);
\draw[usual] (1,2) -- (3,0);
\draw[usual] (3,2) -- (1,0);
\draw[usual] (ML) -- (-1,0);
\node at (4.7,2) {$=$};
\begin{scope}[shift={(6.2,0)}]
\draw[usual] (0,4) -- (0,2);
\draw[usual] (3,4) -- (3,3);
\draw[usual] (4,2) -- (4,0);
\draw[usual] (0,2) -- (2,0);
\draw[usual] (2,2) -- (0,0);
\draw[usual] (3,3) -- (2,2);
\draw[usual] (3,3) -- (4,2);
\end{scope}
\node at (11.65,2) {$-$};
\begin{scope}[shift={(13.2,0)}]
\coordinate (M) at (.8,2);
\draw[usual] (.8,0) -- (M);
\draw[usual] (M) -- (0,4);
\draw[usual] (M) -- (1.6,4);
\draw[usual] (2.55,0) to[out=90,in=180] (3.10,.45);
\draw[usual] (3.10,.45) to[out=0,in=90] (3.65,0);
\end{scope}
\node at (17.9,2) {$+$};
\begin{scope}[shift={(19.5,0)}]
\coordinate (M) at (2.55,2);
\draw[usual] (0,0) to[out=90,in=180] (.55,.45);
\draw[usual] (.55,.45) to[out=0,in=90] (1.1,0);
\draw[usual] (2.55,0) -- (M);
\draw[usual] (M) -- (1.75,4);
\draw[usual] (M) -- (3.35,4);
\end{scope}
\node[slpunct] at (23.35,0) {$.$};
\end{tikzpicture}
\end{equation*}
\end{Proposition}

\begin{proof}
This follows from \autoref{iso of categories} and direct computation of crystal morphisms. The relations \eqref{eq:SL3-0}, \eqref{eq:SL3-2}, \eqref{eq:SL3-3}, \eqref{eq:SL3-4} and \eqref{eq:SL3-5} are $H=I$ relations, and the relations \eqref{eq:SL3-6}, \eqref{eq:SL3-7}, \eqref{eq:SL3-8} and \eqref{eq:SL3-9} are orientation-reversing relations. Note that \eqref{eq:SL3-8} is an orientation-reversing relation of the form \eqref{eq:OR1}, while its horizontal reflection is of the form \eqref{eq:OR2}.
\end{proof}

\begin{Remark}
The presentations of $\Fund(\mathfrak{g}\text{-}\mathrm{Crys})$ generated by applying \autoref{iso of categories} are not necessarily minimal. 
For example, in the above presentation of $\Fund(\mathfrak{sl}_3\text{-}\mathrm{Crys})$, relation \eqref{eq:SL3-7} is redundant: it follows from precomposing the horizonal reflection of \eqref{eq:SL3-9} by the permutation switching the first and second strands, followed by applying \eqref{eq:SL3-8} to the middle factor on the right-hand side.
\end{Remark}

Let us denote the two equivalent expressions in \eqref{eq:SL3-6} by $\tcupL$ or $\tcupR$. 
As in \cite[\S 3.2]{Alqady2025coboundary}, we are able to give an explicit formula for the Jones--Wenzl projector of type $\epsilon$, assuming that the labels of $\epsilon=(\epsilon_1,\dots, \epsilon_m)$ are weakly increasing. Fix such an $\epsilon$. Define an \emph{apt quintuple} to be a quintuple $\mathcal{I}:=(I_1,\dots, I_5), I_1,\dots, I_5 \subseteq \{1,\dots, m\}$, such that $I_1,\dots, I_5$ record a collection of possible indices $i$ such that beginning at $\epsilon_i$ a copy of $\trimergedown, \, \trisplit, \, \capL,  \, \tupL, \, \tcupR$, respectively, may be placed on the domain of $B_{\epsilon}$ (e.g. this means if $i\in I_1$, then $\epsilon_i=\epsilon_{i+1}=1$, if $i\in I_3$, then $\epsilon_i=1, \epsilon_{i+1}=2$, and so on), such that these placements are all compatible and assemble to a bottom diagram $f$ (so that, for example, if $i\in I_k$ for some $1\le k\le 5$, then $i+1\notin I_k$). We then define $c_{\mathcal{I}}$ to be $\overline{f}\circ f$. For example, for $\epsilon=(1,1,1)$, the four apt quintuples are $\emptyset^5, (\{1\},\emptyset^4), (\{2\},\emptyset^4), (\emptyset^3,\{1\},\emptyset)$, corresponding to the four terms below. For each $\mathcal{I}$, we assign a number $\sigma(\mathcal{I})=\sum_{i=1}^3 \lvert I_i\rvert$. 

\begin{gather}\label{jw4 sl3}
\begin{tikzpicture}[anchorbase,scale=0.9]
\begin{scope}
\foreach \x in {0,0.35,0.7} {
\draw[usual,slup] (\x,0) -- (\x,1.0);
}
\end{scope}
\node at (1.4,0.5) {$-$};
\begin{scope}[xshift=2.0cm]
\draw[usual,slup] (-0.45,0) -- (0,0.35);
\draw[usual,slup] (0.45,0) -- (0,0.35);
\draw[usual,sldown] (-0.45,1.0) -- (0,0.7);
\draw[usual,sldown] (0.45,1.0) -- (0,0.7);
\draw[usual,sldown] (0,0.35) -- (0,0.7);
\draw[usual,slup] (0.9,0) -- (0.9,1.0);
\end{scope}
\node at (3.7,0.5) {$-$};
\begin{scope}[xshift=4.3cm]
\draw[usual,slup] (0,0) -- (0,1.0);
\draw[usual,slup] (0.45,0) -- (0.9,0.35);
\draw[usual,slup] (1.35,0) -- (0.9,0.35);
\draw[usual,sldown] (0.45,1.0) -- (0.9,0.7);
\draw[usual,sldown] (1.35,1.0) -- (0.9,0.7);
\draw[usual,sldown] (0.9,0.35) -- (0.9,0.7);
\end{scope}
\node at (6.0,0.5) {$+$};
\begin{scope}[xshift=6.6cm]
\node at (0,0.2) {\tupL};
\node at (0,0.8) {\tdownL};
\end{scope}
\end{tikzpicture}.
\end{gather}

\begin{Proposition}\label{sl3 jw}
Assume $\epsilon$ has labels weakly increasing. With notation as above, we have \[j_{\epsilon} = \sum_{\mathcal{I} \text{ apt}} (-1)^{\sigma(\mathcal{I})} c_{\mathcal{I}}.\]
\end{Proposition}
For example, \eqref{jw4 sl3} gives $j_{(1,1,1)}$. 
\begin{proof}
The two-strand Jones--Wenzl projector is either identity minus cup/cap or identity minus merge/split. By induction, it is easy to see that applying the recursive formula given in \autoref{JW theorem} and using the $H=I$ relations to simplify, yields the claimed formula.
\end{proof}

\subsection{\texorpdfstring{$\mathfrak{sp}_4$}{sp4}}

We now describe the next interesting case in rank two.

\tikzset{
sp4rel/.style={baseline={([yshift=-0.5ex]current bounding box.center)},line cap=round,line join=round},
sp4one/.style={line width=2.0pt,color=black},
sp4two/.style={line width=2.0pt,color=cyan}
}

\setlength{\abovedisplayskip}{1.8em plus 0.25em minus 0.1em}
\setlength{\belowdisplayskip}{1.8em plus 0.25em minus 0.1em}
\setlength{\abovedisplayshortskip}{1.8em plus 0.25em minus 0.1em}
\setlength{\belowdisplayshortskip}{1.8em plus 0.25em minus 0.1em}

\begin{Proposition}
The category \(\Fund(\mathfrak{sp}_4 \text{-} \mathrm{Crys})\) is isomorphic to the $\kk$-linear monoidal category generated
by objects $B_1, B_2$ and morphisms
\[
\spfourcapone,\qquad
\spfourcaptwo,\qquad
\spfourmergeaa,\qquad
\spfourmergeonetwo,\qquad
\spfourcrosstwoone,\qquad
\spfourswitchdown,
\]
together with their vertical reflections
\[
\spfourcupone,\qquad
\spfourcuptwo,\qquad
\spfoursplitaa,\qquad
\spfoursplitonetwo,\qquad
\spfourcrossonetwo,\qquad
\spfourswitchup,
\]
with relations as follows (where we also impose the vertical reflections), where \eqref{eq:SP4-0} may be regarded as a definition (of the atom with domain $(2,1)$). The left-hand side of \eqref{eq:SP4-1} is the other non-identity atom.
\begin{equation*}\tag{SP4.0}\label{eq:SP4-0}
\begin{tikzpicture}[sp4rel,scale=0.9]
\path[use as bounding box] (-.85,-.20) rectangle (4.10,1.90);
\begin{scope}
\draw[sp4two] (-0.55,0) node[below] {$2$} -- (0,0.75);
\draw[sp4one] ( 0.55,0) node[below] {$1$} -- (0,0.75);
\draw[sp4one] (0,0.75) -- (0,1.65) node[above] {$1$};
\end{scope}

\node at (1.55,0.85) {$=$};

\begin{scope}[xshift=3.0cm]
\draw[sp4two] (-0.55,0) node[below] {$2$} -- (0,0.45);
\draw[sp4two] (0,0.45) -- (0.55,0.90);
\draw[sp4one] (0.55,0) node[below] {$1$} -- (0,0.45);
\draw[sp4one] (0,0.45) -- (-0.55,0.90);
\draw[sp4one] (-0.55,0.90) -- (0,1.35);
\draw[sp4two] ( 0.55,0.90) -- (0,1.35);
\draw[sp4one] (0,1.35) -- (0,1.65) node[above] {$1$};
\end{scope}
\end{tikzpicture}
.
\end{equation*}

\begin{equation*}\tag{SP4.1}\label{eq:SP4-1}
\begin{array}{c@{\quad=\quad}c}
\begin{tikzpicture}[sp4rel,x=.62cm,y=.62cm]
\path[use as bounding box] (-.25,-.15) rectangle (2.25,3.05);
\draw[sp4one] (0,0) -- (0,1.45);
\draw[sp4one] (1,0) -- (1.50,.75);
\draw[sp4two] (2,0) -- (1.50,.75);
\draw[sp4one] (1.50,.75) -- (1.50,1.45);
\draw[sp4one] (0,1.45) -- (.75,2.20);
\draw[sp4one] (1.50,1.45) -- (.75,2.20);
\draw[sp4two] (.75,2.20) -- (.75,2.85);
\end{tikzpicture}
&
\begin{tikzpicture}[sp4rel,x=.62cm,y=.62cm]
\path[use as bounding box] (-.25,-.15) rectangle (2.25,3.05);
\draw[sp4one] (0,0) -- (0,1.05);
\draw[sp4one] (1,0) -- (2,1.05);
\draw[sp4two] (2,0) -- (1,1.05);
\draw[sp4one] (0,1.05) -- (1,2.10) -- (1,2.25);
\draw[sp4two] (1,1.05) -- (0,2.10) -- (0,2.85);
\draw[sp4one] (2,1.05) -- (2,2.25);
\draw[sp4one] (1,2.25)
to[out=90,in=180] (1.50,2.85)
to[out=0,in=90] (2,2.25);
\end{tikzpicture}
\end{array}
.
\end{equation*}

\begin{equation*}\tag{SP4.2}\label{eq:SP4-2}
\begin{tikzpicture}[sp4rel,scale=0.95]
\path[use as bounding box] (-.10,-.70) rectangle (4.70,.70);
\draw[sp4one] (.55,0) circle (.55);
\node at (1.55,0) {$=$};
\node at (2.30,0) {$1$};
\node at (3.05,0) {$=$};
\draw[sp4two] (4.05,0) circle (.55);
\end{tikzpicture}
.
\end{equation*}

\begin{equation*}\tag{SP4.3}\label{eq:SP4-3}
\begin{tikzpicture}[sp4rel,scale=0.95]
\coordinate (b) at (0,0);
\coordinate (m) at (0,0.75);
\coordinate (l) at (-0.55,1.35);
\coordinate (r) at (0.55,1.35);
\coordinate (u) at (0,1.95);
\coordinate (t) at (0,2.70);

\draw[sp4two] (b) -- (m);
\draw[sp4one] (m) -- (l);
\draw[sp4one] (m) -- (r);
\draw[sp4one] (l) -- (u);
\draw[sp4one] (r) -- (u);
\draw[sp4two] (u) -- (t);

\node at (1.70,1.35) {$=$};

\begin{scope}[xshift=2.70cm]
\draw[sp4two] (0,0) -- (0,2.70);
\end{scope}
\end{tikzpicture}
.
\end{equation*}

\begin{equation*}\tag{SP4.4}\label{eq:SP4-4}
\begin{array}{c@{\quad=\quad}c@{\qquad\begin{tikzpicture}[sp4rel,x=.72cm,y=.72cm]\path[use as bounding box] (0,-.20) rectangle (0,3.05);\node[inner sep=0pt] at (0,0) {$,$};\end{tikzpicture}\qquad}c@{\quad=\quad}c}
\begin{tikzpicture}[sp4rel,x=.72cm,y=.72cm]
\path[use as bounding box] (-.25,-.20) rectangle (1.35,3.05);
\draw[sp4one] (0,0)
-- (.55,.70) -- (1.10,1.40)
-- (.55,2.10) -- (0,2.80);
\draw[sp4two] (1.10,0)
-- (.55,.70) -- (0,1.40)
-- (.55,2.10) -- (1.10,2.80);
\end{tikzpicture}
&
\begin{tikzpicture}[sp4rel,x=.72cm,y=.72cm]
\path[use as bounding box] (-.25,-.20) rectangle (1.35,3.05);
\draw[sp4one] (0,0) -- (0,2.80);
\draw[sp4two] (1.10,0) -- (1.10,2.80);
\end{tikzpicture}
&
\begin{tikzpicture}[sp4rel,x=.72cm,y=.72cm]
\path[use as bounding box] (-.25,-.20) rectangle (1.35,3.05);
\draw[sp4two] (0,0)
-- (.55,.70) -- (1.10,1.40)
-- (.55,2.10) -- (0,2.80);
\draw[sp4one] (1.10,0)
-- (.55,.70) -- (0,1.40)
-- (.55,2.10) -- (1.10,2.80);
\end{tikzpicture}
&
\begin{tikzpicture}[sp4rel,x=.72cm,y=.72cm]
\path[use as bounding box] (-.25,-.20) rectangle (1.35,3.05);
\draw[sp4two] (0,0) -- (0,2.80);
\draw[sp4one] (1.10,0) -- (1.10,2.80);
\end{tikzpicture}
\end{array}
.
\end{equation*}

\begin{equation*}\tag{SP4.5}\label{eq:SP4-5}
\begin{array}{c@{\quad=\quad}c@{\qquad\qquad}c@{\quad=\quad}c}
\begin{tikzpicture}[sp4rel,x=.86cm,y=.86cm]
\path[use as bounding box] (-.25,-.20) rectangle (1.25,1.25);
\draw[sp4two] (0,0) -- (.33,.50);
\draw[sp4one] (.33,.50) -- (0,1);
\draw[sp4two] (1,0) -- (.67,.50);
\draw[sp4one] (.67,.50) -- (1,1);
\draw[sp4one] (.33,.50) -- (.67,.50);
\end{tikzpicture}
&
\begin{tikzpicture}[sp4rel,x=.86cm,y=.86cm]
\path[use as bounding box] (-.25,-.20) rectangle (1.25,1.25);
\draw[sp4two] (0,0)
to[out=45,in=135] (1,0);
\draw[sp4one] (0,1)
to[out=315,in=225] (1,1);
\end{tikzpicture}
&
\begin{tikzpicture}[sp4rel,x=.86cm,y=.86cm]
\path[use as bounding box] (-.25,-.20) rectangle (1.25,1.25);
\draw[sp4two] (0,0) -- (.33,.50);
\draw[sp4one] (.33,.50) -- (0,1);
\draw[sp4one] (1,0) -- (.67,.50);
\draw[sp4two] (.67,.50) -- (1,1);
\draw[sp4one] (.33,.50) -- (.67,.50);
\end{tikzpicture}
&
\begin{tikzpicture}[sp4rel,x=.86cm,y=.86cm]
\path[use as bounding box] (-.25,-.20) rectangle (1.25,1.25);
\draw[sp4two] (0,0) -- (.50,.33);
\draw[sp4one] (.50,.33) -- (1,0);
\draw[sp4one] (0,1) -- (.50,.67);
\draw[sp4two] (.50,.67) -- (1,1);
\draw[sp4one] (.50,.33) -- (.50,.67);
\end{tikzpicture}
\end{array}
.
\end{equation*}

\begin{equation*}\tag{SP4.6}\label{eq:SP4-6}
\begin{array}{c@{\quad=\quad}c@{\quad-\quad}c@{\quad-\quad}c}
\begin{tikzpicture}[sp4rel,x=.82cm,y=.82cm]
\path[use as bounding box] (-.75,-.20) rectangle (.75,2.25);
\draw[sp4one] (-.50,0) -- (-.50,.50);
\draw[sp4one] ( .50,0) -- ( .50,.50);
\draw[densely dashed,line width=1.1pt] (-.50,.50) -- (.50,.50);
\draw[sp4two] (-.50,.50) -- (-.50,1.00);
\draw[sp4two] ( .50,.50) -- ( .50,1.00);
\draw[sp4two] (-.50,1.00) -- (-.50,1.50);
\draw[sp4two] ( .50,1.00) -- ( .50,1.50);
\draw[densely dashed,line width=1.1pt] (-.50,1.50) -- (.50,1.50);
\draw[sp4one] (-.50,1.50) -- (-.50,2.00);
\draw[sp4one] ( .50,1.50) -- ( .50,2.00);
\end{tikzpicture}
&
\begin{tikzpicture}[sp4rel,x=.82cm,y=.82cm]
\path[use as bounding box] (-.75,-.20) rectangle (.75,2.25);
\draw[sp4one] (-.50,0) -- (-.50,2.00);
\draw[sp4one] ( .50,0) -- ( .50,2.00);
\end{tikzpicture}
&
\begin{tikzpicture}[sp4rel,x=.82cm,y=.82cm]
\path[use as bounding box] (-.75,-.20) rectangle (.75,2.25);
\draw[sp4one] (-.50,0)
to[out=90,in=90] (.50,0);
\draw[sp4one] (-.50,2.00)
to[out=270,in=270] (.50,2.00);
\end{tikzpicture}
&
\begin{tikzpicture}[sp4rel,x=.82cm,y=.82cm]
\path[use as bounding box] (-.75,-.20) rectangle (.75,2.25);
\draw[sp4one] (-.50,0) -- (0,.70);
\draw[sp4one] ( .50,0) -- (0,.70);
\draw[sp4two] (0,.70) -- (0,1.30);
\draw[sp4one] (-.50,2.00) -- (0,1.30);
\draw[sp4one] ( .50,2.00) -- (0,1.30);
\end{tikzpicture}
\end{array}
.
\end{equation*}

\begin{equation*}\tag{SP4.7}\label{eq:SP4-7}
\begin{array}{c@{\;=\;}c@{\;=\;}c@{\;=\;}c@{\;=\;}c@{\;=\;}c@{\;=\;}c}
\begin{tikzpicture}[sp4rel,x=.50cm,y=.50cm]
\path[use as bounding box] (-1.00,-1.12) rectangle (1.65,2.12);
\coordinate (v) at (0,.35);
\coordinate (l) at (-.55,1.05);
\coordinate (r) at (.55,1.05);
\coordinate (fr) at (1.30,1.05);
\draw[sp4one] (l) -- (v);
\draw[sp4one] (l) -- (-.55,1.70);
\draw[sp4one] (r) -- (v);
\draw[sp4two] (0,-.75) -- (v);
\draw[sp4one] (r) to[out=90,in=180] (.92,1.62)
to[out=0,in=90] (fr);
\draw[sp4one] (fr) -- (1.30,-.75);
\end{tikzpicture}
&
\begin{tikzpicture}[sp4rel,x=.50cm,y=.50cm]
\path[use as bounding box] (-1.00,-1.12) rectangle (1.65,2.12);
\coordinate (v) at (0,.35);
\coordinate (l) at (-.55,1.05);
\coordinate (r) at (.55,1.05);
\coordinate (fr) at (1.30,1.05);
\draw[sp4one] (l) -- (v);
\draw[sp4one] (l) -- (-.55,1.70);
\draw[sp4two] (r) -- (v);
\draw[sp4one] (0,-.75) -- (v);
\draw[sp4two] (r) to[out=90,in=180] (.92,1.62)
to[out=0,in=90] (fr);
\draw[sp4two] (fr) -- (1.30,-.75);
\end{tikzpicture}
&
\begin{tikzpicture}[sp4rel,x=.50cm,y=.50cm]
\path[use as bounding box] (-1.65,-1.12) rectangle (1.00,2.12);
\coordinate (v) at (0,.35);
\coordinate (l) at (-.55,1.05);
\coordinate (r) at (.55,1.05);
\coordinate (fl) at (-1.30,1.05);
\draw[sp4one] (r) -- (v);
\draw[sp4one] (r) -- (.55,1.70);
\draw[sp4one] (l) -- (v);
\draw[sp4two] (0,-.75) -- (v);
\draw[sp4one] (l) to[out=90,in=0] (-.92,1.62)
to[out=180,in=90] (fl);
\draw[sp4one] (fl) -- (-1.30,-.75);
\end{tikzpicture}
&
\begin{tikzpicture}[sp4rel,x=.50cm,y=.50cm]
\path[use as bounding box] (-1.65,-1.12) rectangle (1.00,2.12);
\coordinate (v) at (0,.35);
\coordinate (l) at (-.55,1.05);
\coordinate (r) at (.55,1.05);
\coordinate (fl) at (-1.30,1.05);
\draw[sp4one] (r) -- (v);
\draw[sp4one] (r) -- (.55,1.70);
\draw[sp4two] (l) -- (v);
\draw[sp4one] (0,-.75) -- (v);
\draw[sp4two] (l) to[out=90,in=0] (-.92,1.62)
to[out=180,in=90] (fl);
\draw[sp4two] (fl) -- (-1.30,-.75);
\end{tikzpicture}
&
0
&
\begin{tikzpicture}[sp4rel,x=.50cm,y=.50cm]
\path[use as bounding box] (-.40,-1.12) rectangle (2.35,2.12);
\draw[sp4two] (0,1.70) -- (0,.25)
to[out=270,in=180] (.50,-.25)
to[out=0,in=270] (1.00,.25)
-- (1.00,.75)
to[out=90,in=180] (1.50,1.25)
to[out=0,in=90] (2.00,.75)
-- (2.00,-.75);
\end{tikzpicture}
&
\begin{tikzpicture}[sp4rel,x=.50cm,y=.50cm]
\path[use as bounding box] (-.40,-1.12) rectangle (2.35,2.12);
\draw[sp4one] (0,1.70) -- (0,.25)
to[out=270,in=180] (.50,-.25)
to[out=0,in=270] (1.00,.25)
-- (1.00,.75)
to[out=90,in=180] (1.50,1.25)
to[out=0,in=90] (2.00,.75)
-- (2.00,-.75);
\end{tikzpicture}
\end{array}
.
\end{equation*}

\begin{equation*}\tag{SP4.8}\label{eq:SP4-8}
\begin{gathered}
\begin{tikzpicture}[sp4rel,x=.58cm,y=.58cm]
\path[use as bounding box] (-.25,-.15) rectangle (2.25,2.55);
\draw[sp4one] (0,0) -- (0,.50);
\draw[sp4one] (1,0) -- (1,.50);
\draw[densely dashed,line width=1.1pt] (0,.50) -- (1,.50);
\draw[sp4two] (0,.50) -- (0,2.35);
\draw[sp4two] (1,.50) -- (1,1.65);
\draw[sp4two] (2,0) -- (2,1.65);
\draw[sp4two] (1,1.65)
to[out=90,in=180] (1.50,2.35)
to[out=0,in=90] (2,1.65);
\end{tikzpicture}
\;=0=\;
\begin{tikzpicture}[sp4rel,x=.58cm,y=.58cm]
\path[use as bounding box] (-.25,-.15) rectangle (2.25,2.55);
\draw[sp4two] (0,0) -- (0,1.65);
\draw[sp4one] (1,0) -- (1,.50);
\draw[sp4one] (2,0) -- (2,.50);
\draw[densely dashed,line width=1.1pt] (1,.50) -- (2,.50);
\draw[sp4two] (1,.50) -- (1,1.65);
\draw[sp4two] (2,.50) -- (2,2.35);
\draw[sp4two] (0,1.65)
to[out=90,in=180] (.50,2.35)
to[out=0,in=90] (1,1.65);
\end{tikzpicture}
\\[2.55em]
\begin{array}{c@{\quad=\quad}c@{\qquad\begin{tikzpicture}[sp4rel,x=.52cm,y=.52cm]\path[use as bounding box] (0,-.15) rectangle (0,2.55);\node[inner sep=0pt] at (0,0) {$,$};\end{tikzpicture}\qquad}c@{\quad=\quad}c}
\begin{tikzpicture}[sp4rel,x=.52cm,y=.52cm]
\path[use as bounding box] (-.75,-.15) rectangle (2.65,2.55);
\draw[sp4one] (0,0) -- (0,.55);
\draw[sp4one] (1,0) -- (1,.55);
\draw[densely dashed,line width=1.1pt] (0,.55) -- (1,.55);
\draw[sp4two] (0,.55) -- (0,2.35);
\draw[sp4two] (1,.55) -- (1,1.25);
\draw[sp4one] (2,0) -- (2,1.25);
\draw[sp4two] (1,1.25) -- (1.50,1.85);
\draw[sp4one] (2,1.25) -- (1.50,1.85);
\draw[sp4one] (1.50,1.85) -- (1.50,2.35);
\end{tikzpicture}
&
\begin{tikzpicture}[sp4rel,x=.52cm,y=.52cm]
\path[use as bounding box] (-.75,-.15) rectangle (2.65,2.55);
\draw[sp4one] (1,0) to[out=90,in=90] (2,0);
\draw[sp4one] (0,0) -- (0,1.35);
\draw[sp4two] (-.50,2.35) -- (0,1.35);
\draw[sp4one] (.50,2.35) -- (0,1.35);
\end{tikzpicture}
&
\begin{tikzpicture}[sp4rel,x=.52cm,y=.52cm]
\path[use as bounding box] (-.25,-.15) rectangle (2.65,2.55);
\draw[sp4one] (0,0) -- (0,1.25);
\draw[sp4one] (1,0) -- (1,.55);
\draw[sp4one] (2,0) -- (2,.55);
\draw[densely dashed,line width=1.1pt] (1,.55) -- (2,.55);
\draw[sp4two] (1,.55) -- (1,1.25);
\draw[sp4two] (2,.55) -- (2,2.35);
\draw[sp4one] (0,1.25) -- (.50,1.85);
\draw[sp4two] (1,1.25) -- (.50,1.85);
\draw[sp4one] (.50,1.85) -- (.50,2.35);
\end{tikzpicture}
&
\begin{tikzpicture}[sp4rel,x=.52cm,y=.52cm]
\path[use as bounding box] (-.25,-.15) rectangle (2.65,2.55);
\draw[sp4one] (0,0) to[out=90,in=90] (1,0);
\draw[sp4one] (2,0) -- (2,1.35);
\draw[sp4one] (1.50,2.35) -- (2,1.35);
\draw[sp4two] (2.50,2.35) -- (2,1.35);
\end{tikzpicture}
\end{array}
\end{gathered}
.
\end{equation*}

\begin{equation*}\tag{SP4.9}\label{eq:SP4-9}
\begin{array}{c@{\quad=\quad}c@{\quad=\quad}c@{\quad=\quad}c}
\begin{tikzpicture}[sp4rel,x=.86cm,y=.86cm]
\path[use as bounding box] (-.85,-.20) rectangle (.85,2.05);
\draw[sp4two] (0,0) -- (0,.70);
\draw[sp4one] (0,.70) -- (-.55,1.30);
\draw[sp4one] (0,.70) -- (.55,1.30);
\draw[sp4one] (-.55,1.30)
to[out=90,in=180] (0,1.85)
to[out=0,in=90] (.55,1.30);
\end{tikzpicture}
&
\begin{tikzpicture}[sp4rel,x=.86cm,y=.86cm]
\path[use as bounding box] (-.85,-.20) rectangle (.85,2.05);
\draw[sp4two] (-.55,0) -- (-.55,.65);
\draw[sp4two] ( .55,0) -- ( .55,.65);
\draw[densely dashed,line width=1.1pt] (-.55,.65) -- (.55,.65);
\draw[sp4one] (-.55,.65) -- (-.55,1.30);
\draw[sp4one] ( .55,.65) -- ( .55,1.30);
\draw[sp4one] (-.55,1.30)
to[out=90,in=180] (0,1.85)
to[out=0,in=90] (.55,1.30);
\end{tikzpicture}
&
0
&
\begin{tikzpicture}[sp4rel,x=.86cm,y=.86cm]
\path[use as bounding box] (-.85,-.20) rectangle (.85,2.05);
\draw[sp4two] (-.55,.45)
to[out=270,in=180] (0,0)
to[out=0,in=270] (.55,.45);
\draw[sp4two] (-.55,.45) -- (-.55,1.20);
\draw[sp4two] ( .55,.45) -- ( .55,1.20);
\draw[densely dashed,line width=1.1pt] (-.55,1.20) -- (.55,1.20);
\draw[sp4one] (-.55,1.20) -- (-.55,1.85);
\draw[sp4one] ( .55,1.20) -- ( .55,1.85);
\end{tikzpicture}
\end{array}
.
\end{equation*}

\begin{equation*}\tag{SP4.10}\label{eq:SP4-10}
\begin{array}{c@{\quad=\quad}c@{\quad+\quad}c@{\quad+\quad}c}
\begin{tikzpicture}[sp4rel,x=.55cm,y=.55cm]
\path[use as bounding box] (-.25,-.15) rectangle (2.25,2.85);
\draw[sp4one] (0,0) -- (0,1.45) -- (.50,2.05) -- (.50,2.65);
\draw[sp4two] (2,0) -- (1.50,.72) -- (1.00,1.45) -- (.50,2.05);
\draw[sp4one] (1,0) -- (1.50,.72) -- (2.00,1.45) -- (2.00,2.65);
\end{tikzpicture}
&
\begin{tikzpicture}[sp4rel,x=.55cm,y=.55cm]
\path[use as bounding box] (-.25,-.15) rectangle (2.25,2.85);
\draw[sp4one] (0,0) -- (.50,.60);
\draw[sp4one] (1,0) -- (.50,.60);
\draw[sp4two] (.50,.60) -- (.50,1.15);
\draw[sp4two] (2,0) -- (2,1.15);
\draw[sp4two] (.50,1.15) to[out=90,in=90] (2,1.15);
\draw[sp4one] (.50,2.65) to[out=270,in=270] (2,2.65);
\end{tikzpicture}
&
\begin{tikzpicture}[sp4rel,x=.55cm,y=.55cm]
\path[use as bounding box] (-.25,-.15) rectangle (2.25,2.85);
\draw[sp4one] (0,0) -- (.50,.60);
\draw[sp4one] (1,0) -- (.50,.60);
\draw[sp4two] (.50,.60) -- (.50,1.35);
\draw[sp4two] (2,0) -- (2,1.35);
\draw[densely dashed,line width=1.1pt] (.50,1.35) -- (2,1.35);
\draw[sp4one] (.50,1.35) -- (.50,2.65);
\draw[sp4one] (2,1.35) -- (2,2.65);
\end{tikzpicture}
&
\begin{tikzpicture}[sp4rel,x=.55cm,y=.55cm]
\path[use as bounding box] (-.25,-.15) rectangle (2.25,2.85);
\draw[sp4one] (0,0) to[out=90,in=90] (1,0);
\draw[sp4two] (1.50,0) -- (1.50,1.55);
\draw[sp4one] (.75,2.65) -- (1.50,1.55);
\draw[sp4one] (2.25,2.65) -- (1.50,1.55);
\end{tikzpicture}
\end{array}
.
\end{equation*}

\begin{equation*}\tag{SP4.11}\label{eq:SP4-11}
\begin{array}{c@{\quad=\quad}c@{\quad-\quad}c@{\quad+\quad}c}
\begin{tikzpicture}[sp4rel,x=.52cm,y=.52cm]
\path[use as bounding box] (-.25,-.15) rectangle (2.25,2.85);
\draw[sp4one] (0,0) -- (0,1.25);
\draw[sp4two] (1,0) -- (2,1.25) -- (2,2.65);
\draw[sp4one] (2,0) -- (1,1.25);
\draw[sp4one] (0,1.25) -- (.50,1.85);
\draw[sp4one] (1,1.25) -- (.50,1.85);
\draw[sp4two] (.50,1.85) -- (.50,2.65);
\end{tikzpicture}
&
\begin{tikzpicture}[sp4rel,x=.52cm,y=.52cm]
\path[use as bounding box] (-.25,-.15) rectangle (2.25,2.85);
\draw[sp4one] (0,0) -- (1,1.25);
\draw[sp4two] (1,0) -- (0,1.25) -- (0,2.65);
\draw[sp4one] (2,0) -- (2,1.25);
\draw[sp4one] (1,1.25) -- (1.50,1.85);
\draw[sp4one] (2,1.25) -- (1.50,1.85);
\draw[sp4two] (1.50,1.85) -- (1.50,2.65);
\end{tikzpicture}
&
\begin{tikzpicture}[sp4rel,x=.52cm,y=.52cm]
\path[use as bounding box] (-.25,-.15) rectangle (2.25,2.85);
\draw[sp4one] (0,0) -- (0,1.35);
\draw[sp4two] (1,0) -- (1.50,.72);
\draw[sp4one] (2,0) -- (1.50,.72);
\draw[sp4one] (1.50,.72) -- (1.50,1.35);
\draw[densely dashed,line width=1.1pt] (0,1.35) -- (1.50,1.35);
\draw[sp4two] (0,1.35) -- (0,2.65);
\draw[sp4two] (1.50,1.35) -- (1.50,2.65);
\end{tikzpicture}
&
\begin{tikzpicture}[sp4rel,x=.52cm,y=.52cm]
\path[use as bounding box] (-.25,-.15) rectangle (2.25,2.85);
\draw[sp4one] (0,0) -- (.50,.72);
\draw[sp4two] (1,0) -- (.50,.72);
\draw[sp4one] (.50,.72) -- (.50,1.35);
\draw[sp4one] (2,0) -- (2,1.35);
\draw[densely dashed,line width=1.1pt] (.50,1.35) -- (2,1.35);
\draw[sp4two] (.50,1.35) -- (.50,2.65);
\draw[sp4two] (2,1.35) -- (2,2.65);
\end{tikzpicture}
\end{array}
.
\end{equation*}

\begin{equation*}\tag{SP4.12}\label{eq:SP4-12}
\begin{array}{c@{\quad=\quad}c@{\quad}c@{\quad}c@{\quad=\quad}c}
\begin{tikzpicture}[sp4rel,x=.62cm,y=.62cm]
\path[use as bounding box] (-.25,-.15) rectangle (2.25,2.15);
\draw[sp4one] (0,0) -- (0,1.45);
\draw[sp4two] (1,0) -- (2,1.45) -- (2,1.95);
\draw[sp4one] (2,0) -- (1,1.45);
\draw[sp4one] (0,1.45)
to[out=90,in=180] (.50,1.95)
to[out=0,in=90] (1,1.45);
\end{tikzpicture}
&
\begin{tikzpicture}[sp4rel,x=.62cm,y=.62cm]
\path[use as bounding box] (-.25,-.15) rectangle (2.25,2.15);
\draw[sp4one] (0,0) -- (.50,.60);
\draw[sp4two] (1,0) -- (.50,.60);
\draw[sp4one] (.50,.60) -- (.50,.95);
\draw[sp4one] (2,0) -- (2,.95);
\draw[sp4one] (.50,.95) -- (1.25,1.55);
\draw[sp4one] (2,.95) -- (1.25,1.55);
\draw[sp4two] (1.25,1.55) -- (1.25,1.95);
\end{tikzpicture}
&
\begin{tikzpicture}[sp4rel,x=.62cm,y=.62cm]
\path[use as bounding box] (0,-.15) rectangle (0,2.15);
\node[inner sep=0pt] at (0,0) {$,$};
\end{tikzpicture}
&
\begin{tikzpicture}[sp4rel,x=.62cm,y=.62cm]
\path[use as bounding box] (-.25,-.15) rectangle (2.25,2.15);
\draw[sp4one] (0,0) -- (0,.95);
\draw[sp4two] (1,0) -- (1.50,.60);
\draw[sp4one] (2,0) -- (1.50,.60);
\draw[sp4one] (1.50,.60) -- (1.50,.95);
\draw[sp4one] (0,.95) -- (.75,1.55);
\draw[sp4one] (1.50,.95) -- (.75,1.55);
\draw[sp4two] (.75,1.55) -- (.75,1.95);
\end{tikzpicture}
&
\begin{tikzpicture}[sp4rel,x=.62cm,y=.62cm]
\path[use as bounding box] (-.25,-.15) rectangle (2.25,2.15);
\draw[sp4one] (0,0) -- (1,1.45);
\draw[sp4two] (1,0) -- (0,1.45) -- (0,1.95);
\draw[sp4one] (2,0) -- (2,1.45);
\draw[sp4one] (1,1.45)
to[out=90,in=180] (1.50,1.95)
to[out=0,in=90] (2,1.45);
\end{tikzpicture}
\end{array}
.
\end{equation*}

\begin{equation*}\tag{SP4.13}\label{eq:SP4-13}
\begin{array}{c@{\quad=\quad}c}
\begin{tikzpicture}[sp4rel,x=.62cm,y=.62cm]
\path[use as bounding box] (-.25,-.15) rectangle (2.25,2.15);
\draw[sp4two] (0,0) -- (0,1.45);
\draw[sp4one] (1,0) -- (2,1.45) -- (2,1.95);
\draw[sp4two] (2,0) -- (1,1.45);
\draw[sp4two] (0,1.45)
to[out=90,in=180] (.50,1.95)
to[out=0,in=90] (1,1.45);
\end{tikzpicture}
&
\begin{tikzpicture}[sp4rel,x=.62cm,y=.62cm]
\path[use as bounding box] (-.25,-.15) rectangle (2.25,2.15);
\draw[sp4two] (0,0) -- (.50,.60);
\draw[sp4one] (1,0) -- (.50,.60);
\draw[sp4one] (.50,.60) -- (.50,.95);
\draw[sp4two] (2,0) -- (2,.95);
\draw[sp4one] (.50,.95) -- (1.25,1.55);
\draw[sp4two] (2,.95) -- (1.25,1.55);
\draw[sp4one] (1.25,1.55) -- (1.25,1.95);
\end{tikzpicture}
\end{array}
.
\end{equation*}

\begin{equation*}\tag{SP4.14}\label{eq:SP4-14}
\begin{aligned}
\begin{tikzpicture}[sp4rel,x=.40cm,y=.40cm]
\path[use as bounding box] (-1.10,-.15) rectangle (2.75,3.65);
\draw[sp4two] (0,0) -- (0,1.20) -- (0,1.65);
\draw[sp4one] (1,0) -- (2,1.20) -- (2,3.45);
\draw[sp4two] (2,0) -- (1,1.20) -- (1,1.65);
\draw[sp4two] (0,1.65) -- (0,2.05);
\draw[sp4two] (1,1.65) -- (1,2.05);
\draw[densely dashed,line width=1.1pt] (0,2.05) -- (1,2.05);
\draw[sp4one] (0,2.05) -- (0,3.45);
\draw[sp4one] (1,2.05) -- (1,3.45);
\end{tikzpicture}
&=\quad
\begin{tikzpicture}[sp4rel,x=.40cm,y=.40cm]
\path[use as bounding box] (-1.10,-.15) rectangle (2.75,3.65);
\draw[sp4two] (0,0) -- (1,1.05) -- (1,1.35);
\draw[sp4one] (1,0) -- (0,1.05) -- (0,3.45);
\draw[sp4two] (2,0) -- (2,1.35);
\draw[sp4two] (1,1.35) to[out=90,in=90] (2,1.35);
\draw[sp4one] (1,3.45) to[out=270,in=270] (2,3.45);
\end{tikzpicture}
\quad+\quad
\begin{tikzpicture}[sp4rel,x=.40cm,y=.40cm]
\path[use as bounding box] (-1.10,-.15) rectangle (2.75,3.65);
\draw[sp4two] (0,0) -- (1,.85) -- (1,1.15);
\draw[sp4one] (1,0) -- (0,.85) -- (0,1.65);
\draw[sp4two] (2,0) -- (2,1.15);
\draw[sp4two] (1,1.15) to[out=90,in=90] (2,1.15);
\draw[sp4two] (0,1.65) -- (-.50,2.25);
\draw[sp4one] (0,1.65) -- (.50,2.25);
\draw[sp4one] (.50,2.25) -- (.50,3.45);
\draw[sp4two] (-.50,2.25) -- (-.50,2.70);
\draw[sp4one] (-.50,2.70) -- (-1,3.25);
\draw[sp4one] (-.50,2.70) -- (0,3.25);
\draw[sp4one] (-1,3.25) -- (-1,3.45);
\draw[sp4one] (0,3.25) -- (0,3.45);
\end{tikzpicture}
\quad-\quad
\begin{tikzpicture}[sp4rel,x=.40cm,y=.40cm]
\path[use as bounding box] (-1.10,-.15) rectangle (2.75,3.65);
\draw[sp4two] (0,0) -- (0,1.65);
\draw[sp4one] (1,0) -- (1.50,.65);
\draw[sp4two] (2,0) -- (1.50,.65);
\draw[sp4one] (1.50,.65) -- (1.50,3.45);
\draw[sp4one] (0,1.65) -- (-.50,2.20);
\draw[sp4one] (0,1.65) -- (.50,2.20);
\draw[sp4one] (-.50,2.20) -- (-.50,3.45);
\draw[sp4one] (.50,2.20) -- (.50,3.45);
\end{tikzpicture}
\quad+\quad
\begin{tikzpicture}[sp4rel,x=.40cm,y=.40cm]
\path[use as bounding box] (-1.10,-.15) rectangle (2.75,3.65);
\draw[sp4two] (0,0) -- (1,1.05) -- (1,1.65);
\draw[sp4one] (1,0) -- (0,1.05) -- (0,3.45);
\draw[sp4two] (2,0) -- (2,1.65);
\draw[sp4two] (1,1.65) -- (1,2.05);
\draw[sp4two] (2,1.65) -- (2,2.05);
\draw[densely dashed,line width=1.1pt] (1,2.05) -- (2,2.05);
\draw[sp4one] (1,2.05) -- (1,3.45);
\draw[sp4one] (2,2.05) -- (2,3.45);
\end{tikzpicture}
\\[1.7em]
&+\quad
\begin{tikzpicture}[sp4rel,x=.40cm,y=.40cm]
\path[use as bounding box] (-1.10,-.15) rectangle (2.75,3.65);
\draw[sp4two] (0,0) -- (.50,.65);
\draw[sp4one] (1,0) -- (.50,.65);
\draw[sp4one] (.50,.65) -- (.50,3.45);
\draw[sp4two] (2,0) -- (2,1.65);
\draw[sp4one] (2,1.65) -- (1.50,2.20);
\draw[sp4one] (2,1.65) -- (2.50,2.20);
\draw[sp4one] (1.50,2.20) -- (1.50,3.45);
\draw[sp4one] (2.50,2.20) -- (2.50,3.45);
\end{tikzpicture}
\quad-\quad
\begin{tikzpicture}[sp4rel,x=.40cm,y=.40cm]
\path[use as bounding box] (-1.10,-.15) rectangle (2.75,3.65);
\draw[sp4two] (0,0) -- (.50,.60);
\draw[sp4one] (1,0) -- (.50,.60);
\draw[sp4one] (.50,.60) -- (.50,1.15);
\draw[sp4two] (2,0) -- (2,1.15);
\draw[sp4one] (.50,1.15) -- (1.25,1.70);
\draw[sp4two] (2,1.15) -- (1.25,1.70);
\draw[sp4one] (1.25,1.70) -- (1.25,2.10);
\draw[sp4two] (1.25,2.10) -- (.50,2.65);
\draw[sp4one] (1.25,2.10) -- (2,2.65);
\draw[sp4one] (2,2.65) -- (2,3.45);
\draw[sp4two] (.50,2.65) -- (.50,2.90);
\draw[sp4one] (.50,2.90) -- (0,3.45);
\draw[sp4one] (.50,2.90) -- (1,3.45);
\end{tikzpicture}
\quad-\quad
\begin{tikzpicture}[sp4rel,x=.40cm,y=.40cm]
\path[use as bounding box] (-1.10,-.15) rectangle (2.75,3.65);
\draw[sp4two] (0,0) -- (.50,.60);
\draw[sp4one] (1,0) -- (.50,.60);
\draw[sp4one] (.50,.60) -- (.50,1.15);
\draw[sp4two] (2,0) -- (2,1.15);
\draw[sp4one] (.50,1.15) -- (1.25,1.70);
\draw[sp4two] (2,1.15) -- (1.25,1.70);
\draw[sp4one] (1.25,1.70) -- (2,2.55);
\draw[sp4one] (2,2.55) -- (2,3.45);
\draw[sp4one] (0,3.45) to[out=270,in=270] (1,3.45);
\end{tikzpicture}
\end{aligned}
.
\end{equation*}

\begin{equation*}\tag{SP4.15}\label{eq:SP4-15}
\begin{array}{c@{\quad=\quad}c@{\quad+\quad}c}
\begin{tikzpicture}[sp4rel,x=.58cm,y=.58cm]
\path[use as bounding box] (-.25,-.15) rectangle (2.75,2.85);
\draw[sp4two] (0,0) -- (0,1.20);
\draw[sp4two] (1,0) -- (2,1.20) -- (2,2.65);
\draw[sp4one] (2,0) -- (1,1.20);
\draw[sp4two] (0,1.20) -- (.50,1.90);
\draw[sp4one] (1,1.20) -- (.50,1.90);
\draw[sp4one] (.50,1.90) -- (.50,2.65);
\end{tikzpicture}
&
\begin{tikzpicture}[sp4rel,x=.58cm,y=.58cm]
\path[use as bounding box] (-.25,-.15) rectangle (2.75,2.85);
\draw[sp4one] (1.50,2.65) -- (2,1.90);
\draw[sp4two] (2.50,2.65) -- (2,1.90);
\draw[sp4one] (2,1.90) -- (2,0);
\draw[sp4two] (0,0) to[out=90,in=90] (1,0);
\end{tikzpicture}
&
\begin{tikzpicture}[sp4rel,x=.58cm,y=.58cm]
\path[use as bounding box] (-.25,-.15) rectangle (2.75,2.85);
\draw[sp4one] (0,2.65) -- (0,1.30);
\draw[sp4two] (1.50,2.65) -- (1.50,1.90);
\draw[sp4one] (1.50,1.90) -- (1,1.30);
\draw[sp4one] (1.50,1.90) -- (2,1.30) -- (2,0);
\draw[sp4one] (0,1.30) -- (0,.90);
\draw[sp4one] (1,1.30) -- (1,.90);
\draw[densely dashed,line width=1.1pt] (0,.90) -- (1,.90);
\draw[sp4two] (0,.90) -- (0,0);
\draw[sp4two] (1,.90) -- (1,0);
\end{tikzpicture}
\end{array}
.
\end{equation*}

\begin{equation*}\tag{SP4.16}\label{eq:SP4-16}
\begin{array}{c@{\quad=\quad}c}
\begin{tikzpicture}[sp4rel,x=.58cm,y=.58cm]
\path[use as bounding box] (-.25,-.15) rectangle (2.25,2.95);
\draw[sp4two] (0,0) -- (0,1.05);
\draw[sp4two] (1,0) -- (2,1.05);
\draw[sp4one] (2,0) -- (1,1.05);
\draw[sp4two] (0,1.05) -- (1,2.10) -- (1,2.25);
\draw[sp4one] (1,1.05) -- (0,2.10) -- (0,2.75);
\draw[sp4two] (2,1.05) -- (2,2.25);
\draw[sp4two] (1,2.25)
to[out=90,in=180] (1.50,2.75)
to[out=0,in=90] (2,2.25);
\end{tikzpicture}
&
\begin{tikzpicture}[sp4rel,x=.58cm,y=.58cm]
\path[use as bounding box] (-.25,-.15) rectangle (2.25,2.95);
\draw[sp4two] (0,0) -- (0,.75);
\draw[sp4two] (1,0) -- (1,.75);
\draw[densely dashed,line width=1.1pt] (0,.75) -- (1,.75);
\draw[sp4one] (0,.75) -- (0,2.75);
\draw[sp4one] (1,.75) -- (1,2.25);
\draw[sp4one] (2,0) -- (2,2.25);
\draw[sp4one] (1,2.25)
to[out=90,in=180] (1.50,2.75)
to[out=0,in=90] (2,2.25);
\end{tikzpicture}
\end{array}
.
\end{equation*}

\begin{equation*}\tag{SP4.17}\label{eq:SP4-17}
\begin{aligned}
\begin{tikzpicture}[sp4rel,x=.54cm,y=.54cm]
\path[use as bounding box] (-.25,-.15) rectangle (2.25,3.05);
\draw[sp4one] (0,0) -- (0,1.05);
\draw[sp4one] (1,0) -- (2,1.05);
\draw[sp4two] (2,0) -- (1,1.05);
\draw[sp4one] (0,1.05) -- (1,1.82);
\draw[sp4two] (1,1.05) -- (0,1.82) -- (0,2.85);
\draw[sp4one] (2,1.05) -- (2,1.82);
\draw[sp4one] (1,1.82) -- (1.50,2.42);
\draw[sp4one] (2,1.82) -- (1.50,2.42);
\draw[sp4two] (1.50,2.42) -- (1.50,2.85);
\end{tikzpicture}
&=\quad
\begin{tikzpicture}[sp4rel,x=.54cm,y=.54cm]
\path[use as bounding box] (-.25,-.15) rectangle (2.25,3.05);
\draw[sp4one] (0,0) -- (0,1.55);
\draw[sp4one] (1,0) -- (1.50,.75);
\draw[sp4two] (2,0) -- (1.50,.75);
\draw[sp4one] (1.50,.75) -- (1.50,1.55);
\draw[densely dashed,line width=1.1pt] (0,1.55) -- (1.50,1.55);
\draw[sp4two] (0,1.55) -- (0,2.85);
\draw[sp4two] (1.50,1.55) -- (1.50,2.85);
\end{tikzpicture}
\quad+\quad
\begin{tikzpicture}[sp4rel,x=.54cm,y=.54cm]
\path[use as bounding box] (-.25,-.15) rectangle (2.25,3.05);
\draw[sp4one] (0,0) -- (.50,1.05);
\draw[sp4one] (1,0) -- (.50,1.05);
\draw[sp4two] (.50,1.05) -- (.50,2.85);
\draw[sp4two] (2,0) -- (2,2.85);
\end{tikzpicture}
\quad-\quad
\begin{tikzpicture}[sp4rel,x=.54cm,y=.54cm]
\path[use as bounding box] (-.25,-.15) rectangle (2.25,3.05);
\draw[sp4one] (0,0) -- (.50,.65);
\draw[sp4one] (1,0) -- (.50,.65);
\draw[sp4two] (.50,.65) -- (.50,1.30);
\draw[sp4two] (2,0) -- (2,1.30);
\draw[densely dashed,line width=1.1pt] (.50,1.30) -- (2,1.30);
\draw[sp4one] (.50,1.30) -- (.50,2.05);
\draw[sp4one] (2,1.30) -- (2,2.05);
\draw[densely dashed,line width=1.1pt] (.50,2.05) -- (2,2.05);
\draw[sp4two] (.50,2.05) -- (.50,2.85);
\draw[sp4two] (2,2.05) -- (2,2.85);
\end{tikzpicture}
\end{aligned}
.
\end{equation*}

\begin{equation*}\tag{SP4.18}\label{eq:SP4-18}
\begin{array}{c@{\quad}c@{\quad}c@{\quad}c@{\quad}c@{\quad}c@{\quad}c}
\begin{tikzpicture}[sp4rel,x=.46cm,y=.46cm]
\path[use as bounding box] (-.25,-.15) rectangle (2.25,3.35);
\draw[sp4two] (0,0) -- (0,1.05);
\draw[sp4two] (1,0) -- (2,1.05);
\draw[sp4one] (2,0) -- (1,1.05);
\draw[sp4two] (0,1.05) -- (1,2.05) -- (1,2.40);
\draw[sp4one] (1,1.05) -- (0,2.05) -- (0,3.15);
\draw[sp4two] (2,1.05) -- (2,2.40);
\draw[densely dashed,line width=1.1pt] (1,2.40) -- (2,2.40);
\draw[sp4one] (1,2.40) -- (1,3.15);
\draw[sp4one] (2,2.40) -- (2,3.15);
\end{tikzpicture}
&
=
&
\begin{tikzpicture}[sp4rel,x=.46cm,y=.46cm]
\path[use as bounding box] (-.25,-.15) rectangle (2.25,3.35);
\draw[sp4two] (0,0) to[out=90,in=90] (1,0);
\draw[sp4one] (0,3.15) to[out=270,in=270] (1,3.15);
\draw[sp4one] (2,0) -- (2,3.15);
\end{tikzpicture}
&
+
&
\begin{tikzpicture}[sp4rel,x=.46cm,y=.46cm]
\path[use as bounding box] (-.25,-.15) rectangle (2.25,3.35);
\draw[sp4two] (0,0) -- (0,1.25);
\draw[sp4two] (1,0) -- (1,1.25);
\draw[densely dashed,line width=1.1pt] (0,1.25) -- (1,1.25);
\draw[sp4one] (0,1.25) -- (0,3.15);
\draw[sp4one] (1,1.25) -- (1,3.15);
\draw[sp4one] (2,0) -- (2,3.15);
\end{tikzpicture}
&
-
&
\begin{tikzpicture}[sp4rel,x=.46cm,y=.46cm]
\path[use as bounding box] (-.25,-.15) rectangle (2.25,3.35);
\draw[sp4two] (0,0) -- (0,.70);
\draw[sp4two] (1,0) -- (1,.70);
\draw[densely dashed,line width=1.1pt] (0,.70) -- (1,.70);
\draw[sp4one] (0,.70) -- (0,3.15);
\draw[sp4one] (1,.70) -- (1,1.45);
\draw[sp4one] (2,0) -- (2,1.45);
\draw[sp4one] (1,1.45) -- (1.50,2.05);
\draw[sp4one] (2,1.45) -- (1.50,2.05);
\draw[sp4two] (1.50,2.05) -- (1.50,2.50);
\draw[sp4one] (1.50,2.50) -- (1,3.15);
\draw[sp4one] (1.50,2.50) -- (2,3.15);
\end{tikzpicture}
\\[1.8em]
&
+
&
\begin{tikzpicture}[sp4rel,x=.46cm,y=.46cm]
\path[use as bounding box] (-.85,-.15) rectangle (2.25,3.35);
\draw[sp4two] (1,0) -- (1.50,.70);
\draw[sp4one] (2,0) -- (1.50,.70);
\draw[sp4one] (1.50,.70) -- (1.50,3.15);
\draw[sp4two] (0,0) -- (0,1.55);
\draw[sp4one] (0,1.55) -- (-.55,2.25) -- (-.55,3.15);
\draw[sp4one] (0,1.55) -- (.55,2.25) -- (.55,3.15);
\end{tikzpicture}
&
-
&
\begin{tikzpicture}[sp4rel,x=.46cm,y=.46cm]
\path[use as bounding box] (-.25,-.15) rectangle (2.25,3.35);
\draw[sp4two] (0,0) -- (0,.70);
\draw[sp4two] (1,0) -- (1,.70);
\draw[densely dashed,line width=1.1pt] (0,.70) -- (1,.70);
\draw[sp4one] (0,.70) -- (0,3.15);
\draw[sp4one] (1,.70) -- (1,1.25);
\draw[sp4one] (2,0) -- (2,1.25);
\draw[sp4one] (1,1.25) to[out=90,in=90] (2,1.25);
\draw[sp4one] (1,3.15) to[out=270,in=270] (2,3.15);
\end{tikzpicture}
&
-
&
\begin{tikzpicture}[sp4rel,x=.46cm,y=.46cm]
\path[use as bounding box] (-1.25,-.15) rectangle (2.25,3.35);
\draw[sp4two] (0,0) -- (0,.70);
\draw[sp4two] (1,0) -- (1,.70);
\draw[densely dashed,line width=1.1pt] (0,.70) -- (1,.70);
\draw[sp4one] (1,.70) -- (1,1.20);
\draw[sp4one] (2,0) -- (2,1.20);
\draw[sp4one] (1,1.20) to[out=90,in=90] (2,1.20);
\draw[sp4one] (0,.70) -- (0,1.62);
\draw[sp4two] (0,1.62) -- (-.55,2.22);
\draw[sp4one] (0,1.62) -- (.55,2.22) -- (.55,3.15);
\draw[sp4two] (-.55,2.22) -- (-.55,2.55);
\draw[sp4one] (-.55,2.55) -- (-1.10,3.15);
\draw[sp4one] (-.55,2.55) -- (0,3.15);
\end{tikzpicture}
\end{array}
.
\end{equation*}

\begin{equation*}\tag{SP4.19}\label{eq:SP4-19}
\begin{array}{c@{\quad=\quad}c@{\qquad\begin{tikzpicture}[sp4rel,x=.62cm,y=.62cm]\path[use as bounding box] (0,-.15) rectangle (0,3.00);\node[inner sep=0pt] at (0,0) {$,$};\end{tikzpicture}\qquad}c@{\quad=\quad}c}
\begin{tikzpicture}[sp4rel,x=.62cm,y=.62cm]
\path[use as bounding box] (-.25,-.15) rectangle (2.25,3.00);
\draw[sp4one] (1,0) -- (1.50,.70);
\draw[sp4one] (2,0) -- (1.50,.70);
\draw[sp4two] (1.50,.70) -- (1.50,1.35);
\draw[sp4one] (0,0) -- (0,1.35);
\draw[sp4one] (0,1.35) -- (.75,2.05);
\draw[sp4two] (1.50,1.35) -- (.75,2.05);
\draw[sp4one] (.75,2.05) -- (.75,2.85);
\end{tikzpicture}
&
\begin{tikzpicture}[sp4rel,x=.62cm,y=.62cm]
\path[use as bounding box] (-.25,-.15) rectangle (2.25,3.00);
\draw[sp4one] (0,0) -- (.50,.70);
\draw[sp4one] (1,0) -- (.50,.70);
\draw[sp4two] (.50,.70) -- (.50,1.35);
\draw[sp4one] (2,0) -- (2,1.35);
\draw[sp4two] (.50,1.35) -- (1.25,2.05);
\draw[sp4one] (2,1.35) -- (1.25,2.05);
\draw[sp4one] (1.25,2.05) -- (1.25,2.85);
\end{tikzpicture}
&
\begin{tikzpicture}[sp4rel,x=.62cm,y=.62cm]
\path[use as bounding box] (-.25,-.15) rectangle (2.25,3.00);
\draw[sp4one] (1,0) -- (1.50,1.20);
\draw[sp4two] (2,0) -- (1.50,1.20);
\draw[sp4one] (1.50,1.20) -- (1.50,2.20);
\draw[sp4one] (0,0) -- (0,2.20);
\draw[sp4one] (0,2.20)
to[out=90,in=180] (.75,2.85)
to[out=0,in=90] (1.50,2.20);
\end{tikzpicture}
&
\begin{tikzpicture}[sp4rel,x=.62cm,y=.62cm]
\path[use as bounding box] (-.25,-.15) rectangle (2.25,3.00);
\draw[sp4one] (0,0) -- (.50,1.20);
\draw[sp4one] (1,0) -- (.50,1.20);
\draw[sp4two] (.50,1.20) -- (.50,2.20);
\draw[sp4two] (2,0) -- (2,2.20);
\draw[sp4two] (.50,2.20)
to[out=90,in=180] (1.25,2.85)
to[out=0,in=90] (2,2.20);
\end{tikzpicture}
\\[1.5em]
\begin{tikzpicture}[sp4rel,x=.62cm,y=.62cm]
\path[use as bounding box] (-.25,-.15) rectangle (2.25,2.50);
\draw[sp4two] (1,0) -- (1.50,1.05);
\draw[sp4one] (2,0) -- (1.50,1.05);
\draw[sp4one] (1.50,1.05) -- (1.50,1.85);
\draw[sp4one] (0,0) -- (0,1.85);
\draw[sp4one] (0,1.85)
to[out=90,in=180] (.75,2.35)
to[out=0,in=90] (1.50,1.85);
\end{tikzpicture}
&
\begin{tikzpicture}[sp4rel,x=.62cm,y=.62cm]
\path[use as bounding box] (-.25,-.15) rectangle (2.25,2.50);
\draw[sp4one] (0,0) -- (.50,1.05);
\draw[sp4two] (1,0) -- (.50,1.05);
\draw[sp4one] (.50,1.05) -- (.50,1.85);
\draw[sp4one] (2,0) -- (2,1.85);
\draw[sp4one] (.50,1.85)
to[out=90,in=180] (1.25,2.35)
to[out=0,in=90] (2,1.85);
\end{tikzpicture}
&
\begin{tikzpicture}[sp4rel,x=.62cm,y=.62cm]
\path[use as bounding box] (-.25,-.15) rectangle (2.25,2.50);
\draw[sp4one] (1,0) -- (1.50,1.05);
\draw[sp4one] (2,0) -- (1.50,1.05);
\draw[sp4two] (1.50,1.05) -- (1.50,1.85);
\draw[sp4two] (0,0) -- (0,1.85);
\draw[sp4two] (0,1.85)
to[out=90,in=180] (.75,2.35)
to[out=0,in=90] (1.50,1.85);
\end{tikzpicture}
&
\begin{tikzpicture}[sp4rel,x=.62cm,y=.62cm]
\path[use as bounding box] (-.25,-.15) rectangle (2.25,2.50);
\draw[sp4two] (0,0) -- (.50,1.05);
\draw[sp4one] (1,0) -- (.50,1.05);
\draw[sp4one] (.50,1.05) -- (.50,1.85);
\draw[sp4one] (2,0) -- (2,1.85);
\draw[sp4one] (.50,1.85)
to[out=90,in=180] (1.25,2.35)
to[out=0,in=90] (2,1.85);
\end{tikzpicture}
\end{array}
.
\end{equation*}

\begin{equation*}\tag{SP4.20}\label{eq:SP4-20}
\begin{array}{c@{\quad=\quad}c@{\quad-\quad}c@{\quad-\quad}c}
\begin{tikzpicture}[sp4rel,x=.54cm,y=.54cm]
\path[use as bounding box] (-.25,-.15) rectangle (2.25,3.05);
\draw[sp4one] (1,0) -- (1.50,.70);
\draw[sp4one] (2,0) -- (1.50,.70);
\draw[sp4two] (1.50,.70) -- (1.50,1.45);
\draw[sp4two] (0,0) -- (0,1.45);
\draw[densely dashed,line width=1.1pt] (0,1.45) -- (1.50,1.45);
\draw[sp4one] (0,1.45) -- (0,2.85);
\draw[sp4one] (1.50,1.45) -- (1.50,2.85);
\end{tikzpicture}
&
\begin{tikzpicture}[sp4rel,x=.54cm,y=.54cm]
\path[use as bounding box] (-.25,-.15) rectangle (2.25,3.05);
\draw[sp4two] (0,0) -- (1,1.05);
\draw[sp4one] (1,0) -- (0,1.05) -- (0,2.85);
\draw[sp4one] (2,0) -- (2,1.05);
\draw[sp4two] (1,1.05) -- (1.50,1.75);
\draw[sp4one] (2,1.05) -- (1.50,1.75);
\draw[sp4one] (1.50,1.75) -- (1.50,2.85);
\end{tikzpicture}
&
\begin{tikzpicture}[sp4rel,x=.54cm,y=.54cm]
\path[use as bounding box] (-.65,-.15) rectangle (2.25,3.05);
\draw[sp4one] (1,0) to[out=90,in=90] (2,0);
\draw[sp4two] (0,0) -- (0,1.45);
\draw[sp4one] (0,1.45) -- (-.55,2.15) -- (-.55,2.85);
\draw[sp4one] (0,1.45) -- (.55,2.15) -- (.55,2.85);
\end{tikzpicture}
&
\begin{tikzpicture}[sp4rel,x=.54cm,y=.54cm]
\path[use as bounding box] (-.25,-.15) rectangle (2.25,3.05);
\draw[sp4two] (0,0) -- (.50,.70);
\draw[sp4one] (1,0) -- (.50,.70);
\draw[sp4one] (.50,.70) -- (.50,1.35);
\draw[sp4one] (2,0) -- (2,1.35);
\draw[sp4one] (.50,1.35) to[out=90,in=90] (2,1.35);
\draw[sp4one] (.50,2.85) to[out=270,in=270] (2,2.85);
\end{tikzpicture}
\end{array}
.
\end{equation*}

\begin{equation*}\tag{SP4.21}\label{eq:SP4-21}
\begin{tikzpicture}[sp4rel,x=.54cm,y=.54cm]
\path[use as bounding box] (-.25,-.15) rectangle (2.25,3.05);
\draw[sp4one] (0,0) -- (0,2.10);
\draw[sp4two] (1,0) -- (1,.85);
\draw[sp4two] (2,0) -- (2,.85);
\draw[densely dashed,line width=1.1pt] (1,.85) -- (2,.85);
\draw[sp4one] (1,.85) -- (1,2.10);
\draw[sp4one] (2,.85) -- (2,2.85);
\draw[sp4one] (0,2.10)
to[out=90,in=180] (.50,2.65)
to[out=0,in=90] (1,2.10);
\end{tikzpicture}
\quad=\quad
\begin{tikzpicture}[sp4rel,x=.54cm,y=.54cm]
\path[use as bounding box] (-.25,-.15) rectangle (2.25,3.05);
\draw[sp4one] (0,0) -- (.50,.85);
\draw[sp4two] (1,0) -- (.50,.85);
\draw[sp4one] (.50,.85) -- (.50,1.45);
\draw[sp4two] (2,0) -- (2,1.45);
\draw[sp4one] (.50,1.45) -- (1.25,2.10);
\draw[sp4two] (2,1.45) -- (1.25,2.10);
\draw[sp4one] (1.25,2.10) -- (1.25,2.85);
\end{tikzpicture}
.
\end{equation*}

\begin{equation*}\tag{SP4.22}\label{eq:SP4-22}
\begin{tikzpicture}[sp4rel,x=.54cm,y=.54cm]
\path[use as bounding box] (-.25,-.15) rectangle (2.25,3.05);
\draw[sp4one] (0,0) -- (0,1.20) -- (.50,1.80);
\draw[sp4two] (1,0) -- (1,.85);
\draw[sp4two] (2,0) -- (2,.85);
\draw[densely dashed,line width=1.1pt] (1,.85) -- (2,.85);
\draw[sp4one] (1,.85) -- (1,1.20) -- (.50,1.80);
\draw[sp4two] (.50,1.80) -- (.50,2.85);
\draw[sp4one] (2,.85) -- (2,2.85);
\end{tikzpicture}
\quad=\quad
\begin{tikzpicture}[sp4rel,x=.54cm,y=.54cm]
\path[use as bounding box] (-.25,-.15) rectangle (2.25,3.05);
\draw[sp4one] (0,0) -- (1,1.05);
\draw[sp4two] (1,0) -- (0,1.05) -- (0,2.85);
\draw[sp4one] (1,1.05) -- (1,1.35) -- (1.50,1.95);
\draw[sp4two] (2,0) -- (2,1.35) -- (1.50,1.95);
\draw[sp4one] (1.50,1.95) -- (1.50,2.85);
\end{tikzpicture}
\quad-\quad
\begin{tikzpicture}[sp4rel,x=.54cm,y=.54cm]
\path[use as bounding box] (-.70,-.15) rectangle (2.25,3.05);
\draw[sp4two] (1,0) to[out=90,in=90] (2,0);
\draw[sp4one] (0,0) -- (0,1.25);
\draw[sp4two] (0,1.25) -- (-.55,1.95) -- (-.55,2.85);
\draw[sp4one] (0,1.25) -- (.55,1.95) -- (.55,2.85);
\end{tikzpicture}
\qquad
\begin{tikzpicture}[sp4rel,x=.54cm,y=.54cm]
\path[use as bounding box] (0,-.15) rectangle (0,3.05);
\node[inner sep=0pt] at (0,0) {$,$};
\end{tikzpicture}
\qquad
\begin{tikzpicture}[sp4rel,x=.54cm,y=.54cm]
\path[use as bounding box] (-.25,-.15) rectangle (2.25,3.05);
\draw[sp4two] (0,0) -- (0,1.20) -- (.50,1.80);
\draw[sp4two] (1,0) -- (1,.85);
\draw[sp4two] (2,0) -- (2,.85);
\draw[densely dashed,line width=1.1pt] (1,.85) -- (2,.85);
\draw[sp4one] (1,.85) -- (1,1.20) -- (.50,1.80);
\draw[sp4one] (.50,1.80) -- (.50,2.85);
\draw[sp4one] (2,.85) -- (2,2.85);
\end{tikzpicture}
\quad=\quad
\begin{tikzpicture}[sp4rel,x=.54cm,y=.54cm]
\path[use as bounding box] (-.25,-.15) rectangle (2.25,3.05);
\draw[sp4two] (0,0) -- (0,.85);
\draw[sp4two] (1,0) -- (1,.85);
\draw[densely dashed,line width=1.1pt] (0,.85) -- (1,.85);
\draw[sp4one] (0,.85) -- (0,2.85);
\draw[sp4one] (1,.85) -- (1,1.20) -- (1.50,1.80);
\draw[sp4two] (2,0) -- (2,1.20) -- (1.50,1.80);
\draw[sp4one] (1.50,1.80) -- (1.50,2.85);
\end{tikzpicture}
.
\end{equation*}

\begin{equation*}\tag{SP4.23}\label{eq:SP4-23}
\begin{array}{c@{\quad=\quad}c@{\quad+\quad}c}
\begin{tikzpicture}[sp4rel,x=.54cm,y=.54cm]
\path[use as bounding box] (-.25,-.15) rectangle (2.40,3.75);
\draw[sp4two] (0,0) -- (0,1.35);
\draw[sp4two] (1.60,0) -- (1.60,.75);
\draw[sp4one] (1.60,.75) -- (1.05,1.35);
\draw[sp4one] (1.60,.75) -- (2.15,1.35);
\draw[sp4two] (0,1.35) -- (.525,1.85) -- (1.05,2.35);
\draw[sp4one] (1.05,1.35) -- (.525,1.85) -- (0,2.35);
\draw[sp4one] (0,2.35) -- (0,3.55);
\draw[sp4one] (2.15,1.35) -- (2.15,2.35);
\draw[sp4two] (1.05,2.35) -- (1.60,2.95);
\draw[sp4one] (2.15,2.35) -- (1.60,2.95);
\draw[sp4one] (1.60,2.95) -- (1.60,3.55);
\end{tikzpicture}
&
\begin{tikzpicture}[sp4rel,x=.54cm,y=.54cm]
\path[use as bounding box] (-.25,-.15) rectangle (1.45,3.75);
\draw[sp4two] (0,0) -- (0,1.55);
\draw[sp4two] (1.20,0) -- (1.20,1.55);
\draw[densely dashed,line width=1.1pt] (0,1.55) -- (1.20,1.55);
\draw[sp4one] (0,1.55) -- (0,3.55);
\draw[sp4one] (1.20,1.55) -- (1.20,3.55);
\end{tikzpicture}
&
\begin{tikzpicture}[sp4rel,x=.54cm,y=.54cm]
\path[use as bounding box] (-.25,-.15) rectangle (1.45,3.75);
\draw[sp4two] (0,0) to[out=90,in=90] (1.20,0);
\draw[sp4one] (0,3.55) to[out=270,in=270] (1.20,3.55);
\end{tikzpicture}
\end{array}
.
\end{equation*}

\begin{equation*}\tag{SP4.24}\label{eq:SP4-24}
\begin{tikzpicture}[sp4rel,x=.58cm,y=.58cm]
\path[use as bounding box] (-.25,-.15) rectangle (4.35,3.35);
\draw[sp4two] (0,0) -- (0,.85);
\draw[sp4two] (1,0) -- (1,.85);
\draw[sp4two] (2,0) -- (2,.85);
\draw[densely dashed,line width=1.1pt] (1,.85) -- (2,.85);
\draw[sp4one] (1,.85) -- (.50,1.40) -- (0,1.95);
\draw[sp4two] (0,.85) -- (.50,1.40) -- (1,1.95);
\draw[sp4one] (2,.85) -- (2,1.95);
\draw[sp4one] (0,1.95) -- (0,3.10);
\draw[sp4two] (1,1.95) -- (1.50,2.50);
\draw[sp4one] (2,1.95) -- (1.50,2.50);
\draw[sp4one] (1.50,2.50) -- (1.50,3.10);
\node at (3.10,1.55) {$=0$};
\end{tikzpicture}
.
\end{equation*}

\begin{equation*}\tag{SP4.25}\label{eq:SP4-25}
\begin{array}{c@{\quad=\quad}c@{\quad-\quad}c}
\begin{tikzpicture}[sp4rel,x=.58cm,y=.58cm]
  \path[use as bounding box] (-.25,-.15) rectangle (2.25,3.25);
  \draw[sp4one] (0,0) -- (0,.85);
  \draw[sp4one] (1,0) -- (1,.85);
  \draw[sp4one] (2,0) -- (2,.85);
  \draw[densely dashed,line width=1.1pt] (1,.85) -- (2,.85);
  \draw[sp4two] (1,.85) -- (.50,1.40) -- (0,1.95);
  \draw[sp4one] (0,.85) -- (.50,1.40) -- (1,1.95);
  \draw[sp4two] (2,.85) -- (2,1.95);
  \draw[sp4two] (0,1.95) -- (0,3.10);
  \draw[sp4one] (1,1.95) -- (1.50,2.50);
  \draw[sp4two] (2,1.95) -- (1.50,2.50);
  \draw[sp4one] (1.50,2.50) -- (1.50,3.10);
\end{tikzpicture}
&
\begin{tikzpicture}[sp4rel,x=.58cm,y=.58cm]
  \path[use as bounding box] (-.25,-.15) rectangle (2.25,3.25);
  \draw[sp4one] (0,0) -- (.50,1.40);
  \draw[sp4one] (1,0) -- (.50,1.40);
  \draw[sp4two] (.50,1.40) -- (.50,3.10);
  \draw[sp4one] (2,0) -- (2,3.10);
\end{tikzpicture}
&
\begin{tikzpicture}[sp4rel,x=.58cm,y=.58cm]
  \path[use as bounding box] (-.25,-.15) rectangle (2.25,3.25);
  \draw[sp4one] (0,0) -- (.50,0.70);
  \draw[sp4one] (1,0) -- (.50,0.70);
  \draw[sp4two] (.50,0.70) -- (.50,1.20);
  \draw[sp4two] (.50,1.20) -- (1.25,1.90);
  \draw[sp4one] (2,0) -- (2,1.20) -- (1.25,1.90);
  \draw[sp4one] (1.25,1.90) -- (1.25,2.40);
  \draw[sp4two] (1.25,2.40) -- (.50,3.10);
  \draw[sp4one] (1.25,2.40) -- (2,3.10);
\end{tikzpicture}
\end{array}
.
\end{equation*}

\end{Proposition}

\begin{proof}
This follows from \autoref{iso of categories} and routine computations (that the reader might want to do using the code in \cite{HT}). The objects labeled 1 and 2 correspond to the four-dimensional and five-dimensional fundamental representations, respectively. 
\end{proof}

\subsection{\texorpdfstring{$\mathrm{G}_2$}{G2}}

Here we have 19 non-identity atoms and 9 non-identity two-strand atoms. The relations are too numerous to list here; see \cite{HT} instead.

\subsection{\texorpdfstring{$\mathrm{SO}_3$}{SO3}}\label{S:SO3}

We now give an example where $\mathcal{C}$ is a proper subcategory of $\gcrys$.
We define 
$\mathrm{SO}_3$-Crys to be the subcategory of $\mathfrak{sl}_2$-Crys generated by the highest weight crystals whose cardinality is an odd number. Then $\Fund(\mathrm{SO}_3\text{-}\mathrm{Crys})$ is monoidally generated by the highest weight crystal $B(2\omega_1)$.

\tikzset{
so3rel/.style={baseline={([yshift=-0.5ex]current bounding box.center)},line cap=round,line join=round}
}
\newcommand{\soThreeGen}[1]{%
\begin{tikzpicture}[baseline={(0,0)},scale=0.95]
\path[use as bounding box] (-0.75,-0.60) rectangle (0.75,0.60);
#1%
\end{tikzpicture}%
}
\newcommand{\soThreeComma}{%
\begin{tikzpicture}[baseline={(0,0)},scale=0.95]
\path[use as bounding box] (-0.10,-0.60) rectangle (0.24,0.60);
\node[inner sep=0pt] at (0,-0.32) {$,$};
\end{tikzpicture}\mspace{16mu}%
}

\begin{Proposition}
The category $\Fund(\mathrm{SO}_3\text{-}\mathrm{Crys})$ is isomorphic to the $\kk$-linear monoidal category generated by a single object $B_1$ and generating morphisms
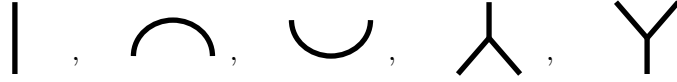
\[
\soThreeGen{\draw[usual] (0,-0.5) to (0,0.5);}
\soThreeComma
\soThreeGen{\draw[usual] (-0.55,-0.25) to[out=90,in=180] (0,0.25) to[out=0,in=90] (0.55,-0.25);}
\soThreeComma
\soThreeGen{\draw[usual] (-0.55,0.25) to[out=270,in=180] (0,-0.25) to[out=0,in=270] (0.55,0.25);}
\soThreeComma
\soThreeGen{\draw[usual] (0.433013,-0.5) to (0,0); \draw[usual] (0,0.5) to (0,0); \draw[usual] (-0.433013,-0.5) to (0,0);}
\soThreeComma
\soThreeGen{\draw[usual] (0.433013,0.5) to (0,0); \draw[usual] (0,-0.5) to (0,0); \draw[usual] (-0.433013,0.5) to (0,0);}
\]
with relations as follows, where we also impose the vertical reflections.
\begin{equation*}\tag{SO3.1}\label{eq:SO3-circle}
\begin{tikzpicture}[so3rel,scale=0.95]
\path[use as bounding box] (-0.70,-0.70) rectangle (2.65,0.70);
\draw[usual] (0,0) circle (0.48);
\node at (1.15,0) {$=$};
\node at (2.05,0) {$1$};
\end{tikzpicture}
.
\end{equation*}
\begin{equation*}\tag{SO3.2}\label{eq:SO3-digon}
\begin{tikzpicture}[so3rel,scale=0.78]
\path[use as bounding box] (-0.70,-0.10) rectangle (3.35,2.10);
\begin{scope}
\coordinate (b) at (0,0);
\coordinate (v) at (0,0.55);
\coordinate (l) at (-0.46,1.00);
\coordinate (r) at (0.46,1.00);
\coordinate (w) at (0,1.45);
\coordinate (t) at (0,2.00);
\draw[usual] (b) -- (v);
\draw[usual] (v) -- (l);
\draw[usual] (v) -- (r);
\draw[usual] (l) -- (w);
\draw[usual] (r) -- (w);
\draw[usual] (w) -- (t);
\end{scope}
\node at (1.35,1.00) {$=$};
\begin{scope}[xshift=2.55cm]
\draw[usual] (0,0) -- (0,2.00);
\end{scope}
\end{tikzpicture}
.
\end{equation*}
\begin{equation*}\tag{SO3.3}\label{eq:SO3-HI}
\begin{array}{c@{\quad=\quad}c@{\quad+\quad}c}
\begin{tikzpicture}[so3rel,scale=0.88]
\path[use as bounding box] (-0.18,-0.12) rectangle (1.18,1.12);
\draw[usual] (0,0) to (0.33,0.5) to (0,1);
\draw[usual] (1,0) to (0.67,0.5) to (1,1);
\draw[usual] (0.33,0.5) to (0.67,0.5);
\end{tikzpicture}
&
\begin{tikzpicture}[so3rel,scale=0.88]
\path[use as bounding box] (-0.18,-0.12) rectangle (1.18,1.12);
\draw[usual] (0,0) to (0.5,0.33) to (1,0);
\draw[usual] (0,1) to (0.5,0.67) to (1,1);
\draw[usual] (0.5,0.33) to (0.5,0.67);
\end{tikzpicture}
&
\begin{tikzpicture}[so3rel,scale=0.88]
\path[use as bounding box] (-0.18,-0.12) rectangle (1.18,1.12);
\draw[usual] (0,0) to[out=45,in=135] (1,0);
\draw[usual] (0,1) to[out=315,in=225] (1,1);
\end{tikzpicture}
\end{array}
.
\end{equation*}
\begin{equation*}\tag{SO3.4}\label{eq:SO3-zigzag}
\begin{array}{c@{\quad=\quad}c@{\quad=\quad}c}
\begin{tikzpicture}[so3rel,x=.78cm,y=.78cm]
\path[use as bounding box] (-.65,-.70) rectangle (1.45,1.00);
\draw[usual] (0.433,0.25) -- (0,0);
\draw[usual] (0,-0.50) -- (0,0);
\draw[usual] (-0.433,0.25) -- (0,0);
\draw[usual] (0.433,0.25)
to[out=90,in=180] (0.815,0.60)
to[out=0,in=90] (1.197,0.25);
\end{tikzpicture}
&
0
&
\begin{tikzpicture}[so3rel,x=.78cm,y=.78cm]
\path[use as bounding box] (-1.45,-.70) rectangle (.65,1.00);
\draw[usual] (0.433,0.25) -- (0,0);
\draw[usual] (0,-0.50) -- (0,0);
\draw[usual] (-0.433,0.25) -- (0,0);
\draw[usual] (-0.433,0.25)
to[out=90,in=0] (-0.815,0.60)
to[out=180,in=90] (-1.197,0.25);
\end{tikzpicture}
\end{array}
.
\end{equation*}
\begin{equation*}\tag{SO3.5}\label{eq:SO3-tadpole}
\begin{tikzpicture}[so3rel,scale=0.88]
\path[use as bounding box] (-0.65,-0.62) rectangle (1.80,0.92);
\draw[usual] (0.433013,0.25) to (0,0);
\draw[usual] (0,-0.5) to (0,0);
\draw[usual] (-0.433013,0.25) to (0,0);
\draw[usual] (-0.433013,0.25) to[out=90,in=180] (0,0.65) to[out=0,in=90] (0.433013,0.25);
\node at (0.95,0.05) {$=$};
\node at (1.55,0.05) {$0$};
\end{tikzpicture}
.
\end{equation*}
\begin{equation*}\tag{SO3.6}\label{eq:SO3-assoc}
\begin{array}{c@{\quad=\quad}c}
\begin{tikzpicture}[so3rel,x=.62cm,y=.62cm]
\path[use as bounding box] (-.25,-.15) rectangle (2.25,2.55);
\draw[usual] (0,0) -- (.50,.72);
\draw[usual] (1,0) -- (.50,.72);
\draw[usual] (.50,.72) -- (1,1.52);
\draw[usual] (2,0) -- (1,1.52);
\draw[usual] (1,1.52) -- (1,2.35);
\end{tikzpicture}
&
\begin{tikzpicture}[so3rel,x=.62cm,y=.62cm]
\path[use as bounding box] (-.25,-.15) rectangle (2.25,2.55);
\draw[usual] (1,0) -- (1.50,.72);
\draw[usual] (2,0) -- (1.50,.72);
\draw[usual] (1.50,.72) -- (1,1.52);
\draw[usual] (0,0) -- (1,1.52);
\draw[usual] (1,1.52) -- (1,2.35);
\end{tikzpicture}
\end{array}
.
\end{equation*}
\begin{equation*}\tag{SO3.7}\label{eq:SO3-cap-assoc}
\begin{array}{c@{\quad=\quad}c}
\begin{tikzpicture}[so3rel,x=.72cm,y=.72cm]
\path[use as bounding box] (-.25,-.15) rectangle (2.25,1.95);
\draw[usual] (1,0) -- (1.50,.62);
\draw[usual] (2,0) -- (1.50,.62);
\draw[usual] (1.50,.62) -- (1.50,1.15);
\draw[usual] (0,0) -- (0,1.15);
\draw[usual] (0,1.15)
to[out=90,in=180] (.75,1.62)
to[out=0,in=90] (1.50,1.15);
\end{tikzpicture}
&
\begin{tikzpicture}[so3rel,x=.72cm,y=.72cm]
\path[use as bounding box] (-.25,-.15) rectangle (2.25,1.95);
\draw[usual] (0,0) -- (.50,.62);
\draw[usual] (1,0) -- (.50,.62);
\draw[usual] (.50,.62) -- (.50,1.15);
\draw[usual] (2,0) -- (2,1.15);
\draw[usual] (.50,1.15)
to[out=90,in=180] (1.25,1.62)
to[out=0,in=90] (2,1.15);
\end{tikzpicture}
\end{array}
.
\end{equation*}
\end{Proposition}

\begin{proof} 
This follows from \autoref{iso of categories}.
\end{proof}

\begin{Remark}
The relation \autoref{eq:SO3-HI} is the famous $H=I$ relation of SO3, which inspired the nomenclature.
\end{Remark}

By \eqref{eq:SO3-assoc} and \eqref{eq:SO3-cap-assoc}, the following type of shorthands are unambiguous:
\[
\begin{tikzpicture}[baseline={(base)},scale=1]
\coordinate (base) at (0,0);
\path[use as bounding box] (-0.12,-0.12) rectangle (1.12,0.62);
\draw[usual] (0,0) to[out=90,in=180] (0.50,0.50);
\draw[usual] (1,0) to[out=90,in=0] (0.50,0.50);
\draw[usual] (0.50,0) -- (0.50,0.50);
\end{tikzpicture}
\quad,\quad
\begin{tikzpicture}[baseline={(base)},scale=1]
\coordinate (base) at (0,0);
\path[use as bounding box] (-0.65,-0.12) rectangle (0.65,0.92);
\draw[usual] (0,0.38) -- (0,0.85);
\draw[usual] (0,0.38) -- (-0.50,0);
\draw[usual] (0,0.38) -- (0,0);
\draw[usual] (0,0.38) -- (0.50,0);
\end{tikzpicture}.
\]
Just as in \autoref{sl3 jw}, the Jones--Wenzl projector in this case can be constructed as an alternating sum built from all possible symmetric diagrams involving only \webcap, \tup \, and trees with an arbitrary number of branches, where we assign a sign of $-1$ to each pair of cup/cap or merge/split. This follows either from the recursive formula in \autoref{JW theorem}, which always holds in rank 1 (cf. \autoref{pairwise jw lemma}), or can be verified by a direct computation. For example, we have
\begin{equation*}
\begin{aligned}
j_3={}&
\soThreeJW{\draw[usual] (0,0) -- (0,1); \draw[usual] (0.55,0) -- (0.55,1); \draw[usual] (1.10,0) -- (1.10,1);}
\;-\;
\soThreeJW{\draw[usual] (0,1) to[out=270,in=180] (0.275,0.78) to[out=0,in=270] (0.55,1); \draw[usual] (0,0) to[out=90,in=180] (0.275,0.22) to[out=0,in=90] (0.55,0); \draw[usual] (1.10,0) -- (1.10,1);}
\;-\;
\soThreeJW{\draw[usual] (0,0) -- (0,1); \draw[usual] (0.55,1) to[out=270,in=180] (0.825,0.78) to[out=0,in=270] (1.10,1); \draw[usual] (0.55,0) to[out=90,in=180] (0.825,0.22) to[out=0,in=90] (1.10,0);}
\;-\;
\soThreeJW{\draw[usual] (0,1) -- (0.275,0.70); \draw[usual] (0.55,1) -- (0.275,0.70); \draw[usual] (0.275,0.70) -- (0.275,0.30); \draw[usual] (0.275,0.30) -- (0,0); \draw[usual] (0.275,0.30) -- (0.55,0); \draw[usual] (1.10,0) -- (1.10,1);}
\;-\;
\soThreeJW{\draw[usual] (0,0) -- (0,1); \draw[usual] (0.55,1) -- (0.825,0.70); \draw[usual] (1.10,1) -- (0.825,0.70); \draw[usual] (0.825,0.70) -- (0.825,0.30); \draw[usual] (0.825,0.30) -- (0.55,0); \draw[usual] (0.825,0.30) -- (1.10,0);}
\;+\;
\soThreeJW{\draw[usual] (0,1) -- (0.55,0.70); \draw[usual] (0.55,1) -- (0.55,0.70); \draw[usual] (1.10,1) -- (0.55,0.70); \draw[usual] (0.55,0.70) -- (0.55,0.30); \draw[usual] (0,0) -- (0.55,0.30); \draw[usual] (0.55,0) -- (0.55,0.30); \draw[usual] (1.10,0) -- (0.55,0.30);}
\;+\;
\soThreeJW{\draw[usual] (0,1) to[out=270,in=180] (0.55,0.74) to[out=0,in=270] (1.10,1); \draw[usual] (0.55,1) -- (0.55,0.74); \draw[usual] (0,0) to[out=90,in=180] (0.55,0.26) to[out=0,in=90] (1.10,0); \draw[usual] (0.55,0) -- (0.55,0.26);}.
\end{aligned}
\end{equation*}

\subsubsection{Coboundary structure}

In \cite[\S 4.2]{Alqady2025coboundary} a diagrammatic description of the crystal commutor of \cite{Kamnitzcrystal} is given for $\mathfrak{sl}_2$. This generalizes easily to the present case. 

Write $\m$ for $B^{\otimes m}$. Denote by \(\mathcal{D}_k\) the set of all distinguished bottom diagrams whose domain has $k$ strands. Given two diagrams $y\in\mathcal{D}_m,y'\in\mathcal{D}_n$, there is a unique $x\in\mathcal{D}_{m+n}$ such that $x_{\le m}=y$, $x_{>m}=y'$, where $x_{\le m}$ (resp. $x_{>m}$) is the diagram obtained from $x$ by deleting the strands $m+1,\dots, m+n$ (resp. $1,\dots,m$) and changing any resulting half cap or isolated edge into through strands. In such case we write $x=y \odot_{\epsilon_l} y'$ where $\epsilon_l=(l_1,l_2)$ records the number of cut-off branches ($l_1$) followed by split caps ($l_2$), with $l=l_1+l_2:=hk_m(x)$ being such that $th(x)+2l=th(x_{\le m})+th(x_{>m})$. Note that planarity implies branches must be cut off before caps. 

We write $\kappa_{m,n}(x)$ for the diagram $y'\odot_{\epsilon_l} y$; this is a bottom diagram with the same number of through strands, obtained by hooking the $hk_m$ rightmost through strands according to the sequence $(l_1,l_2)$. We write $\tau_{m,n}(x)$ for $\overline{\kappa_{m,n}(x)}\circ j_{th(x)}\circ x \in \Hom(\underline{m+n},\underline{m+n})$.  
As in \cite[Theorem 4.10]{Alqady2025coboundary}, we define the coboundary structure on $\Fund(\mathrm{SO}_3$-Crys$)$ by 
\begin{gather}\label{eq:so3 formula}
\sigma_{\m,\n}=\sum_{x\in\mathcal{D}_{m+n}}\tau_{m,n}(x)=\sum_{x\in\mathcal{D}_{m+n}} \overline{\kappa_{m,n}(x)}\circ j_{th(x)}\circ x,
\end{gather}
where $\sigma_{\m,\n}: \m\otimes \n \to \n \otimes \m$.

For example, if $\epsilon_2=(1,1)$, we have
\[
\begin{tikzpicture}[baseline={(current bounding box.center)},scale=0.6]
\begin{scope}[xshift=0cm,yshift=0cm]
\node[anchor=east] at (-2.0,0) {$x=$};
\draw[usual] (-1.4,0.9) -- (-1.4,-0.9);
\draw[usual] (0,0.9) -- (0,0.15);
\draw[usual] (0,0.15) -- (-0.9,-0.65);
\draw[usual] (0,0.15) -- (0.9,-0.65);
\draw[usual] (-0.35,-0.65) .. controls (-0.15,-0.35) and (0.15,-0.35) .. (0.35,-0.65);
\draw[densely dashed,gray] (0,1.15) -- (0,-1.05);
\end{scope}

\node at (1.1,0) {$=$};

\begin{scope}[xshift=2.55cm,yshift=0cm]
\foreach \x in {0,0.22,0.44} {
\draw[usual] (\x,0.9) -- (\x,-0.9);
}
\node at (1.3,0) {$\odot_{\epsilon_2}\quad$};
\foreach \x in {1.55,1.77} {
\draw[usual] (\x,0.9) -- (\x,-0.9);
}
\end{scope}

\begin{scope}[xshift=10.0cm,yshift=0cm]
\node[anchor=east] at (-0.2,0) {$\kappa_{3,2}(x)=$};
\begin{scope}[xshift=0.2cm,yshift=0cm]
\foreach \x in {0,0.22} {
\draw[usual] (\x,0.9) -- (\x,-0.9);
}
\node at (1.05,0) {$\odot_{\epsilon_2}\quad $};
\foreach \x in {1.25,1.47,1.69} {
\draw[usual] (\x,0.9) -- (\x,-0.9);
}
\node at (2.2,0) {$=$};
\end{scope}
\begin{scope}[xshift=4.6cm]
\draw[usual] (0,0.9) -- (0,0.15);
\draw[usual] (0,0.15) -- (-0.9,-0.65);
\draw[usual] (0,0.15) -- (0.9,-0.65);
\draw[usual] (-0.35,-0.65) .. controls (-0.15,-0.35) and (0.15,-0.35) .. (0.35,-0.65);
\draw[usual] (1.35,0.9) -- (1.35,-0.9);
\end{scope}
\end{scope}
\end{tikzpicture}\]

\begin{Proposition}
This formula \eqref{eq:so3 formula} defines the commutor of \cite{Kamnitzcrystal} on $\Fund(\mathrm{SO}_3$-\textrm{Crys}$)$.
\end{Proposition}

\begin{proof}
The proof of \cite[Theorem 4.10]{Alqady2025coboundary} can be straightforwardly adapted to the present setting, as it is mostly formal and relies only on \autoref{wedderburn basis} and the bijection between $\mathcal{D}_n$ and components of $B^{\otimes n}$, which in this case is illustrated below.
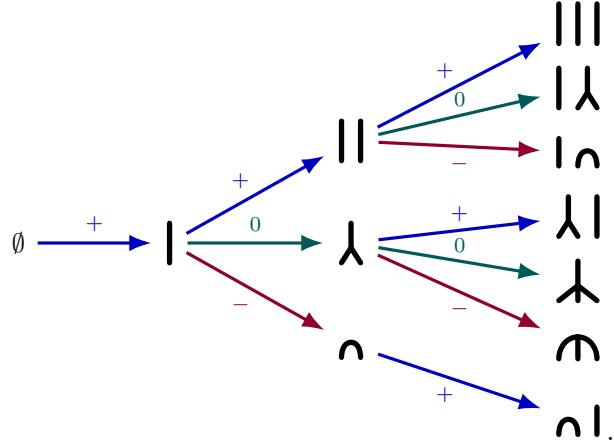
\[
\begin{tikzpicture}[
>=Latex,
font=\small,
v/.style={inner sep=1pt},
plusedge/.style={->, very thick, blue!75!black, shorten >=3pt, shorten <=3pt},
zeroedge/.style={->, very thick, teal!70!black, shorten >=3pt, shorten <=3pt},
minusedge/.style={->, very thick, purple!75!black, shorten >=3pt, shorten <=3pt},
edgelab/.style={font=\scriptsize, inner sep=1pt}
]
\node[v] (a) at (0,0)        {\soThreeBranchEmpty};
\node[v] (b) at (2.0,0)      {\soThreeBranchI};

\node[v] (c) at (4.4,1.35)   {\soThreeBranchII};
\node[v] (d) at (4.4,0)      {\soThreeBranchMerge};
\node[v] (e) at (4.4,-1.35)  {\soThreeBranchCap};

\node[v] (f) at (7.4,2.90)   {\soThreeBranchIII};
\node[v] (g) at (7.4,2.05)   {\soThreeBranchIMerge};
\node[v] (h) at (7.4,1.20)   {\soThreeBranchICap};
\node[v] (i) at (7.4,0.35)   {\soThreeBranchMergeI};
\node[v] (j) at (7.4,-0.50)  {\soThreeBranchTree};
\node[v] (k) at (7.4,-1.35)  {\soThreeBranchTup};
\node[v] (l) at (7.4,-2.35)  {\soThreeBranchCapI};

\draw[plusedge] (a) -- node[edgelab, above=3pt] {$+$} (b);

\draw[plusedge] (b) -- node[edgelab, above left=1pt] {$+$} (c);
\draw[zeroedge] (b) -- node[edgelab, above=3pt] {$0$} (d);
\draw[minusedge] (b) -- node[edgelab, below left=1pt] {$-$} (e);

\draw[plusedge] (c) -- node[edgelab, above left=1pt] {$+$} (f);
\draw[zeroedge] (c) -- node[edgelab, above=2pt] {$0$} (g);
\draw[minusedge] (c) -- node[edgelab, below=2pt] {$-$} (h);

\draw[plusedge] (d) -- node[edgelab, above=2pt] {$+$} (i);
\draw[zeroedge] (d) -- node[edgelab, above=2pt] {$0$} (j);
\draw[minusedge] (d) -- node[edgelab, below=2pt] {$-$} (k);

\draw[plusedge] (e) -- node[edgelab, below left=1pt] {$+$} (l);
\end{tikzpicture}.
\]
Here the highest weight of a summand is the net number of $+$'s of the branching path. This graph, with at most three possible branches at each point, suggests the necessary modification to \cite[Lemma 4.15]{Alqady2025coboundary}. 
\end{proof}

\subsection{Motzkin}\label{S:Motzkin}

We now give an example where $\CC=\gcrys$ but we use a nonstandard choice of generators, i.e. not the crystals of fundamental representations. Let $\CC =\mathfrak{sl}_2$-Crys, but instead of $B_1=B(\omega_1)$ we take the generator $B_1$ of $\Fund(\CC, B)$, $B=\{B_1\}$ to be the crystal corresponding to $B(\omega_1)\oplus B(0)$. Call $\Fund(\CC, B)$ the Motzkin crystal category, denoted by $\CrysMo$.

\begin{Proposition}
The category $\CrysMo$ is the $\kk$-linear monoidal category generated by a single object and the generating morphisms
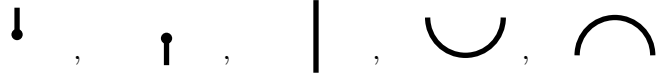
\[
\motzTopDot
\motzComma
\motzBottomDot
\motzComma
\motzIdentity
\motzComma
\motzCup
\motzComma
\motzCap
\]
where the dots are morphisms from/to the monoidal unit $\mathbbm{1},$ subject to relations
\begin{equation*}\tag{M.1}\label{eq:motzkin-zigzag}
\begin{array}{c@{\quad=\quad}c@{\quad=\quad}c}
\motzSnakeLeft & 0 & \motzSnakeRight
\end{array}
.
\end{equation*}
\begin{equation*}\tag{M.2}\label{eq:motzkin-cap-dot}
\begin{array}{c@{\quad=\quad}c@{\quad=\quad}c}
\motzCapDotLeft & 0 & \motzCapDotRight
\end{array}
.
\end{equation*}
\begin{equation*}\tag{M.3}\label{eq:motzkin-circle}
\begin{array}{c@{\quad=\quad}c@{\quad=\quad}c}
\motzCircle & \motzDotDot & 1
\end{array}
.
\end{equation*}
\end{Proposition}

\begin{proof}
This follows from \autoref{iso of categories}.
\end{proof}

\begin{Remark}
As in \cite[\S 3]{Alqady2025coboundary}, $\CrysMo$ may be obtained by first renormalizing the usual Motzkin category (see e.g. \cite[\S 7.2]{hu2019presentationsdiagramcategories}) by scaling the cap and bottom dot by $q$, and then setting $q=0$. Note that \cite{hu2019presentationsdiagramcategories} draws the dots using a different convention.
\end{Remark}

There is a diagrammatic description of the coboundary structure on $\CrysMo$ exactly analogous to the one given in \cite[\S 4.2]{Alqady2025coboundary} for $\mathfrak{sl}_2$ and in \eqref{eq:so3 formula} for $\mathrm{SO}_3$-Crys, which can be proven in the same manner. 

\subsubsection{The associated monoid}

Again write $\n$ for $B_1^{\otimes n}$, and for fixed $n$, let $r_k$ be the number of distinguished bottom diagrams with domain $\n$ and codomain $\kkk$. From \autoref{wedderburn basis}, we obtain the semisimple decomposition \[\End(\n)\cong \prod_{k\le n} \Mat_{r_k}(\kk).\] 

This algebra has $n+1$ simple modules $S_r$ up to isomorphism, labeled by the number of through strands (= weight) $r$. Unravelling the isomorphism, $S_r$ has a basis indexed by distinguished top diagrams $\underline{r}\to \n$, with $\End(\n)$-action given by \[x\cdot y = \begin{cases}
xy & xy \text{\ is a distinguished top diagram with $r$ strands} \\
0 & \text{else} .
\end{cases}\]

We also have the associated monoid $M_n$, whose elements are the usual Motzkin $n$-diagrams (see \cite{benkart2014motzkin}) together with a zero element $z$, with multiplication given by the relations of $\CrysMo$. We have $\End(\n)\cong \kk M_n/\kk z$, i.e. it is a contracted monoid algebra, similar to \cite[\S 3.4]{Alqady2025coboundary}.

\begin{Proposition}
The monoid $M_n$ is an inverse monoid, that is, for $m\in M_n$, there is a unique $m^*$ such that $mm^*m=m,m^*mm^*=m^*$. 
\end{Proposition}

\begin{proof}
Similar to \cite[Lemma 3.41]{Alqady2025coboundary}. We have $z^*=z$, and if $x$ is a decomposition into top and bottom diagrams, then $x^*=\overline{x}$.   
\end{proof}

\begin{Proposition}\label{greens motzkin}
Declare $th(z)=-\infty$, then the Green's cells are given by: $x\le_{\mathcal{J}} y$ iff $th(x)\le th(y)$, $x\sim_\mathcal{L}y$ iff the bottom diagrams agree (otherwise incomparable), and $x\sim_\mathcal{R}y$ iff the top diagrams agree (otherwise incomparable).
\end{Proposition}

\begin{proof}
Similar to \cite[Lemma 3.40]{Alqady2025coboundary}.
\end{proof}

It is easy to see that all maximal subgroups of $M_n$ are trivial. An inverse monoid all of whose maximal subgroups are trivial is semisimple over any field, by \cite[Corollary 9.4]{steinberg2016representation}, so $\kk M_n$ is always semisimple. The representation theory of a semisimple monoid algebra is determined by its character table, see \cite[\S 7]{steinberg2016representation}. 

\begin{Proposition}
The character table for $M_n$ is obtained from the character table $S$ for standard modules of $Mo_n$ (see \cite[Proposition 2B.15]{he2025tensorpowersrepresentationsdiagram}) by adding a column of zeros corresponding to the class of $z$, and a row of ones (trivial representation) coming from $z$. It looks as follows:
\[
\scalebox{0.9}{$
\begin{array}{c|ccccc}
& z & 0 & 1 & 2 & \cdots \\
\hline
z & 1 & 1 & 1 & 1 & \cdots \\
0 & 0 & 1 & M(1,0) & M(2,0) & \cdots \\
1 & 0 & 0 & 1 & M(2,1)& \cdots \\
\vdots & \vdots & \vdots & \vdots& \vdots & \ddots
\end{array}
$}
\]
where we index rows and columns by $r$ representing through strands, and $M(i,j)$ is the number of Motzkin top diagrams with domain $\underline{i}$ and $j$ through strands, $$M(i,j)=\begin{cases}\sum_{t=0}^{i} \frac{j+1}{j+t+1} \binom{i}{j+2t} \binom{j+2t}{t}
& i \ge j,\\ 0 & \text{else.}\end{cases}.$$ Let $C$ denote the character table viewed as a matrix, then $C^{-1}=
\begin{psmallmatrix}
1 & -\,\mathbf{1}^{\mathsf T}S^{-1}\\
0 & S^{-1}
\end{psmallmatrix},$ where \(\mathbf{1}=(1,1,1,\dots)^{\mathsf T}\)
and\[
S^{-1}_{ij}=
\begin{cases}
[x^{j-i}](1+x+x^2)^{-(i+1)}
=
(-1)^{j-i}\displaystyle\sum_{r=0}^{\left\lfloor\frac{j-i}{2}\right\rfloor}
(-1)^r
\binom{i+r}{r}
\binom{j-r}{j-i-2r},
& i\le j,\\[1.25em]
0, & i>j.
\end{cases}
\]
\end{Proposition}

\begin{proof}
Let $\chi_j$ be the character for $S_j$ in notation above. Note that $S_j$ is a permutation representation. For the conjugacy class indexed by $i$, take as representative $x$ the diagram with $i$ through lines and $n-i$ dots. If $y$ is a top diagram with $j$ through lines, then $x\cdot y=y$ iff $y$ has $n-i$ dots where $x$ has them, leaving $i$ positions free. It follows that the number of such $y$'s is $M(i,j)$. This implies that the form of the character table is as claimed. The inverse of the matrix is obtained from \cite[Proposition 2B.15]{he2025tensorpowersrepresentationsdiagram}.     
\end{proof}

Since any $\kk M_n$-module $M$ is determined by its character, the above formula enables the determination of the number of summands in $M$ corresponding to any simple module. In particular, the structure of the representation ring is easily computable, cf. \cite[Proposition 2B.18]{he2025tensorpowersrepresentationsdiagram}. We could also work directly with the semisimple algebra $\End(\n)$: the calculation above shows that it has a well-defined character theory, encoded in the matrix $S$, which is the same as the character theory of the cell modules of the usual Motzkin algebra.

\begin{Remark}
An analogous statement holds for the Temperley--Lieb monoid with zero, denoted by $T_n$ in \cite[\S 3.4]{Alqady2025coboundary}. One just has to replace the matrix $S$ with the matrix in \cite[Proposition 2B.15]{he2025tensorpowersrepresentationsdiagram} (which was also computed in \cite[\S 2]{halverson1995characters}), and the inverse $S^{-1}$ by the matrix in \cite[Proposition 2B.17]{he2025tensorpowersrepresentationsdiagram}.  
\end{Remark}

\begin{Remark}
The growth problem (in the sense of, e.g., \cite{CoOsTu-growth}) for the 
crystal Motzkin monoid is thus easy, and can be directly obtained from the above and 
\cite{HeTu-growth-finite-monoids}.
\end{Remark}


\end{document}